\documentclass{amsart}
\usepackage{color}
\usepackage[frame,cmtip,arrow,matrix,line,graph]{xy}
\usepackage{comment}
\usepackage{hyperref}
\usepackage{amsmath}
\usepackage{amsfonts}
\usepackage{amssymb}
\usepackage{graphicx}
\usepackage[frame,cmtip,arrow,matrix,line,graph]{xy}
\usepackage{color}
\usepackage{lpic,enumerate}
\usepackage{stmaryrd}
\usepackage{scrextend}

\definecolor{darkgreen}{RGB}{47,139,79}
\definecolor{darkblue}{RGB}{36,24,130}
\hypersetup{
    colorlinks=true,     linktoc=all,         linktocpage,
    citecolor=darkgreen,
    filecolor=black,
    linkcolor=darkblue,
    urlcolor=black
}
\let\oldtocsection=\tocsection
\let\oldtocsubsection=\tocsubsection
\let\oldtocsubsubsection=\tocsubsubsection
\renewcommand{\tocsection}[2]{\hspace{0em}\oldtocsection{#1}{#2}}
\renewcommand{\tocsubsection}[2]{\hspace{1em}\oldtocsubsection{#1}{#2}}
\renewcommand{\tocsubsubsection}[2]{\hspace{2em}\oldtocsubsubsection{#1}{#2}}
\DeclareRobustCommand{\SkipTocEntry}[5]{}
\newtheorem{thm}{Theorem}[section]
\newtheorem{lem}[thm]{Lemma}
\newtheorem{prop}[thm]{Proposition}
\newtheorem{cor}[thm]{Corollary}

\newtheorem{Th}{Theorem}

\theoremstyle{definition}
\newtheorem{Def}[thm]{Definition}

\newtheorem{ex}[thm]{Example}
\theoremstyle{remark}
\newtheorem{rem}[thm]{Remark}

\numberwithin{equation}{section}
\input xy
\xyoption{all}
\newcommand{\al}{\alpha}
\newcommand{\B}{\mathcal{B}}
\newcommand{\Bb}{\overline{\B}}

\newcommand{\del}{\partial}

\newcommand{\De}{\Delta}
\newcommand{\F}{\mathcal{F}}
\newcommand{\FF}{\mathbb{F}}
\newcommand{\ga}{\gamma}
\newcommand{\la}{\lambda}

\newcommand{\La}{\Lambda}

\newcommand{\s}{\sigma}
\newcommand{\eps}{\varepsilon}

\newcommand{\RR}{\mathbb{R}}
\newcommand{\Si}{\Sigma}
\newcommand{\Z}{\mathbb{Z}}

\newcommand{\HH}{\nu_{C\! S}}

\newcommand{\maps}{\operatorname{Maps}}
\newcommand{\id}{\operatorname{id}}
\newcommand{\U}[1]{\underline{#1}}

\newcommand{\inc}{\hookrightarrow}

\newcommand{\rar}{\longrightarrow}
\newcommand{\sta}{\stackrel}

\newcommand{\minus}{\backslash}
\newcommand{\x}{\times}
\newcommand{\ot}{\otimes}
\newcommand{\op}{\oplus}
\newcommand{\st}{\star}

\newcommand{\poc}{\widehat{\vee}}
\newcommand{\lgl}{\langle}
\newcommand{\rgl}{\rangle}
\newcommand{\func}{\operatorname}

\newcommand{\GH}{\mathrm{GH}}
\newcommand{\CS}{\mathrm{CS}}
\DeclareMathOperator\concat{concat}
\DeclareMathOperator\cut{cut}
\DeclareMathOperator\sgn{sgn}
\DeclareMathOperator\Cr{Cr}
\DeclareMathOperator\Thom{Th} 
\DeclareMathOperator\AW{AW}
\DeclareMathOperator\Poin{P}

\begin{document}
\title{Product and coproduct in string topology} 
\author{Nancy Hingston}
\author{Nathalie Wahl}
\date{\today }

\begin{abstract}
Let $M$ be a closed Riemannian manifold.  We extend the product of Goresky-Hingston, on
the cohomology of the free loop space of $M$ relative to the constant loops,
to a nonrelative product.  It is graded associative and commutative, and
compatible with the length filtration on the loop space, like the original
product.  We prove the following new geometric property of the dual
homology coproduct:  the nonvanishing of the $k$-th iterate of the
coproduct on a homology class ensures the existence of a loop with a
$(k+1)$-fold self-intersection in every representative of the class. For spheres
and projective spaces, we show that this is sharp, in the sense that
the $k$-th iterated coproduct vanishes precisely on those classes that have support in
the loops with at most $k$-fold self-intersections. 
We study the interactions between this cohomology product and the better-known Chas-Sullivan product. 
We give explicit integral chain level constructions of these loop products
and coproduct, including a new construction of the Chas-Sullivan product,
which avoid the technicalities of infinite dimensional tubular neighborhoods
and delicate intersections of chains in loop spaces.
%
%
  \end{abstract}

\maketitle

  Let $M$ be a closed oriented manifold of dimension $n$, for which we pick a
Riemannian metric, and let $\Lambda M=\maps(S^1,M)$ denote its free
loop space. Goresky and the first author defined in \cite{GorHin} a product $%
\oast$ on the cohomology $H^*(\Lambda M,M)$, 
relative to the constant loops $M\subset \Lambda M$. They showed that this
cohomology product behaves as a form of ``dual'' of the Chas-Sullivan
product \cite{CS99} on the homology $H_*(\Lambda M)$, at least through the
eyes of the energy functional. In the present paper, we lift this relative
cohomology product and associated homology coproduct to chain level
non-relative product $\widehat{\oast}\colon C^*(\Lambda M)\otimes
C^*(\Lambda M)\to C^*(\Lambda M)$ and coproduct $\widehat{\vee} \colon
C_*(\Lambda)\to C_*(\Lambda)\otimes C_*(\Lambda)$, on integral singular chains.
The operations $\widehat{\oast}$ and $\widehat{\vee}$ are the ``extension by
zero'' of their relative predecessors $\oast$ and $\vee$ under the splitting 
on homology and cohomology induced by the evaluation map $\Lambda M\to M$.
We study the algebraic and geometric properties of these two new operations, 
$\widehat{\vee}$ and $\widehat{\oast}$, showing in particular that the
properties of the original product and coproduct proved in \cite{GorHin},
such as associativity, graded commutativity, and compatibility with the
length filtration, remain valid for the extended versions. Our main result
is the following: for $[A]\in H_*(\Lambda M)$, the non-vanishing of the
iterated coproduct $\widehat{\vee}^k[A]$ implies the existence of loops with 
$(k+1)$-fold intersections in the image of any chain representative of $[A]$%
; on spheres and projective spaces, we show that this is a complete
invariant in the sense that the converse implication also holds.

An advantage of having a coproduct defined on $H_*(\Lambda M)$ is that it
makes it possible to study its interplay with the Chas-Sullivan product; we
show that the equation $\wedge\circ \vee=0$ holds. We also
exhibit, through computations, the failure of an expected Frobenius equation.

\medskip

To state our main results more precisely, we need the following ingredients:
Fix a small $\varepsilon>0$ smaller than the injectivity radius and let 
\begin{equation*}
U_M=\{(x,y)\in M^2\ |\ |x-y|<\varepsilon\}
\end{equation*}
be the $\varepsilon$-neighborhood of the diagonal $\Delta M$ in $M^2$, with 
\begin{equation*}
\tau_M\in C^n(M^2,M^2\backslash U_M)
\end{equation*}
a Thom class for an associated tubular embedding $\nu_M\colon TM\overset{%
\simeq}{\longrightarrow} U_M$. (A specific model is given in Section~\ref%
{tubsec}.) Consider the evaluation maps 
\begin{equation*}
e\times e\,\colon \Lambda M^2 \longrightarrow M^2 \ \ \ \text{and}\ \ \
e_I\colon \Lambda M\times I\longrightarrow M^2
\end{equation*}
taking a pair $(\gamma,\delta)$ to $(\gamma(0),\delta(0))$ and $(\gamma,s)$
to $(\gamma(0),\gamma(s))$. The tubular embedding $\nu_M$ gives rise to a
retraction $U_M\to \Delta M$, which we lift to retraction maps

\medskip

\begin{tabular}{lrcrcl}
$R_{\CS}\colon$ & $(e\times e)^{-1}(U_M)$ & $\longrightarrow$ & $(e\times
e)^{-1}(\Delta M)$ & $=$ & $\{(\gamma,\delta)\in \Lambda M^2\ | \
\gamma(0)=\delta(0)\}$ \\ 
$R_{\GH}\colon$ & $e_I^{-1}(U_M)$ & $\longrightarrow$ & $e_I^{-1}(\Delta M)$
& $=$ & $\{(\gamma,s)\in \Lambda M\times I\ | \ \gamma(s)=\gamma(0)\}$%
\end{tabular}

\medskip

\noindent as depicted in Figure~\ref{fig:sticks}: each map adds two geodesic
sticks to the loops to make them intersect as required---see Sections~\ref%
{sec:newCS} and~\ref{sec:newGH} for precise definitions. Finally, let 
\begin{equation*}
\concat\colon \Lambda M\times_M\Lambda M = (e\times e)^{-1}(\Delta M)
\longrightarrow M \ \ \ \text{and}\ \ \ \cut\colon e_I^{-1}(\Delta M)
\longrightarrow \Lambda M\times \Lambda M
\end{equation*}
be the concatenation and cutting maps, the latter cutting the loop at its
self-intersection time $s$.

\begin{figure}[h]
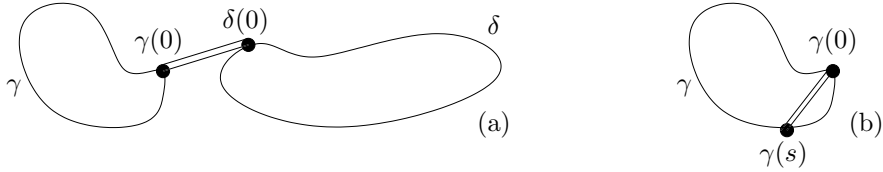

\centering
\begin{lpic}{sticks(0.55,0.55)}
\lbl[b]{-1,10;$\ga$}
\lbl[b]{34,20;$\ga(0)$}
\lbl[b]{55,25;$\delta(0)$}
\lbl[b]{115,25;$\delta$}
\lbl[b]{161,10;$\ga$}
\lbl[b]{197,20;$\ga(0)$}
\lbl[t]{185,-1;$\ga(s)$}
\lbl[b]{115,0;(a)}
\lbl[b]{205,0;(b)}
\end{lpic}
\caption{The retraction maps $R_{\CS}$ and $R_{\GH}$.}
\label{fig:sticks}
\end{figure}

Our results are rooted in the following chain-level construction of the
Chas-Sullivan product and the coproduct, based on the ideas of Cohen and
Jones in \cite{CohJon} but avoiding the technicalities of infinite
dimensional tubular neighborhoods, and avoiding the subtle limit arguments
of \cite{GorHin}.

\begin{Th}
\label{Th:CSGH} The Chas-Sullivan $\wedge$ product of \cite{CS99} and the
coproduct $\vee$ of Goresky-Hingston \cite{GorHin} (see also Sullivan \cite%
{Sul04}) admit the following integral chain level descriptions: for $A\in
C_p(\Lambda M)$ and $B\in C_q(\Lambda M)$, 
\begin{equation*}
A\wedge B=(-1)^{n-np} \concat \left(R_{\CS}((e\times e)^*(\tau_M)\cap (A\times
B))\right) .
\end{equation*}
and for $C\in C_k(\Lambda M,M)$,  
\begin{equation*}
\vee C= \sgn (\cut (R_{\GH} (e_I^*(\tau_M)\cap (C\times I))))
\end{equation*}
with $A\wedge B\in C_{p+q-n}(\Lambda)$ and $\vee C\in \bigoplus_{p+q=k+1-n}C_p(\Lambda M,M)\otimes C_q(\Lambda M,M)$,
with sign change $\sgn$ given by $(-1)^{n-np}$ on the terms of bidegree $(p,q)$.
\end{Th}

\noindent We note that the chain complex $C_{*}(\Lambda M,M)\otimes C_*(\Lambda M,M)$ computes the
relative homology group $H_*(\Lambda\times\Lambda,M\times\Lambda\cup
\Lambda\times M)$, see Remark~\ref{rem:+U} for more details about this.

\smallskip

In the case of the Chas-Sullivan product, there exist many approaches to
chain level constructions, in particular using more ``geometric'' chains,
applying more directly the original idea of Chas and Sullivan of
intersecting chains (see e.g.~\cite{Cha05,Lau11,Iri17}). The idea of using
small geodesics to make up for ``almost intersections'' was already
suggested in \cite{Tam} and can also be found in \cite{DCP}.

We emphasize the similarity between the definitions of $\wedge$ and $\vee$
given above. Note that cutting and concatenating are essentially inverse
maps. The product and coproduct are corrected by a sign for better algebraic
properties; see Theorems~\ref{thm:CSalg} and \ref{thm:coGHalg} and Appendix~\ref{app:signs} for more details about this.
In Propositions~\ref{prop:transCS} and \ref{prop:strong4}, we show that for cycles parametrized by manifolds that, after applying the evaluation map,  intersect the diagonal in $M\x M$ transversally, the product and coproduct can be computed by a geometric intersection followed by concatenation or cut maps.

\medskip

The splitting of $H_*(\Lambda M)\cong H_*(\Lambda M,M)\oplus H_*(M)$ by the
evaluation map makes it possible to define an ``extension by 0'' of the
coproduct, setting it to be trivial on the constant loops. The same holds
for chains:

\begin{Th}
\label{Th:poc} There is a unique lift 
\begin{equation*}
\widehat{\vee}\colon C_*(\Lambda M)\longrightarrow \bigoplus_{p+q=*+1-n}C_p(\Lambda M)\otimes C_q(\Lambda M)
\end{equation*}
of the coproduct $\vee$ of Theorem~\ref{Th:CSGH} satisfying that 
\begin{equation*}
\langle x\times y,\widehat \vee Z\rangle=0 \ \ \ \ \text{ if } \ x\in
e^*C^{\ast }(M), \; y\in e^*C^*(M), \text{\textit{\ or }} Z\in C_{\ast}(M)
\end{equation*}
for $e\colon \Lambda M\to M$ the evaluation at $0$, and where we identify $M$ with the constant loops in $\La M$. For $A\in C_*(\Lambda M)$%
, with $p_*A\in C_*(\Lambda M,M)$ its projection, it is defined by 
\begin{equation*}
\widehat{\vee} A= (1-e_*)\times (1-e_*) \vee \! (p_*A).
\end{equation*}
The induced map in homology $\widehat{\vee}\colon H_*(\Lambda)\to
H_{*-n+1}(\Lambda\times\Lambda)$ has the following properties:

\begin{enumerate}[(i)]

\item (Vanishing)  
\begin{equation*}
\wedge\circ \widehat{\vee}=0.
\end{equation*}

\item (Support) If $Z\in C_{\ast }(\Lambda )$ has the property that every
nonconstant loop in its image has at most $k$-fold intersections, then 
\begin{equation*}
\widehat{\vee }^{k}[Z]=\widehat{\vee }^{k}(\Delta \lbrack Z])=0,
\end{equation*}
for $\Delta\colon H_*(\Lambda M)\to H_{*+1}(\Lambda M)$ the map induced by
the circle action on $\Lambda M$.
\end{enumerate}
\end{Th}

The vanishing part is a form of involutive identity. Its proof uses the intermediate result that the Chas-Sullivan product is almost completely trivial on the subspace $\La\x_M\La\subset \La\x \La$ of pairs of loops that already have the same basepoint, see Proposition~\ref{prop:CSF2}. 
The last part of the theorem confirms the essence of the geometry behind the
homology coproduct: it looks for self-intersections and cuts them apart. The corresponding result for the unlifted coproduct $\vee$ is stated in Theorem~\ref{thm:coproduct2}. The
proof of this result was made possible by Theorem~\ref{Th:CSGH}.

We define in Section~\ref{sec:int} the \emph{intersection multiplicity} $%
\func{int}([Z])$ of a homology class $[Z]\in H_*(\Lambda)$; it has the
property that $\func{int}([Z])\le k$ whenever 
there exists a representative of $[Z]$ with image in the loops with at most $%
k$-fold intersections (in the sense of Definition~\ref{def:kfold}). The
above result shows that if $\func{int}([Z])\le k$, then $\widehat{\vee}^k
[Z] = 0$, or equivalently that when $\widehat{\vee}^k[Z]\neq 0$, then \emph{%
any} chain representative of $[Z]$ must have at least one loop with a $k+1$%
-fold intersection in its image. For spheres and projective spaces, we get
the following stronger statement:

\begin{Th}
  \label{Thprop:int}If $M=S^{n}$, $\mathbb{R}P^{n}$, $\mathbb{C}P^{n}$, $\mathbb{H}P^{n}$ or $\mathbb{O}P^{2}$, 
  then for any $[Z]\in H_{\ast }(\Lambda )$,  and $k\geq 1$, 
\begin{equation*}
\func{int}([Z])\leq k\ \ \Longleftrightarrow\ \ \widehat{\vee}^{k}[Z]=0.
\end{equation*}
\end{Th}

In Section~\ref{sec:compspheres1} and \ref{sec:computations}, we give a complete computation of the
coproduct $\widehat{\vee}$ on the homology $H_*(\Lambda S^n)$ with $S^n$ an
odd sphere. We use this computation to show in Remark~\ref{rem:Frob} that the ``Frobenius formula''
fails: 
\begin{equation*}
\widehat{\vee}\circ \wedge \neq (1\times \wedge)\circ \widehat{%
\vee} + (\wedge\times 1)\circ \widehat{\vee}
\end{equation*}
and give geometric reasons for this. We also show in Remark~\ref{rem:wedgepoc} that, in contrast to $\wedge\circ \widehat{\vee}$, the composition $\wedge\circ \widehat{\vee}_{\Thom}
$, for $\widehat{\vee}_{\Thom}$ the non-sign corrected version of the coproduct $\widehat{\vee}$, is non-zero in general.

\medskip

Finally, we relate the dual cohomology product to the product of \cite%
{GorHin} and show that it still has the same properties as the unlifted
product. In the statement, $\widehat{\vee}^*$ denotes the dual of the map $%
\widehat{\vee}$ of Theorem~\ref{Th:poc}.

\begin{Th}
\label{Th:cop} The dual product 
\begin{equation*}
  \xymatrix{\widehat{\circledast }\colon C^{p}(\Lambda M)\otimes C^{q}(\Lambda M) 
    \ar[r]^-{\poc^*} & C^{p+q+n-1}(\Lambda M)}
\end{equation*}%
is a chain-level definition of the canonical extension by 0 of the relative
cohomology Goresky-Hingston product \cite{GorHin}. The induced degree $-1$ product on the shifted cohomology $H^{*-n}(\La)$ has the following properties:

\begin{enumerate}[(i)]

\item It satisfies the graded associative and commutative relations 
\begin{equation*}
[x]\widehat{\oast} [y]=(-1)^{pq+1}[y]\widehat{\oast} [x] \ \ \textrm{and}\ \ ([x]\widehat{\oast} [y])\widehat{\oast} [z]=(-1)^{r+1}[x]\widehat{\oast} ([y]\widehat{\oast} [z]) 
\end{equation*}
for all $[x]\in H^{p-n}(\Lambda)$, $[y]\in H^{q-n}(\Lambda)$ and $[z]\in H^{r-n}(\Lambda)$.

\item (Morse theoretic inequality) Let $[x],[y]\in H^{\ast}(\Lambda )$ with $%
[x]\widehat{\oast} [y]\neq 0$, then 
\begin{equation*}
\Cr([x]\widehat{\circledast }[y])\geq \Cr([x])+\Cr([y])
\end{equation*}
where $\Cr([a])=\sup\{l\in \mathbb{R} \, |\, [a] \text{ is supported on }
\Lambda^{\ge L}\}$ for $\Lambda^{\ge L}$ the subspace of loops of energy at
least $L^2$.
\end{enumerate}
\end{Th}

Note that by definition, we have the following relation between the
operations $\widehat{\vee}$ and $\widehat{\oast}$: for any $x\in
C^{p}(\Lambda )$, $y\in C^q(\Lambda)$ and $Z\in C_{p+q+n-1}(\Lambda )$, 
\begin{equation*}
\langle x\widehat{\circledast}y,Z\rangle=\langle x\times y,%
\widehat{\vee }Z\rangle.
\end{equation*}

\medskip

\noindent \emph{Organization of the paper.} In Section~\ref{sec:setup} we
give a precise construction of chain representatives $\wedge, \vee$ and $%
\oast$ for the Chas-Sullivan and Goresky-Hingston products and coproduct. In
Section~\ref{sec:equivalence}, we prove that our definitions do indeed model
the earlier defined products and coproducts. This involves the explicit
construction of a tubular neighborhood of $\Lambda M\times_M\Lambda M$
inside $\Lambda^2$, see Proposition~\ref{TubCS}, and of $%
F_{(0,1)}:=e_I^{-1}(\Delta M)$ inside $\Lambda M\times (0,1)$, see
Proposition~\ref{coemb}. Theorem~\ref{Th:CSGH} is proved by combining
Propositions~\ref{prop:CSequiv} and~\ref{prop:coequiv}.  The section ends with  a detailed proof of the algebraic properties of the coproduct (Theorem~\ref{thm:coGHalg}). 
Section~\ref{sec:computing} starts by the geometric computation of the product and coproduct of appropriately transverse cycles (Propositions~\ref{prop:transCS} and Proposition~\ref{prop:strong4}). We then prove the support result for the coproduct (Theorem~\ref{thm:coproduct2}), the main ingredient of the second part of Theorem~\ref{Th:poc}. The section contains a number of example computations, including for the coproduct in the case of odd spheres (Proposition~\ref{prop:spheres1}).   
In Section~\ref{sec:extend}, we define the lifted product $\widehat{\vee}$ and coproduct $%
\widehat{\oast}$, and prove Theorem~\ref{Th:cop} (as Theorem~\ref%
{thm:liftedproduct1}) and finish the proof of Theorem~\ref{Th:poc} (Theorems~%
\ref{thm:extendedcoproduct} and~\ref{thm:trivial})---the vanishing statement
in Theorem~\ref{Th:poc} is shown in Section~\ref{sec:vanishing} to be
closely related to the almost complete vanishing of the so-called ``trivial
coproduct'', as pointed out by Tamanoi in \cite{Tam}.
The proof of Theorem~\ref{Thprop:int} is given in Section~\ref{sec:int} as Theorem~\ref{thm:int}. 
Finally, the appendix contains a section about the properties of the cap
product used in the paper, and a section about signs in the
intersection product.

In the present paper, all homology and cohomology is taken with integral
coefficients, and the chains and cochains are the singular chains and
cochains, unless explicitly otherwise specified.

\addtocontents{toc}{\SkipTocEntry}

\subsection*{Acknowledgements}

The present paper owes its existence to Ralph Cohen, who originally was to
be a coauthor; both authors would like to thank him for suggesting that we
work together, and thank him for his input in the first developments of this
project. Many people along the way helped us in our struggle with the fine
behavior of the loop space and its homology. We would like in particular to
thank S{\o }ren Galatius, John Klein, Florian Naef, Alexandru Oancea, Manuel Rivera, Karen Uhlenbeck and Zhengfang Wang
for helpful conversations, and the referees for insisting that we go a step further at several places in the paper.
In the course of this work, the first author was
supported by the Institute for Advanced Study and the second author by the
Danish National Sciences Research Council (DNSRC) and the European Research
Council (ERC)  under the European Union’s Horizon 2020 research and innovation programme (grant agreement No. 772960), as well as by the Danish National Research Foundation through
the Centre for Symmetry and Deformation (DNRF92) and the Copenhagen Centre for Geometry and Topology (DNRF151). The first author would
also like to thank the University of Copenhagen for its hospitality and
support.

\tableofcontents

\section{Chain level definitions of the loop products}

\label{sec:setup}

Assume that $M$ is a closed, compact, oriented Riemannian manifold of
dimension $n$. Fix $0<\varepsilon <\frac{\rho }{14}$ where $\rho $ is the
injectivity radius of $M$. In this section, we use an explicit tubular
neighborhood of the diagonal embedding $\Delta \colon M\hookrightarrow
M\times M$ to give chain-level definitions for the loop products. We start
by introducing our model of the free loop space.

\subsection{$H^1$-loops}

\label{H1loopssec}

A path $\gamma \colon \lbrack 0,1]\rightarrow M$ is of Sobolev class $H^{1}$ if
$\gamma$ and its weak derivative $\ga'$ are of class $L^2$. 
This means in particular that $\gamma ^{\prime }(t)$ is defined almost everywhere, and that 
the energy 
\begin{equation*}
E(\gamma )=\int_{0}^{1}|\gamma ^{\prime }(t)|^{2}dt
\end{equation*}%
is finite. The \textit{length} 
\begin{equation*}
\ell (\gamma )=\int_{0}^{1}|\gamma ^{\prime }(t)|dt
\end{equation*}%
is also defined and finite for $H^{1}$-paths, and 
\begin{equation*}
E(\gamma )\geq (\ell (\gamma ))^{2},
\end{equation*}%
with equality if and only if $\gamma $ has constant speed. For technical
reasons the best function on path spaces is the square root of the energy
(which is also the $L^{2}$-norm) 
\begin{equation*}
\mathcal{L}(\gamma):=\sqrt{E(\gamma )}=\left( \int_{0}^{1}|\gamma ^{\prime
}(t)|^{2}dt\right) ^{\frac{1}{2}},
\end{equation*}%
which has units of length, but (like energy) it extracts a penalty for bad
parametrization. The reader who thinks of $\mathcal{L}$ as the length will
not go far wrong.

Our model for the loop space of $M$ will be the \textit{closed} $H^{1}$-\textit{paths}, also called $H^{1}$-\textit{loops}: 
\begin{equation*}
\Lambda=\Lambda M=\{\gamma \colon[0,1]\rightarrow M\ |\ \gamma \text{ is }%
H^{1}\text{ and } \gamma (0)=\gamma (1)\}\text{.}
\end{equation*}
Note that the inclusions 
\begin{equation*}
C^{\infty }\text{-paths }\subset \text{ piecewise }C^{\infty }\text{-paths }%
\subset H^{1}\text{-paths }\subset C^{0}\text{-paths}
\end{equation*}%
and 
\begin{equation*}
C^{\infty }\text{-loops }\subset \text{ piecewise }C^{\infty }\text{-loops }
\subset H^{1}\text{-loops }\subset C^{0}\text{-loops}
\end{equation*}
are continuous and induce homotopy equivalences (see \cite[Thm 1.2.10]{Kli78}
for the last inclusion, and e.g.~\cite[Thm 4.6]{Sta09}). 
The space $\Lambda$ of $H^1$-loops on $M$ is a Hilbert manifold \cite[Thm
1.2.9]{Kli78}, 
such that the energy is defined, and satisfies Condition C of Palais and Smale \cite{PS64}
(see \cite{Eli72}). 
Thus $\Lambda$ is appropriate for infinite dimensional Morse theory.

\medskip

A \textit{constant speed path, }or path \textit{parametrized proportional
to arc length} is an $H^{1}$-path $\gamma $ with $\left\vert \gamma ^{\prime
}(t)\right\vert $ constant where it is defined.  In this case $\left\vert
\gamma ^{\prime }(t)\right\vert =\mathcal{L(\gamma )}$ for all $t$ where the
derivative is defined.  We denote by $\overline{\Lambda }\subset \Lambda $
the space of constant speed loops.  Note that $\mathcal{L=\ell }$ on $%
\overline{\Lambda }$. Anosov \cite[Thm 2]{Ano80} has proved that there is a
continuous map 
\[
\Psi \colon \Lambda \longrightarrow \overline{\Lambda }
\]%
that takes an $H^{1}$-loop $\gamma \in \Lambda M$ to the constant speed
reparametrization of $\gamma $, exhibiting $\overline{\Lambda }$ as a
deformation retract of $\Lambda $. (See Theorem 3 in the same paper.) This
model of the loop space of $M$ will at times be useful for computations.

\subsection{Concatenation of paths}

\label{sec:concat}

If $\gamma ,\delta \colon \lbrack 0,1]\rightarrow M$ are paths with $\gamma
(1)=\delta (0)$, and $s \in (0,1)$ then the \emph{concatenation at time }%
$s:$
\[
\gamma \ast _{s}\delta \colon \lbrack 0,1]\rightarrow M
\]%
is defined by 
\[
(\gamma \ast _{s}\delta )(t)=\left\{ 
\begin{array}{c}
\gamma (\frac{t}{s})\text{ if }0\leq t\leq s \\ 
\delta (\frac{t-s}{1-s})\text{ if } s\leq t\leq 1%
\end{array}%
\right\}. 
\]

If $\gamma ,\delta \colon \lbrack 0,1]\rightarrow M$ are paths with $\gamma
(1)=\delta (0)$, then the \emph{optimal concatenation} 
\[
\gamma \star \delta \colon \lbrack 0,1]\rightarrow M
\]%
is given by 
\[
\gamma \star \delta =\gamma \ast _{\sigma}\delta 
\]%
where 
\[
\sigma =\frac{\mathcal{L}(\gamma )}{\mathcal{L}(\gamma )+\mathcal{L}(\delta )}
\]%
This concatenation has the property of being strictly associative, and has
minimal energy among reparametrizations of the form $\gamma \ast
_{s}\delta $. When restricted to based loops, the constant loop becomes
also a strict unit with this concatenation.

\subsection{Tubular neighborhood of the diagonal in $M$ and Thom classes}

\label{tubsec}

The normal bundle of the diagonal $\Delta \colon M\hookrightarrow M\times M$ is
isomorphic to the tangent bundle $TM$ of $M$. We will throughout the paper
identify $TM$ with its open disc bundle of radius $\varepsilon$, where we recall that $\eps$ is chosen small with respect to the injectivity radius $\rho$ of $M$. In particular, the exponential map is defined on vectors of $TM$.

We make
here the following explicit choice of tubular neighborhood for $\Delta $:
let
\begin{equation}\label{equ:TMeps}
  TM\equiv TM_\eps:= \{(x,V)\ |\ x\in M,V\in T_{x}M,|V|<\varepsilon \}
  \end{equation}
and 
\begin{equation}
\nu_M\colon TM \longrightarrow M^{2}  \label{nuM}
\end{equation}%
be the map taking $(x,V)$ to $(x,\exp _{x}(V))$, where $\exp
_{x}\colon T_{x}M\rightarrow M$ denotes the exponential map at the point $x\in M$.
Let $U_{M}=U_{M,\varepsilon }\subset M^{2}$ be the neighborhood of the
diagonal defined by 
\begin{equation*}
U_{M}=U_{M,\eps}:=\{(x,y)\in M^{2}\ |\ |x-y|<\varepsilon \}.
\end{equation*}%
Then the map $\nu _{M}$ has image $U_{M}$ and is a homeomorphism onto its
image. Let 
\begin{equation*}
\Thom(TM)= DTM/STM
\end{equation*}%
denote the Thom space of $TM$, its $\varepsilon _{0}$-disc bundle modulo
its $\varepsilon _{0}$-sphere bundle, where $0<\varepsilon _{0}<\varepsilon 
$. The \emph{Thom collapse map} associated to the tubular neighborhood $\nu
_{M}$ is the map 
\begin{equation*}
\kappa _{M}\colon M^{2}\longrightarrow \Thom(TM)
\end{equation*}%
defined by 
\begin{equation*}
\kappa _{M}(x,y)=\left\{ 
\begin{array}{ll}
(x,\exp _{x}^{-1}(y)) & |x-y|<\varepsilon _{0} \\ 
\ast & |x-y|\geq \varepsilon _{0}%
\end{array}%
\right.
\end{equation*}%
where $\ast $ is the basepoint of $\Thom(TM)$. \ Note that 
\begin{equation*}
\nu _{M}\colon TM \longrightarrow M^{2} \ \ \ \ \ \ \ \text{and}\ \ \ \ \ \
\kappa _{M}\colon M^{2}\longrightarrow \Thom(TM)
\end{equation*}%
are inverses of one another when restricted to the open $\varepsilon_0$%
-disc bundle $TM_{\varepsilon_0}$ and the $\varepsilon_0$-neighborhood $%
U_{M,\varepsilon_0}$ of the diagonal in $M^2$. 
We pick a cochain representative 
\begin{equation*}
u_{M}\in C^n(TM,TM_{\varepsilon_0}^c)
\end{equation*}
of the Thom class of the tangent bundle $TM$, which in particular means that $\kappa_M^*u_M$ vanishes on
chains supported on the complement of $U_{M,\varepsilon _{0}}$. We will in
the paper write 
\begin{equation*}
\tau_M:=\kappa_M^*u_M\in C^n(M^2,U_{M,\varepsilon_0}^c).
\end{equation*}
Its defining property is that $$\tau_M\cap [M\x M]=\De_*[M]\in H_n(M\times M).$$
By abuse of notations, we will also call $\tau_M$ its restriction to $C^n(U_{M},U_{M,\varepsilon_0}^c)$.

\begin{rem}\label{rem:orientations}
   We need to orient the normal bundle of $\De(M)$ inside $M\x M$ in such a way that the isomorphism $N\De(M)\oplus T\De(M)\cong T(M\x M)|_{\De(M)}$ preserves the orientation, so that capping with the Thom class does give the relation $\tau_M\cap [M\x M]=\De_*[M]$ without any sign. 
 One can check that this means that under the identification $N(\Delta M)\cong TM$,  
 the fibers of the normal bundle are oriented as $(-1)^n$ times the canonical orientation of those in the tangent bundle. Such details will be relevant when doing computations in
 Section~\ref{sec:computing}.
\end{rem}

The cap product plays a crucial role in our construction of the loop
(co)products. We make precise in Appendix~\ref{app:cap} which maps and
associated cap products we use, and give their relevant 
properties.  
In Example~\ref{ex:cap}, we detail how to define the cap product as a map 
\begin{equation*}
[\tau _{M}\cap]\colon C_{\ast }(M^{2})\longrightarrow C_{\ast}(U_{M}).
\end{equation*}%
The ingredients are the usual formula for the cap product, our tubular
neighborhood of $\Delta M$ in $M^{2}$, and a chain map from standard singular chains to ``small simplices":
for $A\in C_{\ast }(M^2)$,
$$[\tau_M\cap ](A)=\tau_M\cap \rho(A)$$
where $\rho\colon C_*(M^2)\to C_*(M^2)$ is a chain map that consistently takes barycentric subdivisions of the simplices, in such a way that 
 $\rho(A)$ is homologous to $A$ and every simplex $\sigma $ in $\rho(A)$ has
support in either $U_M$ or in $U_{M,\varepsilon _{0}}^{c}$ (see Section~\ref{sec:smallsimp}). Capping
with the Thom class kills all the simplices with support in $U_{M,\varepsilon _{0}}^{c}$.

\subsection{The homology product}

\label{sec:newCS}

Recall from Section~\ref{H1loopssec} that $\Lambda=\Lambda M$ denotes the
space of $H^1$-loops in $M$. We denote by 
\begin{equation*}
e\colon \Lambda\to M
\end{equation*}
the evaluation at $0$, that is the map taking a loop $\gamma$ to its
evaluation $e(\gamma)=\gamma(0)$. 

Let $\Lambda \times _{M}\Lambda \subset \Lambda ^{2}$ be the
\textquotedblleft figure-eight space", the subspace 
\begin{equation*}
\Lambda \times _{M}\Lambda =\{(\gamma ,\lambda )\in \Lambda ^{2}\ |\ \gamma
_{0}=\lambda _{0}\},
\end{equation*}%
where $\gamma _{0}=\gamma (0)$ and $\lambda _{0}=\lambda (0)$ are the
basepoints of the loops, and consider its neighborhood inside $\Lambda ^{2}$
defined by 
\begin{equation*}
U_{\CS}=U_{\CS,\varepsilon }:=\{(\gamma ,\lambda )\in \Lambda ^{2}\ |\ |\gamma
_{0}-\lambda _{0}|<\varepsilon \}=(e\times e)^{-1}(U_{M}).
\end{equation*}%
We define the retraction 
\begin{equation*}
R_{\CS}\colon U_{\CS}\longrightarrow \Lambda \times _{M}\Lambda
\end{equation*}%
by 
\begin{equation*}
R_{\CS}(\gamma ,\lambda )=(\gamma \text{ },\text{ }\overline{\gamma
_{0}\lambda_{0}}\;\star \;\lambda \;\star \;\overline{\lambda _{0}\gamma
_{0}})
\end{equation*}%
where, for $x,y\in M$ with $|x-y|<\rho $, $\overline{xy}$ denotes the
minimal geodesic path $[0,1]\rightarrow M$ from $x$ to $y$, and $\star $ is
the optimal concatenation of paths defined above. That is, we add small
\textquotedblleft sticks" $\overline{\gamma _{0}\lambda _{0}}$ \ and $%
\overline{\lambda _{0}\gamma _{0}}$ to $\lambda$ to make it have the same
basepoint as $\gamma$. (See Figure~\ref{fig:sticks}(a).) We will see in Lemma~\ref{retract} below that $R_{\CS}$ is a homotopy inverse to the inclusion $\La\x_M\La\inc U_{\CS}$.

Define 
\begin{equation*}
\tau _{\CS}:=(e\times e)^{\ast }\tau _{M}\in C^{n}(\Lambda \times \Lambda,U_{\CS,\varepsilon _{0}}^{c})
\end{equation*}%
as the pulled-back Thom class $\tau _{M}$ by the evaluation map. The class $\tau_{\CS}$ has support in $U_{\CS}$, and capping with this
class thus defines a map 
\begin{equation*}
[\tau _{\CS}\,\cap] \;\colon C_{\ast }(\Lambda \times \Lambda )\longrightarrow
C_{\ast -n}(U_{\CS}).
\end{equation*}%
Just like in the case of $M\x M$ above, given $A\in C_{\ast }(\Lambda \times
\Lambda )$, the maps has the form $$[\tau_{\CS}\cap](A)=\tau_{\CS}\cap \rho(A)$$ for $\rho(A)$ a chain homologous to $A$ with the property that
each simplex in $\rho(A)$ is supported in either $U_{\CS}$ or $U_{\CS,\varepsilon _{0}}^{c}$ so that  $\tau _{\CS}\cap\rho A$ has support in $U_{\CS}$. 
(See Appendix~\ref{app:cap}, in particular equation (\ref{equ:capdef}) and Example~\ref{ex:cap}, for more details.)

Note that the retraction $R_{\CS}$ is well-defined on the image of $[\cap \tau_{\CS}]$. Hence we can make the following definition. 

\begin{Def}
\label{def:stickypro} We define the loop product $\wedge_{\Thom}$ as the map 
\begin{equation*}
\wedge_{\Thom}\colon C_p(\Lambda)\otimes C_q(\Lambda)\longrightarrow
C_{p+q-n}(\Lambda)
\end{equation*}
given by 
\begin{equation*}
A\wedge_{\Thom} B=\concat \left( R_{\CS}([\tau _{\CS} \cap] (A\times B))\right) 
\end{equation*}
where $\concat\colon\Lambda\times_M \Lambda\to \Lambda$ is the concatenation $\star$.
\end{Def}

The map $(A,B)\rightarrow A\wedge_{\Thom} B$ is a composition of chain maps,
and thus descends to a product on homology:%
\begin{equation*}
\wedge_{\Thom}\colon H_{p}(\Lambda )\otimes H_{q}(\Lambda )\rightarrow
H_{p+q-n}(\Lambda )
\end{equation*}

\smallskip

\emph{Note that this is an integral singular chain level definition of a
loop product, and that it does not involve any infinite dimensional tubular
neighborhood, but only a simple retraction map and the pullback by the
evaluation map of the Thom class of $TM$. Indeed the retraction map
interacts only weakly with the loops in $U_{\CS}$, adding ``sticks" that
depend only on the basepoints of $\gamma$ and $\delta$. We will show in
Proposition~\ref{prop:CSequiv} below that this definition of the
Chas-Sullivan product coincides in homology with that of Cohen and Jones in 
\cite{CohJon}. We also show in Proposition~\ref{prop:transCS} that it intersects and then concatenates transverse chains in the way originally envisionned by Chas and Sullivan.}

\medskip

We emphasized in the notation $\wedge_{\Thom}$ of the product that it is
defined using a Thom-Pontrjagin like construction. To have nice algebraic
properties, such as associativity and graded commutativity in appropriately degree-shifted homology, one needs to
correct the above product by a sign: 

\begin{Def}\label{def:CSalg}
The \emph{algebraic} loop product $\wedge$ is the map 
\begin{equation*}
\wedge\colon C_p(\Lambda)\otimes C_q(\Lambda)\longrightarrow
C_{p+q-n}(\Lambda)
\end{equation*}
given by 
\begin{equation*}
A\wedge B:=(-1)^{n-np}A\wedge_{\Thom}B.
\end{equation*}
\end{Def}

The product is homotopy commutative via a geometric homotopy that rotates the target loop, while the associativity relation is closer to holding on chains, though it does not hold on the nose on the chain level in particular  because of the maps $[\tau_{\CS}\cap]$.  (See Theorem~\ref{thm:CSalg}  and Proposition~\ref{prop:intalg}.)

The sign $(-1)^n$ is a normalization, that makes $[M]$ the unit of the product (see Example~\ref{ex:unit} for the sign computation), while the sign $(-1)^{np}$ could be associated to a degree $n$ suspension. Appendix~\ref{app:signs} gives details about the corresponding sign issue for the intersection product.

We note that the idea of using ``sticks'' instead of deforming the loops can
already be found in \cite{Tam}, where the same sign convention is also taken.

\subsection{The homology coproduct}\label{sec:newGH}

Let $\F\subset \Lambda \times I$ be the subspace  
\begin{equation*}
\F =\{(\gamma ,s)\in \Lambda \times I\ |\ \gamma _{0}=\gamma (s)\},
\end{equation*}%
with $\gamma _{0}=\gamma (0)$ as above, and consider its neighborhood 
\begin{equation*}
U_{\GH}=U_{\GH,\varepsilon }:=\{(\gamma ,s)\in \Lambda \times I:|\gamma
_{0}-\gamma (s)|<\varepsilon \}=(e_{I})^{-1}(U_{M})\ \subset \
\Lambda \times I
\end{equation*}%
 where 
\begin{equation*}
e_{I}\colon \Lambda \times I\longrightarrow M^{2}
\end{equation*}%
is our notation for the map that takes a pair $(\gamma ,s)$ to the pair $%
(\gamma _{0},\gamma (s))$. Define the retraction 
\begin{equation*}
R_{\GH}\colon U_{\GH}\longrightarrow \F 
\end{equation*}%
by 
\begin{equation*}
R_{\GH}(\gamma ,s)=\Big(\big(\,\gamma \lbrack 0,s]\text{ }\star \text{ }%
\overline{\gamma (s)\gamma _{0}}\text{ }\big)\ast _{s}\big(\,\overline{%
\gamma _{0}\gamma (s)}\star \text{ }\gamma \lbrack s,1]\,\big)\ ,\ s\Big)
\end{equation*}%
where $\gamma [a,b]$ denotes the restriction of $\gamma $ to the
interval $[a,b]\subset [0,1]$ precomposed by the scaling map $%
[0,1]\rightarrow [a,b]$, and where $\overline{xy}$ denotes again the
shortest geodesic from $x$ to $y$. (See Figure~\ref{fig:sticks}(b).) Note
that the retraction is well-defined also when $s=0$ and $s=1$; in fact $%
R_{\GH}(\gamma ,s)=(\gamma ,s)$ for any $\gamma $ whenever $s=0$ or $s=1$ or whenever $\ga$ is constant, that is it restricts to the identity on
$$\B:=M\x I\cup \La\x \del I \ \subset \ \F\ \subset \ U_{\GH}.$$
We will see in Lemma~\ref{retract}  that $R_{\GH}$ is a homotopy inverse to the inclusion
$$(\F,\B)\inc (U_{\GH},\B).$$

Just as above, given $A\in C_{\ast }(\Lambda \times I)$ we have a homologous
chain $\rho(A)\in C_{\ast }(\Lambda \times I)$ with the property that
each simplex  in $\rho(A)$ is supported in either $U_{\GH}$ or in $U_{\GH,\varepsilon _{0}}^{c}$, so that  $e_{I}^{\ast }\tau _{M}\cap
\rho(A)$ has support in $U_{\GH,\varepsilon }$. Thus if 
\begin{equation*}
\tau _{\GH}:=(e_{I})^{\ast }\tau _{M}\in C^{n}(\Lambda \times
I,U_{\GH,\varepsilon _{0}}^{c})
\end{equation*}%
is the pulled-back Thom class, it is meaningful to apply the retraction map $%
R_{\GH}$ to $$[\tau_{\GH}\cap](A)=\tau _{\GH}\cap \rho(A)\in C_*(U_{\GH}).$$

Let 
\begin{equation*}
\cut\colon \F \rightarrow \Lambda \times \Lambda
\end{equation*}
be the cutting map, defined by $\cut(\gamma ,s)=(\gamma[0,s],\gamma[s,1])$.

\begin{Def}
\label{def:stickyco} We define the loop coproduct $\vee_{\Thom}$ as the map  
\begin{equation*}
\xymatrix{C_*(\La,M) \ar[r]^-{\vee_{\Thom}} & C_*(\La,M)\ot C_*(\La,M)
}
\end{equation*}
of degree $1-n$ given by 
\begin{equation*}
\vee_{\Thom} A=\AW (\cut (R_{\GH}([\tau _{\GH} \cap] (A\times I))));
\end{equation*}
that is, $\vee_{\Thom}$ is the composition of the following maps of relative chains: 
\begin{eqnarray*}
C_k(\La,M)=C_k(\Lambda)/C_k(M) &\overset{\times I}{\longrightarrow }&C_{k+1}(\Lambda
\times I)/\big(C_{k+1}(M\times I) +C_{k+1}(\Lambda \times \del I)\big) \\
& \xrightarrow{[\tau_{\GH}\cap]}&C_{k+1-n}(U_{\GH})/\big(C_{k+1-n}(M\times I) +C_{k+1-n}(\Lambda \times \del I)\big) \\
& \xrightarrow{R_{\GH}}&C_{k+1-n}(\F )/\big(C_{k+1-n}(M\times I) +C_{k+1-n}(\Lambda \times \del I)\big) \\
& \xrightarrow{\cut}&C_{k+1-n}(\Lambda \times \Lambda)/\big(C_{k+1-n}(M\times \Lambda) + C_{k+1-n}(\Lambda \times M)\big)\\
                     & \xrightarrow{\AW}&\! \! \bigoplus_{p+q=k+1-n}\!\! C_p(\Lambda,M) \otimes C_q(\Lambda,M), 
\end{eqnarray*}
where the last map is the Alexander-Whitney quasi-isomorphism. 
\end{Def}

\emph{As for the product,  Definition \ref{def:stickyco} is a chain level definition of the coproduct, and involves only a simple
retraction map and the pullback via the exponential map of the Thom class
from $TM$. It is a composition of chain maps and descends to a map $%
H_{k}(\Lambda ,M)\rightarrow H_{k+1-n}(\Lambda \times \Lambda ,M\times
\Lambda \cup \Lambda \times M)$.}

We will show in Proposition~\ref{prop:coequiv} below that this definition
coincides on homology with that given by Goresky and Hingston in \cite{GorHin}.

\begin{rem} \label{rem:+U}
(1) Recall that the Alexander-Whitney map
$$\AW\colon  C_*(\La\x \La)\sta{\sim}{\rar} C_*(\La)\otimes C_*(\La)$$
is a homotopy inverse to the cross product, or Eilenberg-Zilber map. It induces a quasi-isomorphism 
$$C_*(\La\x \La)/\big(C_*(M\times \Lambda) + C_*(\Lambda \times M)\big) \sta{\simeq}{\rar} C_*(\La,M)\otimes C_*(\La,M)$$
as it takes $C_*(M\x\La)$ quasi-isomorphically to $C_*(M)\ot C_*(\La)$ and $C_*(\La\x M)$ to $C_*(\La)\ot C_*(M)$. 
The complex $C_*(M\times \Lambda) + C_*(\Lambda \times M)$ is a subcomplex of the complex $C_*(M\times \Lambda
\cup \Lambda \times M)$, and the quotient map 
\begin{equation*}
C_{*}(\Lambda \times \Lambda)/\big(C_*(M\times \Lambda) + C_*(\Lambda
\times M)\big) \twoheadrightarrow C_{*}(\Lambda \times
\Lambda)/C_*(M\times \Lambda \cup \Lambda \times M)
\end{equation*}
is a quasi-isomorphism.
(This can for example be seen by replacing $M$ by the homotopy equivalent subspace 
$\Lambda^{<\alpha}$ of small loops in $\Lambda$, with $\alpha<\rho$, and
using ``small simplices'' associated to the covering $U_1=\Lambda\times
\Lambda^{<\alpha}$ and $U_2=\Lambda^{<\alpha}\times \Lambda$ of $\Lambda\times \Lambda^{<\alpha}\cup \Lambda^{<\alpha}\times \Lambda$.)
In particular, both the above complexes, and hence also the target complex $C_*(\La,M)\ot C_*(\La,M)$ of the coproduct, have homology the relative homology group
$$H_{*+1-n}(\Lambda \times \Lambda, M\times \Lambda \cup \Lambda \times M).$$

\smallskip

\noindent
(2) If we take coefficients in a field $\mathbb{F}$, or if the homology is
torsion free, then we can use the K\"unneth isomorphism to
get a coproduct 
\begin{equation*}
\vee_{\Thom} \colon H_{*}(\Lambda ,M;\mathbb{F})\ \longrightarrow\ H_{*+1-n}(\Lambda \times \Lambda, M\times \Lambda \cup \Lambda \times M;\mathbb{F})\ \cong\!\!\!\! \underset{p+q=*+1-n}{%
\bigoplus}\!\!\!\! H_{p}(\Lambda ,M;\mathbb{F})\otimes H_{q}(\Lambda ,M;\mathbb{F}).
\end{equation*}%
However,  there is in general no K\"unneth isomorphism, and indeed no natural
map from $H_{\ast }(U\times V)\cong H_*(C_*(U)\otimes C_*(V))$ to $H_{\ast }(U)\otimes H_{\ast }(V)$.

\smallskip

\noindent
(3) The use of relative chains is dictated by the fact that we cross with an interval, which only induces a chain map in relative chains. One could instead cross with a circle, but in that case the cutting map will take value in $\La\x \La/(\Z/2)$, with $\Z/2$ exchanging the loops. Indeed, the cup map at times  $0=1$  would not know anymore which is the left and which is the right loop. 
\end{rem}

\medskip

Just like the homology product, the above definition caping with a Thom class does not have good  coassociative and cocommutative properties.
We will show in Theorem~\ref{thm:coGHalg}  that the same sign change as for the product gives better algebraic properties. This justifies the following definition:

\begin{Def}
The \emph{algebraic} loop coproduct $\vee$ is the map 
\begin{equation*}
\vee\colon C_p(\Lambda)\longrightarrow C_*(\La,M)\ot C_*(\La,M)   
\end{equation*}
given by 
\begin{equation*}
\vee A:=\sum (-1)^{n+np}A^0_{p}\ot  A^1_q.
\end{equation*}
for $A_p^0\ot A_q^1$ the degree $(p,q)$ component of $\vee_{\Thom}A$. 
\end{Def}

Our sign change for the algebraic coproduct differs from the sign chosen in \cite{GorHin}. It is justified by Theorem~\ref{thm:coGHalg}, see also Remark~\ref{rem:HH}, and has the advantage of being the same sign change as for the product, with results like Theorem~\ref{thm:trivial} working both with Thom signs and algebraic signs.
The coproduct was defined, just like the product, as a lift of the Thom-signed intersection product. However in the case of the coproduct, it is more difficult to directly deduce that sign change from that of the intersection product.

\subsection{The cohomology product}

The homology coproduct $\vee_{\Thom}$ defined above induces a dual map 
\begin{equation*}
\vee_{\Thom} ^{\ast }\colon C^*(\Lambda,M) \otimes C^*(\Lambda,M)\  \longrightarrow\ \ C^{*}(\Lambda,M)
\end{equation*}%
of degree $n-1$.

\begin{Def}
\label{def:cast} The product 
\begin{equation*}
\circledast_{\Thom}\colon C^{p}(\Lambda ,M)\otimes C^{q}(\Lambda ,M)\rightarrow
C^{p+q+n-1}(\Lambda ,M)
\end{equation*}
is defined by 
\begin{equation*}
a\circledast_{\Thom} b=\vee_{\Thom}^*(a\otimes b)
\end{equation*}
where $\vee_{\Thom}$ is given in Definition~\ref{def:stickyco}.
\end{Def}

Applying homology, this defines a cohomology product 
\begin{equation*}
\circledast_{\Thom}\colon H^{p}(\Lambda ,M)\otimes H^{q}(\Lambda ,M) \sta{\x}{\longrightarrow} H^{p+q}(\La\x \La,M\x\La\cup \La\x M) \xrightarrow{\vee_{\Thom}^*}
H^{p+q+n-1}(\Lambda ,M)
\end{equation*}
dual to the homology coproduct, but which is defined with integer
coefficients on the tensor product of the cohomology groups as there is no K\"unneth formula issue here. We will show in Theorem~\ref{thm:cocoequ}
 that this product is equivalent on cohomology (up to sign) to the
cohomology product defined in \cite{GorHin}.

\medskip

Just like for the loop product and coproduct, the cohomology product needs to be corrected
by a sign for better associativity and commutativity properties
(see Theorem~\ref{thm:coGHalg}).

\begin{Def}
\label{def:castalg} The \emph{algebraic} product 
\begin{equation*}
\circledast\colon C^{p}(\Lambda ,M)\otimes C^{q}(\Lambda ,M)\rightarrow
C^{p+q+n-1}(\Lambda ,M)
\end{equation*}
is defined by 
\begin{equation*}
a\circledast b=(-1)^{n+np}a\circledast_{\Thom}
b=\vee^*\!(a\times b).
\end{equation*}
\end{Def}

We will show in Theorem~\ref{thm:coGHalg} that this product is a graded associative and anti-commutative product of degree $-1$ in degree-shifted cohomology. It is not unital (see eg.~the computation \cite[15.3]{GorHin} in the case of odd spheres).
On cochains, the product is homotopy commutative (after degree-shifting) 
via a geometric homotopy that flips the interval, explaining that it is anti-commutative rather than commutative in cohomology, and the associativity relation does not hold exactly on the nose on cochains because of the maps $[\tau_{\GH}\cap]$.  See Theorem~\ref{thm:coGHalg}  for more details.

\section{Equivalence of the new and old definitions}

\label{sec:equivalence}

Cohen and Jones' approach to defining the loop product, later also used in 
\cite{GorHin} for the coproduct, was via collapse maps associated to certain
tubular neighborhoods in loop spaces. To show that our definition is
equivalent with these earlier definitions, we will write down explicit such
tubular neighborhoods, and show that the associated collapse maps can be
modeled by sticks as described above. Chas and Sullivan originally defined
their product in \cite{CS99} by intersecting chains. This other approach has
been made precise by several authors (see e.g.~\cite{Cha05,Mei09,Lau11}),
and shown to be equivalent to the tubular neighborhood approach in e.g.~\cite%
{Mei09}. We will give in Section~\ref{sec:computing} a direct proof that, under appropriate transversality condition, the product and coproduct in our definition are computed as originally envisionned, by appropriately intersecting chains, and then applying the relevant concatenation or cutting map. 

\subsection{Pushing points in $M$}

We define a map that allows us to \textquotedblleft push
points\textquotedblright\ continuously in $M$. This map will allow us in
later sections to lift the tubular neighborhood of the diagonal to tubular
neighborhoods of analogous subspaces of the loop space of $M$ or related
spaces.

\bigskip

\begin{lem}
\label{H} Let $M$ be a Riemannian manifold with injectivity radius $\rho$.
Let $U_{M,\rho}\subset M\times M$ be the set of points $(u,v)$ where $%
|u-v|<\rho$. There is a smooth map 
\begin{eqnarray*}
h\colon  U_{M,\rho}\times M\subset M\times M \times M &\longrightarrow &M
\end{eqnarray*}%
with the following properties:

\begin{enumerate}[(i)]

\item For each $(u,v)\in U_{M,\frac{\rho}{14}}$, $h(u,v):=h(u,v,\_)\colon M\to M$
is a diffeomorphism.

\item $h(u,v)(u)=v$ if $|u-v|<\rho /14$.

\item $h(u,v)(w)=w$ if $|w-u|\ge \rho/2$.

\item $h(u,u)(w)=w$.
\end{enumerate}
So $h(u,v)$ takes $u$ to $v$ if they are close enough and does not move $w$
if $w$ is far from $u$.
\end{lem}

We encourage the reader to realize that it is quite clear that such a map
should exist, but do give a proof for completeness, giving an explicit such map $h$, that allows us to directly check that it has the desired property. The reader could alternatively define $h$ as the flow associated to an appropriately chosen family of vector fields on $M$.  

\begin{proof}
Let $\mu \colon [0,\infty )\rightarrow \mathbb{R}$ be a smooth map with

\begin{minipage}{0.6\textwidth}
\begin{equation*}
\begin{array}{rcl}
\mu (r) & = & 1\text{ if }r\leq \rho /4 \\ 
\mu (r) & = & 0\text{ if }r\geq \rho /3 \\ 
\frac{-13}{\rho } & \leq & \mu ^{\prime }(r)\leq 0%
\end{array}
\end{equation*}
\end{minipage}
\begin{minipage}{0.3\textwidth}
\begin{lpic}{bump(0.45,0.45)}
\lbl[t]{45,1;$\frac{\rho}{4}$}
\lbl[t]{65,1;$\frac{\rho}{3}$}
\lbl[t]{55,37;$\mu$}
\lbl[t]{2,36;$1$}
\end{lpic}
\end{minipage}

\medskip

\noindent If $u,v\in M$ with $|u-v|<\rho $, we define $h(u,v)\colon M\rightarrow M$
by 
\begin{eqnarray*}
h(u,v)(w) &=&\left\{ 
\begin{array}{ll}
\exp _{u}\big(\exp _{u}^{-1}(w)+\mu (|u-w|)\exp _{u}^{-1}(v)\big) & \text{
if }|u-w|\leq \rho /2 \\ 
w & \text{ if }|u-w|\geq \rho /2%
\end{array}%
\right.
\end{eqnarray*}%
First we will show that $h(u,v)(w)$ depends smoothly on $u,v,$ and $w$. It
is enough to show that $\exp _{u}^{-1}(w)$ depends smoothly on $u$ and $w$
if $|u-w|<\rho $. Recall that $TM$ is the open disk bundle of radius $%
\varepsilon$, and $\varepsilon < \rho $. Recall from Section~\ref{tubsec}
the map $\nu_M\colon TM\rightarrow M\times M$ defined by 
\begin{equation*}
\nu_M (u,W) = (u,\exp _{u}W).
\end{equation*}
It is a smooth bijection onto $U_M=\{ (u,w)\ |\ |u-w|<\varepsilon \}$, since
there is a unique geodesic of length $<\varepsilon$ from $u$ to $w$ if $%
|u-w|<\varepsilon$. Thus it will be enough to show that the derivative $%
d\nu_M $ has maximal rank at each point in the domain $TM$ of $\nu_M$. But
at a point $(u,w)$ in the image, the image of $d\nu_M \colon TTM\rightarrow
TM\times TM$ clearly contains all vectors of the form $(0,V)$, since $\exp
_{u}\colon T_{u}M\rightarrow M$ is a diffeomorphism onto $\left\{ w\ |\
|u-w|<\varepsilon \right\}$. And the projection of the image of $d\exp $
onto the first factor $TM$ is clearly surjective, since ``we can move $u$ in 
$TM$". \ Thus $h(u,v)(w)$ is smooth in $u,v$, and $w$.

To see that $h(u,v)$ is a diffeomorphism when $|u-v|<\frac{\rho }{14}$, use $%
\exp _{u}$ to identify $T_{u}M $ with the neighborhood $\left\{ w\ |\
|u-w|<\varepsilon \right\} $ of $u$ in $M $. \ In these coordinates, if $%
v=\exp _{u}V$ and $W\in T_{u}M$, 
\begin{equation*}
h(u,V)(W):=\exp _{u}^{-1}h(u,v)(\exp _{u}W)=\left\{ 
\begin{array}{ll}
W+\mu (|W|)V & \text{ if }|W|\leq \rho /2 \\ 
W & \text{ if }|W|\geq \rho /2.%
\end{array}%
\right.
\end{equation*}
Note that $h(u,V)$ preserves the ``lines" $\{W_{0}+sV\ |\ s\in \mathbb{R}\}$. On
a fixed line we have 
\begin{eqnarray*}
\frac{d}{ds}h(u,V)(W_{0}+sV) &=&\frac{d}{ds}\left( (W_{0}+sV)+\mu
(|W_{0}+sV|)V\right) \\
&=&(1+\mu ^{\prime }(|W_{0}+sV|)\text{ }\frac{d}{ds}(|W_{0}+sV|))V
\end{eqnarray*}%
where the coefficient $(1+\mu ^{\prime }(|W_{0}+sV|)$ $\frac{d}{ds}%
(|W_{0}+sV|))$ $\ $of $V$ on the right is at least 
\begin{equation*}
(1-\frac{13}{\rho }\frac{\rho }{14})>0
\end{equation*}%
because $\mu^{\prime}(r)>-\frac{13}{\rho }$ for all $r$ and $|V|=|u-v|<\frac{%
\rho }{14}$. Thus $h(u,V)$\ is monotone on each line. It follows that, for
each $u,v$, the map \ $h(u,V)$ is a local diffeomorphism and is bijective,
and thus the same is true for $h(u,v)\colon M\rightarrow M$.
\end{proof}

\subsection{Tubular neighborhood of the figure-eight space and equivalence
to Cohen-Jones' definition of the Chas-Sullivan product}

\label{sec:CStub}

There is a pull-back diagram 
\begin{equation*}
\xymatrix{\La\x_M\La \ \ar[d]_{e}\ar@{^(->}[r] & \La\x \La \ar[d]^{e\x e} \\
M\ \ar@{^(->}[r]^-\De & M\x M. }
\end{equation*}%
Let $e^{\ast }(TM)\rightarrow \Lambda \times _{M}\Lambda $ be the pull-back
of the tangent bundle of $M$ along the evaluation at $0$. A number of
authors have already noticed that a tubular neighborhood of $M$ sitting as
the diagonal inside $M\times M$ can be lifted to a tubular neighborhood of $%
\Lambda \times _{M}\Lambda $ inside $\Lambda \times \Lambda $ (see e.g.~\cite%
{AbbSch,CohJon,Mei09}). We give now an explicit such tubular neighborhood,
compatible with the tubular neighborhood $\nu _{M}$ of the diagonal in $%
M^{2} $ constructed in Section~\ref{tubsec}. This construction is suggested
in \cite{AbbSch}.

\begin{prop}
\label{TubCS} Let $M$ be a Riemannian manifold with injectivity radius $\rho$
and let $0<\varepsilon< \frac{\rho}{14}$. Identify as above $TM$ with its $%
\varepsilon$-tangent bundle. The subspace $\Lambda\times_M\Lambda$ of $%
\Lambda\times\Lambda$ admits a tubular neighborhood 
\begin{equation*}
\nu_{\CS}\colon \xymatrix{e^*(TM)\ \ar[d]\ar@{^(->}[r] & \Lambda\times\Lambda,\\
\La\x_M\La & }
\end{equation*}
that is $\nu_{\CS}$ restrict to the inclusion on the zero-section $%
\Lambda\times_M\Lambda$ and is a homeomorphism onto its image $U_{\CS}=\{(\gamma,\lambda)\ |\ |\gamma(0)-\lambda(0)|< \varepsilon\}$. Moreover,
one can choose this tubular neighborhood so that it is compatible with our
chosen tubular neighborhood of the diagonal in $M^2$ in the sense that the
following diagram commutes: 
\begin{equation*}
\xymatrix{e^*(TM)\ \ \ar@{^(->}[r]^{\HH}\ar[d]_{e} &\  \La\x \La\ar[d]^{e\x
e}\\ TM\ \ \ar@{^(->}[r]^{\nu_M}&\  M\x M }
\end{equation*}
where $\nu_M$ is the tubular neighborhood of the diagonal (\ref{nuM}).
\end{prop}

\begin{proof} Recall that vectors $V\in TM\equiv TM_{\eps}$ are assumed to have length at most $\eps<\frac{\rho}{14}$. In particular, the exponential map makes sense on such vectors.
  
Let $(\gamma,\delta,V)\in e^*(TM)$, so $(\gamma ,\delta )\in \Lambda \times
_{M}\Lambda $ and $V\in T_{\gamma_0}M=T_{\delta_0}M$, for $\gamma_0=\delta_0$
the startpoints of the loops. The map $\nu _{\CS}$ is defined by $\nu
_{\CS}(\gamma ,\delta ,V)=(\gamma ,\lambda)\in \Lambda\times
\Lambda$ with 
\begin{equation*}
\lambda =h\big(\delta_0,\exp _{\delta_0}(V)\big)\circ \delta 
\end{equation*}%
where $h$ is the map of Lemma~\ref{H}. The map $\nu_{\CS}$ is continuous with image the space
${U}_{\CS}:=\{(\gamma ,\lambda )\ |\ |\gamma_0-\lambda_0|<\varepsilon \}\subset
\Lambda \times \Lambda $. Define now $\kappa _{\CS}\colon {U}_{\CS}\rightarrow
e^{\ast }(TM)$ by $\kappa_{\CS} (\gamma ,\lambda )=(\gamma ,\delta ,V)$
where 
\begin{eqnarray*}
\delta  &=&(h(\gamma_0,\lambda_0))^{-1}\circ \lambda  \\
V &=&\exp _{\gamma_0}^{-1}\lambda_0
\end{eqnarray*}%
The map $\kappa _{\CS}$ is again continuous, and is an inverse for $\nu
_{\CS}$. The compatibility of $\nu_{\CS}$ and $\nu_M$ follows from the
properties of $h$.
\end{proof}

Note that the inverse $\kappa_{\CS}$ of the tubular embedding $\nu_{\CS}$
extends to a ``collapse map'', for which we use the same notation, 
\begin{equation}  \label{muLa}
\kappa_{\CS}\colon \Lambda^2 \longrightarrow \Thom(e^*TM)
\end{equation}
by taking $(\gamma,\delta)$ to $\kappa(\gamma,\delta)=\nu_{C\!
S}^{-1}(\gamma,\delta)$ if $(\gamma,\delta)\in U_{\CS}$ and to the basepoint
otherwise. Restricted to $U_{\CS}$, this map has two components: $\kappa_{C\!
S}=(k_{\CS},v_{\CS})$, where 
\begin{equation*}
k_{\CS}\colon U_{\CS}\longrightarrow \Lambda\times_M\Lambda
\end{equation*}
is a retraction map. We will now show that our retraction map $R_{\CS}$ of
Section~\ref{sec:newCS} that add sticks instead of deforming the loops,
gives an approximation of the retraction component $k_{\CS}$ of the
collapse map, and that both maps are homotopy inverses to the inclusion. 

\begin{lem}
\label{retract} The maps $$R_{\CS},k_{\CS}\colon U_{\CS}\longrightarrow\Lambda\times_M\Lambda$$
are homotopic. Moreover, they both define deformation retractions of $U_{\CS}$ onto $\La\x_M\La$. 
\end{lem}

\begin{figure}[h]
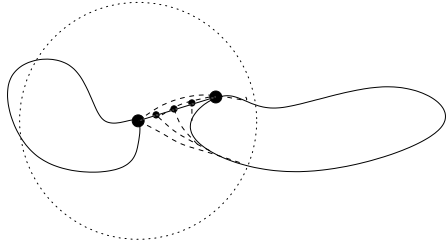

\centering
\begin{lpic}{CShtpy(0.5,0.5)}
 \lbl[r]{-1,40;$\ga$}
\lbl[l]{117,38;$\delta$}
\end{lpic}
\caption{Homotopy between the retractions $k_{\CS}$ and $R_{\CS}$.}
\label{fig:CShtpy}
\end{figure}

\begin{proof} We first show that the two maps are homotopic. 
Note that the only difference between the maps $k _{\CS}$ and $R_{\CS}$ is
their value on the second factor of $\Lambda\times_M\Lambda$; for $(\gamma,\lambda)\in U_{\CS}$, this second factor is given as 
\begin{equation*}
\delta =(h(\gamma_0,\lambda_0))^{-1}\circ \lambda 
\end{equation*}%
in the first case and the loop $\overline{\gamma_0\lambda_0}\star\lambda\star\overline{\lambda_0\gamma_0}$ in the second case. To prove
the first part of the  lemma, it is enough to define a homotopy $H\colon U_{\CS}\times I \rightarrow
\Lambda$ between these two maps, with the property that $H(\gamma ,\lambda,s)(0)=\gamma_0$ for all $s$. Such a homotopy $H$ can be given as follows:
set $H(\gamma ,\lambda ,s)$ to be the loop 
\begin{equation*}
\overline{\gamma_0\lambda_0}\,[0,s]\ \star\ \big((h(\overline{\gamma_0\lambda_0}(s),\lambda_0)^{-1}\circ \lambda\big)\text{ }\star \text{ }\overline{\gamma_0\lambda_0}\,[0,s]^{-1}
\end{equation*}%
where $\overline{\gamma_0\lambda_0}[0,s]$ is the geodesic from $\gamma_0$
to the point $\overline{\gamma_0\lambda_0}(s)$ at time $s$ on the geodesic,
parametrized on $[0,1]$, from $\gamma_0$ to $\lambda_0$, and $\overline{%
\gamma_0\lambda_0}\,[0,s]^{-1}$ is the same geodesic but in the reversed
direction. (See Figure~\ref{fig:CShtpy} for and illustration of this
homotopy.)

\smallskip

Let $j\colon \La\x_M\La\to U_{\CS}$ denote the inclusion. 
Both maps $R_{\CS}$ and $k_{\CS}$ restrict to the identity on $j(\La\x_M\La)$, which is the subspace of pairs of loops $(\lambda,\delta)$ such that $\lambda_0=\delta_0$. For the second part of the statement,
we need to show that there are homotopies $G_R,G_k\colon U_{\CS}\x I\to U_{\CS}$ between $j\circ R_{\CS}$ (resp.~$j\circ k_{\CS}$) and the identity, relative to $\La\x_M\La$. 
It is enough to show that one such homotopy because the homotopy $R_{\CS}\simeq k_{\CS}$ just constructed restricts to the identity on $\La\x_M\La$. We define $G_R$ by the following formula:
$$G_R(\ga,\lambda,s)=\big(\ga,\overline{\gamma_0\lambda_0}[s,1]\star\lambda\star\overline{\ga_0\lambda_0}[s,1]^{-1}\big),$$
i.e.~$G_R$ retracts the added geodesic path $\overline{\ga_0\lambda_0}$. 
\end{proof}

Cohen and Jones define the Chas-Sullivan product using a composition 
\begin{equation*}
C_*(\Lambda)\otimes C_*(\Lambda) \xrightarrow{\x} C_*(\Lambda\times\Lambda)\longrightarrow C_*(\Thom(e^*(TM)|_{\La\x_M\La})) \xrightarrow{\cong} C_{*-n}(\La\x_M \La)\xrightarrow{\concat} C_{*-n}(\Lambda)
\end{equation*}
where 
the second is the collapse map associated to a (not explicitly given) tubular neighborhood of $\Lambda\times_M\Lambda$ inside $\Lambda\times \Lambda$ of the form
constructed in Proposition~\ref{TubCS}, and the third one is the 
Thom isomorphism. (See \cite[Thm 1]{CohJon} or \cite[(1.7) and Prop 5]{CJY}
for a detailed description.)

\begin{prop}
\label{prop:CSequiv} The loop product $\wedge_{\Thom}$ of Definition~\ref{def:stickypro} is chain homotopic to the explicit representative of the
Cohen-Jones product of \cite{CohJon} obtained using Proposition~\ref{TubCS}.
\end{prop}

\begin{proof}
The map $\kappa_{\CS}$ of (\ref{muLa}) is an explicit collapse map associated
to a chosen tubular neighborhood of the form considered in \cite[Sec 1]%
{CohJon}, allowing the following explicit comparison of their approach to ours, where $[\tau\cap]$ is the Thom isomophism, capping with the Thom class of $e^*TM$:
$$\xymatrix{C_*(\La\x \La) \ar[r]^-{\kappa_{\CS}}\ar[d] &  C_*(\Thom(e^*(TM)|_{\La\x_M\La})) \ar[r]_-\cong^-{[\tau\cap]} &  C_{*-n}(\La\x_M \La) \\
C_*(\La\x \La,U_{\CS}^c) \ar[r]^-{[\tau_{\CS}\cap]}  \ar[ur]^-{\kappa_{\CS}} & C_{*-n}(U_{\CS}) \ar[ur]_-{R_{\CS}}  &
  }$$
  We have that $\kappa_{\CS}=(k_{\CS},v_{\CS})\simeq (R_{\CS},v_{\CS})$ by Lemma~\ref{retract}, and the Thom classes are compatible under this map as they are both pulled back from $\tau_M$.
 Hence the square commutes by the naturality of the cap product (see Appendix~\ref{app:cap}), which proves the result. 
\end{proof}

\begin{thm}
  \label{thm:CSalg} The algebraic loop product $\wedge$ is a unital, associative and graded commutative product of degree 0 on degree shifted homology
 $$\wedge\colon
 H_{p+n}(\Lambda)\otimes H_{q+n}(\Lambda)\to H_{(p+q)+n}(\Lambda).$$
 In particular, for $[A]\in H_{p+n}(\La)$ and $[B]\in H_{q+n}(\La)$, 
\begin{equation*}
[A]\wedge [B]=(-1)^{pq}[B]\wedge [A]. 
\end{equation*}
\end{thm}

\begin{proof}
This result can be extracted from the litterature. However, as there are so
many different definitions and sign conventions in the literature, we give
here instead a quick direct proof of the sign-commutativity using our
definition. The associativity and unitality can likewise be proved by lifting the
associativity and unitality of the algebraic intersection product to the
loop space; a proof of associativity of the intersection product that can be lifted to the present situation is given in Proposition~\ref{prop:intalg}. For the unit, a proof for classes represented by manifolds is given in Example~\ref{ex:unit}. 

\medskip

Consider the diagram 
\begin{equation*}
\xymatrix{H_p(\La)\ot H_q(\La) \ar[r]^-{\x} \ar[d]_{t_\ot} & H_{p+q}(\La\x\La)
\ar[r]^-{t^*\!\tau_{\CS}\cap}\ar[d]^t & H_{p+q-n}(U_{\CS}) \ar[r]^-{R_{\CS}} \ar[d]^t
& H_{p+q-n}(\La\x_M\!\La) \ar[r]^-{\concat}\ar[d]^t & H_{p+q-n}(\La)\ar[d]^{id} \\
H_p(\La)\ot H_q(\La) \ar[r]^-{\x} & H_{p+q}(\La\x\La) \ar[r]^-{\tau_{\CS}\cap} &
H_{p+q-n}(U_{\CS}) \ar[r]^-{R_{\CS}} & H_{p+q-n}(\La\x_M\!\La) \ar[r]^-{\concat} & H_{p+q-n}(\La) }
\end{equation*}
where $t_{\ot}$ denotes the twist map that takes $[A]\otimes [B]$ to $%
(-1)^{pq}[B]\otimes [A]$, while the map $t\colon \La\x\La\to \La\x\La$, and its restriction to $U_{\CS}$ and $\La\x_M\La$, take a pair $(\gamma,\delta)$ to $(\delta,\gamma)$.
The first square commutes by the commutativity of the cross product. The second square
commutes by the naturality of the cap product (Appendix~\ref{app:cap}). The
third square commutes up to homotopy on the space level by a homotopy that
slides the basepoint along the geodesic stick added by the retraction.
Finally the last square commutes up to the homotopy on the space level, using that the
concatenation map is homotopic to the map that concatenates at time $s=\frac{%
1}{2}$, and that the identity map on $\Lambda$ is homotopic to the map that
reparametrizes the loops by precomposing with half a rotation of $S^1$. Now $%
t^*\tau_{\CS}=(-1)^n\tau_{\CS}$ because the corresponding fact holds for $%
\tau_M$. Hence the top row is the $(1)^n$ times the product $\wedge_{\Thom}$,
while the bottom row is $\wedge_{\Thom}$. The diagram thus gives that $%
(-1)^{n}[A]\wedge_{\Thom} [B]=(-1)^{pq}[B]\wedge_{\Thom} [A]$. Hence 
\begin{align*}
[A]\wedge [B]=(-1)^{n-np}[A]\wedge_{\Thom} [B]=(-1)^{pq-np}[B]\wedge_{\Thom}
  [A]&=(-1)^{pq-np+n-nq}[B]\wedge [A] \\
  &=(-1)^{(n-p)(n-q)}[B]\wedge [A],
\end{align*}
which gives the result after the given degree shifting by $n$.
\end{proof}

\subsection{Tubular neighborhood of the loop space inside the path space}

To define a tubular neighborhood for the coproduct, we will use a tubular neighborhood of $\Lambda $ inside the
path space $PM$, constructed in this section.

\smallskip

Let 
\begin{equation*}
P M=\{\gamma \colon [0,1]\rightarrow M\ |\ \gamma \text{ is } H^1\}
\end{equation*}
be the space of $H^1$-paths in $M$, with evaluation maps $e_0$, $%
e_{1}\colon PM\rightarrow M$ at $0$ and $1$ so that 
\begin{equation*}
\Lambda =(e_0,e_{1})^{-1}\Delta
\end{equation*}
for $(e_0,e_1)\colon PM\to M^2$. Consider the vector bundle $e_{1}^{\ast}(TM)%
\rightarrow \Lambda $.

\begin{lem}
\label{pathtub} The embedding $\Lambda\subset PM$ admits a tubular
neighborhood in the sense that there is an embedding 
\begin{equation*}
\nu_P\colon \xymatrix{e_{1}^{\ast }(TM)\ \ar[d] \ar@{^(->}[r] & PM \\ \La & }
\end{equation*}
restricting to the inclusion on the zero-section and with image $%
U_P=\{\gamma \in PM\ |\ \gamma(1) -\gamma(0) |<\varepsilon \}$. Moreover,
this tubular neighborhood is compatible with that of the diagonal in $M^2$
in the sense that the following diagram commutes: 
\begin{equation*}
\xymatrix{e_1^*(TM)\ \ \ar@{^(->}[r]^-{\nu_P}\ar[d]_{e_1} &\ 
PM\ar[d]^{(e_0,e_1)}\\ TM\ \ \ar@{^(->}[r]^-{\nu_M}&\  M\x M }
\end{equation*}
where $\nu_M$ is the map (\ref{nuM}).
\end{lem}

This lemma is parallel to Proposition~\ref{TubCS}, and so is its proof. In
the current case though we cannot move the whole loop; we need to keep the
starting point fixed while moving the endpoint.

\begin{proof} Recall again that we identify $TM$ with $TM_\eps$, its subbundle of vectors of length at most $\eps$.  
Given a loop $\gamma \in \Lambda $ and a vector $V\in T_{\gamma_0}M$, we
will define a path $\tau \in PM$. As before, we let $\gamma _{0}
=\gamma(0)=\gamma (1)$ denote the basepoint of $\gamma$, and we write $%
\tau_0=\tau(0)$ and $\tau_1=\tau(1)$ for the endpoints of a path $\tau$.

The map $\nu _{P}$ is defined by $\nu _{P}(\gamma ,V)=\tau$
where 
\begin{equation*}
\tau(t)=h(\gamma _{0},\exp _{\gamma _{0}}(tV))(\gamma (t))
\end{equation*}%
where $h$ is the map of Lemma~\ref{H}. Note that, from the properties of $h$
we have $\tau_0=\gamma _{0}$ and $\tau_1=\exp_{\gamma _{0}}(V)$ hence $\nu_P$ is compatible with $\nu _{M}$. Also it
restricts to the identity on the zero section, i.e.~when $V=0$. The image of 
$\nu _{P}$ lies inside $U_P=\{\tau \in PM\ |\ |\tau _{1}-\tau
_{0}|<\varepsilon \}$, an open neighborhood of $\Lambda$ inside $P$. We will
show that $\nu _{P}$ is a homeomorphism onto $U_{P}$ by defining an inverse.

Define $\kappa _{P}\colon U_{P}\rightarrow e_{1}^{\ast }(TM)$ by $\kappa _{P}(\tau
)=\big(k_{P}(\tau ),v_{P}(\tau )\big)=\big(\gamma,\exp _{\tau _{0}}^{-1}(\tau _{1})\big)$ where 
\begin{align}  \label{equkP}
\gamma (t)&=(h(\tau _{0},\overline{\tau _{0}\tau_{1}}(t)))^{-1}(\tau (t)) 
\end{align}

We have $\gamma(1)=h(\tau _{0},\tau _{1})^{-1}\tau
_{1}=\tau_0=\gamma_0$, and hence $\ga$ is a loop.
The maps $\nu _{P}$ and $\kappa _{P}$ are continuous. Note that 
\begin{equation*}
\overline{\tau _{0}\tau _{1}}(t)=\exp _{\tau _{0}}tv_{P}(\tau )\text{.}
\end{equation*}%
It follows that $\nu _{P}$ and $\kappa _{P}$ inverses of each other which
finishes the proof.
\end{proof}

\begin{rem}
Note that the tubular embedding $\nu_P$ takes constant paths to geodesics:
if $\gamma $ is constant, then $\nu _{P}(\gamma,V) $ is a geodesic.
\end{rem}

Just as in our definitions of the loop product and coproduct, we can give a
``sticky'' version of the retraction map $k_P$ of (\ref{equkP}): Let $U_P\subset PM$ be the $%
\varepsilon$-neighborhood of $\Lambda$ inside $PM$ as above and define 
\begin{equation*}
R_{P}\colon U_{P}\ \longrightarrow \ \ \Lambda
\end{equation*}%
by 
\begin{equation*}
R_{P}(\tau )=\tau \star \overline{\tau _{1}\tau _{0}}
\end{equation*}%
That is, we add a ``stick" $\overline{\tau _{1}\tau _{0}}$ at the end of $%
\tau $ to get back to the starting point $\tau _{0}$. Note that, by
definition, $R_{P}(\tau )=\tau $ if $\tau \in \Lambda $.
Note also that $k_P$ and $R_P$ both keep the startpoint of the paths constant,
i.e.~they define maps over the evaluation at $0$: 
\begin{equation*}
\xymatrix{PM\ar[rr]^-{k_P,R_P} \ar[dr]_{e_0} & & \ \La. \ar[dl]^{e_0}\\ & M&
}
\end{equation*}

\begin{lem}
\label{lem:htpypath} The maps $k_{P},R_{P}\colon U_{P}\longrightarrow \Lambda $
are homotopic, through a homotopy leaving $\Lambda\subset U_P$ fixed at all
time, and such that the homotopy is over the identity on $M$ when evaluating
at $0$, i.e.~the basepoints of the paths stay unchanged throughout the
homotopy.
\end{lem}

\begin{proof}
Define a homotopy $H\colon U_P\times I\to \Lambda$ by 
\begin{equation}  \label{equP}
H(\tau,s)=k_{P}^{s}(\tau )\star \overline{\tau_0\tau_1}\,[0,s]^{-1}
\end{equation}
where 
\begin{equation*}
k_{P}^{s}(\tau )(t)=h\big(\overline{\tau _{0}\tau _{1}}(st),\overline{\tau
_{0}\tau _{1}}(t)\big)^{-1}(\tau (t)).
\end{equation*}
The right side of (\ref{equP}) is the $\star $-concatenation of two paths;
the first, $k_{P}^{s}(\tau )$, begins ($t=0$) at 
\begin{equation*}
h\big(\overline{\tau _{0}\tau _{1}}(0),\overline{\tau _{0}\tau _{1}}(0)\big)%
^{-1}(\tau _{0})=h(\tau _{0},\tau _{0})^{-1}(\tau _{0})=\tau _{0}
\end{equation*}%
and ends ($t=1$) at 
\begin{equation*}
h\big(\overline{\tau _{0}\tau _{1}}(s),\overline{\tau _{0}\tau _{1}}(1)\big)%
^{-1}(\tau_1)=h\big(\overline{\tau _{0}\tau _{1}}(s),\tau _{1})\big)%
^{-1}(\tau _{1})=\overline{\tau _{0}\tau _{1}}(s).
\end{equation*}%
The second is a "stick" beginning at $\overline{\tau _{0}\tau _{1}}(s)$ and
ending at $\tau _{0}$. \ Thus $H(\tau,s)\in \Lambda $ for all $\tau \in
U_{P} $ and $s\in \lbrack 0,1]$. \ When $s=0$ the first path is 
\begin{equation*}
k_{P}^{0}(\tau )(t)=(h(\tau _{0},\overline{\tau _{0}\tau _{1}}(t)))^{-1}\tau
(t)=k_{P}(\tau )(t)
\end{equation*}%
and the second is the constant path at $\tau _{0}$.  Thus 
\begin{equation*}
H(\tau,0)=k_{P}(\tau)
\end{equation*}%
We leave it to the reader to check that 
\begin{equation*}
H(\tau,1)=R_{P}(\tau)
\end{equation*}%
and that $H(-,s)$ fixes $\Lambda$ for all $s$. Moreover, at all time $s$ we
have that $H(\tau,s)(0)=\tau_0$.
\end{proof}

One could also prove that $k_P$ and $R_P$ are deformation retraction of $P$ onto $\La$ relative to the constant loops $M$, but we will not need this statement.

\subsection{Tubular neighborhood for the coproduct and equivalence to
Goresky-Hingston's definition of the homology coproduct and cohomology product}

\label{sec:GHtub}

In this section, we use the tubular neighborhood of the loop space inside the path space to show the equivalence between our definition of the homology coproduct and the one given in \cite{GorHin}. We also separately prove that the dual cohomology product is also equivalent to the cohomology product of \cite{GorHin}, since that product is not directly defined as a dual of the coproduct in that paper.

\smallskip

For the coproduct, the figure-eight space used in the homology product  is replaced by the
space 
\begin{equation*}
\F=\{(\gamma ,s)\in \Lambda \times I\ |\ \gamma (s)=\gamma
(0)\}\subset \Lambda \times I.
\end{equation*}
This space can be defined as a pull-back in the following diagram: 
\begin{equation*}
\xymatrix{ & \F\ \ \ar@{^(->}[r]\ar[d]_{e} & \ \La\x I \ar[d]^{e_I} \\ TM
\ar[r] \ar@{-->}@/_1pc/[rr]_(.4){\nu_M} & M\ \ \ar@{^(->}[r]^\De & \ M\x M }
\end{equation*}%
where $e_I$ as before denotes the map taking $(\gamma ,s)$ to $(\gamma_0,\gamma (s))$, for $\ga_0:=\ga(0)$ as above. Again we can pull-back the $\varepsilon $-tangent bundle
of $M$ along the evaluation at 0 to get a bundle $e^{\ast }(TM)\rightarrow F$%
.

The original construction of the coproduct in \cite{GorHin} used that there
exists a tubular neighborhood of $\F_{(0,1)}:=\F\cap (\Lambda \times (0,1)) $
inside $\Lambda \times (0,1)$. It is argued in \cite{GorHin} that such a
tubular neighborhood exists, but, just as for the Chas-Sullivan product, we
need an explicit construction to be able to compare the resulting coproduct
to the one defined in Section~\ref{sec:newGH}.

\smallskip

Consider the bundle $e^{\ast }(TM)\rightarrow \F_{(0,1)}$.

\begin{prop}
\label{coemb} The embedding $\F_{(0,1)}\subset \Lambda\times (0,1)$ admits a
tubular neighborhood in the sense that there is an embedding 
\begin{equation*}
\nu_{\GH}\colon \xymatrix{e^{\ast }(TM)\ \ar[d] \ar@{^(->}[r] & \Lambda \times
(0,1) \\ \F_{(0,1)} &}
\end{equation*}
that restricts to the inclusion on the zero-section $\F_{(0,1)}$ and is a
homeomorphism onto its image $U_{\GH}=\{(\gamma,s)\in \Lambda\times (0,1)\ |\
|\gamma(0)-\gamma(s)|<\varepsilon \}$. Moreover this tubular neighborhood
restricts to a tubular embedding over each $s\in (0,1)$ and is compatible
with that of the diagonal in $M^2$ in the sense that the following diagram
commutes: 
\begin{equation*}
\xymatrix{e^*(TM)\ \ \ar@{^(->}[r]^{\nu_{\GH}}\ar[d]_{e} &\  \La\x
(0,1)\ar[d]^{e_I}\\ TM\ \ \ar@{^(->}[r]^{\nu_M}&\  M\x M }
\end{equation*}
where $\nu_M$ is the map (\ref{nuM}) and $e_I\colon \Lambda\times (0,1)\to M$
takes $(\gamma,s)$ to $(\gamma_0,\gamma(s))$.
\end{prop}

\begin{proof}
Let $(\gamma ,s,V)\in (e^{\ast }(TM)\rightarrow \F_{(0,1)})$, so $\gamma $
satisfies $\gamma (s)=\gamma_0$ and $V\in T_{\gamma_0}M=T_{\gamma (s)}M$
is a vector with $|V|<\varepsilon $ and $0<s<1$.

Recall from Lemma~\ref{pathtub} the map $\nu _{P}\colon (e_{1}^{\ast
}(TM)\rightarrow \Lambda )\longrightarrow PM$. To define $\nu _{\GH}$, we
apply $\nu _{P}$ to $\gamma \lbrack 0,s]$, the restriction of $\gamma $ to
the interval $[0,s]$ (rescaled to be parametrized again by $[0,1]$), and to $%
\gamma \lbrack s,1]^{-1}$, the path $\gamma \lbrack s,1]$ taken in the
reverse direction: Define $\nu _{\GH}(\gamma ,s,V)=(\tau ,s)$ with 
\begin{equation*}
\tau =\big(\nu _{P}(\gamma \lbrack 0,s],V)\big)\ast _{s}\big(\nu _{P}(\gamma
\lbrack s,1]^{-1},V)\big)^{-1}
\end{equation*}%
and where $\ast _{s}$ indicates that the concatenation takes place at time $%
s $.

Note first that this is a loop, as $\nu _{P}(\gamma \lbrack 0,s],V)$ is a
path starting at $\gamma_0=\gamma_1$ and ending at $\exp _{\gamma (s)}V$%
, just like $\nu _{P}(\gamma \lbrack s,1]^{-1},V)$. As the latter is
reversed one more time, the two paths can be glued together. The glued path $%
\tau $ has $\tau_0=\gamma_0$ and $\tau (s)=\exp _{\gamma (s)}V$. In
particular, $|\tau (s)-\tau_0|<\varepsilon$.

Let $U_{\GH}^{(0,1)}=\{(\tau ,s)\in \Lambda \times (0,1)\ |\ |\tau (s)-\tau_0|<\varepsilon \}$. We define an inverse map 
\begin{equation*}
\kappa _{\GH}\colon U_{\GH}^{(0,1)}\longrightarrow \big(e^{\ast }(TM)\rightarrow
\F_{(0,1)}\big)
\end{equation*}%
by setting $\kappa _{\GH}(\tau ,s)=(\gamma ,s,V)$ with $V=\exp _{\tau_0}^{-1}(\tau(s)) $, and 
\begin{equation}\label{equ:kGH}
\gamma :=k_{\GH}(\tau ,s)=\big(k_{P}(\tau \lbrack 0,s])\big)\ast _{s}\big(%
k_{P}(\tau \lbrack s,1]^{-1})\big)^{-1}
\end{equation}%
where $k_{P}$ is the map of Lemma~\ref{pathtub}. We check that $\gamma $ is
well-defined: $\tau \lbrack 0,s]$ is a path with endpoints at distance at
most $\varepsilon $, so it is in the source of the generalized map $k_{P}$.
Likewise for $\tau \lbrack s,1]^{-1}$. Moreover, the image of those paths
under $k_{P}$ are loops based at $\gamma_0=\gamma_1$. Hence it makes
sense to concatenate them, yielding a loop based at that point, now
parametrized by $[0,1]$ and with a self-intersection at time $s$.

The two maps $\nu_{\GH}$ and $\kappa_{\GH}$ are continuous and are
inverse for each other because $\nu_P$ and $\kappa_P$ are inverse of each
other.
\end{proof}

\begin{rem}
\label{rem:J} The constructions on $\F_{(0,1)}$ could also be done on 
\begin{equation*}
\F_{\frac{1}{2}}=\{\gamma \in \Lambda \ |\ \gamma (\frac{1}{2})=\gamma_0\}
\end{equation*}
and then expanded to the remainder of $\F_{(0,1)}$ using the natural map $\F_{s}\overset{\approx }{\rightarrow }\F_{\frac{1}{2}}$ induced by
precomposing with the piecewise linear bijection $\theta _{\frac{1}{2}%
\rightarrow s}\colon [0,1]\rightarrow [0,1]$ that takes $0\rightarrow 0$, $%
\frac{1}{2}\rightarrow s$, $1\rightarrow 1$. These maps appear in the
definition of the cohomology product in \cite[Sec 9.1]{GorHin}. Note
that when $s=0$ we have a map but not a bijection $\F_{0}\rightarrow \F_{\frac{%
1}{2}}$; the map $\nu _{\GH}$ also does not extend to $s=0,1$.
\end{rem}

The map 
\begin{equation*}
k_{\GH}\colon U_{\GH}^{(0,1)}\longrightarrow \F_{(0,1)}
\end{equation*}
of (\ref{equ:kGH}) in the proof above is a 
retraction that preserves $s\in (0,1)$. We will now show that it is
homotopic to our stick map $R_{\GH}$ from Section~\ref{sec:newGH} restricted
to $U_{\GH}^{(0,1)}=U_{\GH}\cap \Lambda\times (0,1)$, with both maps defining deformation retractions.

\begin{lem}
\label{lem:GHhtpy} The maps $k_{\GH},R_{\GH}\colon U_{\GH}^{(0,1)}\longrightarrow
\F_{(0,1)}$ are homotopic. Moreover, $k_{\GH}$ defines a deformation retration of $U_{\GH}^{(0,1)}$ onto $\F_{(0,1)}$, and 
$R_{\GH}\colon U_{\GH}\longrightarrow \F$   a deformation retraction of $U_{\GH}$ onto $\F$. 
\end{lem}

\begin{proof}
The first part of the statement follows from applying Lemma~\ref{lem:htpypath} on both sides of the
concatenation at $s$, keeping $s$ constant. Because the homotopy in  Lemma~\ref{lem:htpypath} fixes $\La$ and $s$ is kept constant, the resulting homotopy will fix $\F_{(0,1)}$. Hence for the second part of the statement it is enough to show that $R_{\GH}$ is a deformation retraction, keeping $s$ fixed.

So we are left to find a homotopy $G\colon U_{\GH}\x I\to U_{\GH}$ between $j\circ R_{\GH}$ and the identity that restricts to the identity on $\F$ and keeps $s$ constant, where $j\colon \F\to U_{\GH}$ denotes the inclusion.  Such a homotopy can be obtained by retracting the geodesic sticks added by $R_{\GH}$ when $(\ga,s)\notin\F$, just as in the case of $R_{\CS}$ in Lemma~\ref{retract}.  
\end{proof}

\begin{prop}
\label{prop:coequiv} The tubular neighborhood $\nu_{\GH}$ allows us to give
an explicit description of the Goresky-Hingston coproduct \cite[8.4]{GorHin}, and the coproduct $\vee$ of Definition~\ref{def:stickyco} is a chain model
for it.
\end{prop}

\begin{proof}
The definition of the Goresky-Hingston coproduct is made most precise in the
proof of Lemma~8.3 of the paper \cite{GorHin}, and we recall it here. The
coproduct is defined via a sequence of maps 
\begin{equation*}
\begin{aligned} H_*(\Lambda,M) &\longrightarrow H_{*+1}(\Lambda\times I,\B) \longrightarrow
H_{*+1-n}(\F,\B)  \longrightarrow H_{*+n-1}(\Lambda\times \Lambda,M\times \Lambda\cup
\Lambda \times M) \end{aligned}
\end{equation*}
with $\B=M\times I\cup \Lambda\times \{0,1\}$ as before, where the first map crosses with $I$ and the last map is the cut-map, just
as in our definition of the coproduct $\vee$ (Definition~\ref{def:stickyco}%
). The middle map needs to be constructed, and the paper \cite{GorHin}
produces it as follows.

For a small $\alpha>0$, let 
\begin{equation*}
V_\alpha=\Lambda^{<\alpha}\times I\ \cup\ \Lambda\times \big([0,\alpha)\cup
(1-\alpha,1]\big),
\end{equation*}
where $\Lambda^{<\alpha}$ is the space of loops of energy smaller than $%
\alpha^2$. Let $V_\alpha^{(0,1)}=V_\alpha\cap \big(\Lambda\times (0,1)\big)$%
. A tubular neighborhood of $\F_{(0,1)}$ inside $\Lambda\times (0,1)$
together with excision allows to define maps 
\begin{equation*}
\xymatrix{ H_*(\La\x (0,1),V_\alpha^{(0,1)}) \ar[d]_\cong\ar[r]^-{[\tau\cap]}_-\cong &
H_{*-n}(\F_{(0,1)},\F_{(0,1)}\cap V^{(0,1)}_\alpha) \ar[d]^\cong\\ H_*(\La\x
I,V_\alpha) & H_{*-n}(\F_{[0,1]},\F_{[0,1]}\cap V_\alpha)}
\end{equation*}
where the horizontal map is defined by capping with the Thom class and
retracting (or equivalently applying the Thom collapse map followed by the
Thom isomorphism, see the proof of Proposition~\ref{prop:CSequiv}). Then it is argued in \cite{GorHin}
that taking a limit for $\alpha\to 0$ gives a well-defined map 
\begin{equation*}
H_*(\Lambda\times I,\B)\longrightarrow H_{*-n}(\F_{[0,1]},\B),
\end{equation*}
which finishes the definition of the coproduct.

We have here given an explicit tubular neighborhood of $\F_{(0,1)}$ inside $%
\Lambda\times (0,1)$ and Lemma~\ref{lem:GHhtpy} shows that the retraction
map $k_{\GH}$ associated to this tubular neighborhood is homotopic to the map 
$R_{\GH}$ used in our definition of the coproduct, and the Thom classes are compatible, in both cases coming from the Thom class $\tau_M$ of the tangent bundle. We thus have chain
homotopic maps 
\begin{equation*}
\xymatrix{C_*(\La\x (0,1),V_\alpha^0) \ar[r]^-{\tau_{\GH}\cap}&
C_{*-n}(U_{\GH}\cap \La\x (0,1),U_{\GH}\cap V_\alpha^0)
\ar@<-3ex>[d]_{k_{\GH}}\ar@<3ex>[d]^{R_{\GH}} \\ &
C_{*-n}(\F_{(0,1)},\F_{(0,1)}\cap V^0_\alpha). }
\end{equation*}
Now the definition of $R_{\GH}$ does not need a tubular neighborhood and is
well-defined also when $\alpha=0$. This provides a different proof that the
limit as $\alpha$ tends to 0 exists, already on the chain level using the
map $R_{\GH}$. At the same time it shows that our definition of the coproduct
agrees in homology with that of \cite{GorHin}.
\end{proof}

The cohomology product is defined in \cite{GorHin} using a slightly
different sequence of maps than the dual of the above maps, to avoid
complications with taking limits. As we do not need to take limits with our
new definition, it is now easier to show that this other construction is
indeed the dual of the coproduct $\vee$.

\begin{thm}
\label{thm:cocoequ} The cohomology product 
\begin{equation*}
\circledast_{\Thom}\colon H^p(\Lambda,M)\otimes H^q(\Lambda,M)\to H^{p+q-1+n}(\Lambda,M)
\end{equation*}
induced in cohomology by the map $\oast_{\Thom}$ of Definition~\ref{def:cast}
agrees with the cohomology product defined in \cite[Sec 9]{GorHin} up to the sign $(-1)^{q(n-1)}$.
\end{thm}

\begin{proof}
In \cite{GorHin}, the cohomology product is  defined as $(-1)^{q(n-1)}\widetilde\vee^*(x\times y)$, for a map $\widetilde \vee$ which
we recall below. We will show that $\vee_{\Thom}$ and $\widetilde \vee$ induce the
same map on cohomology.

Recall from Remark~\ref{rem:J} the reparametrization map $\theta _{\frac{1}{2%
}\rightarrow s}$, which is defined for all $s\in [0,1]$. This can be used to
define a map $J\colon \Lambda\times I\to \Lambda$ by $J(\gamma,s)=\gamma\circ
\theta _{\frac{1}{2}\rightarrow s}$. Note that $J$ restricts to a map $%
J\colon \F\to \F_{\frac{1}{2}}$ and also commutes with the evaluation
at 0 and defines a map of pairs: 
\begin{equation*}
\xymatrix{(\La\x I,\La\x \del I \cup M\x I) \ar[rr]^-J \ar[dr]_{e} & & (\La, M*_{\frac{1}{2}}\!\La\cup \La*_{\frac{1}{2}}\!M) \ar[dl]^{e}\\ & M & }
\end{equation*}
where $ M*_{\frac{1}{2}}\!\La\cup \La*_{\frac{1}{2}}\!M$ is the subspace of $\Lambda$ of
loops that are constant on their first or second half. Now consider the
diagram 
\begin{equation*}
\xymatrix{H^*(\La,M) \ar@{=}[r] & H^*(\La,M) \\ H^{*+1}(\La\x I,\B) \ar[u]_{(\x I)^*} \ar@{=}[r] & H^{*+1}(\La\x I,\B) \ar[u]^{(\x I)^*}\\
  H^{*+1}(\La,  M*_{\frac{1}{2}}\!\La\cup \La*_{\frac{1}{2}}\!M)
\ar[u]_-{J^*} & \\ H^{*+1-n}(U_{\frac{1}{2}}, M*_{\frac{1}{2}}\!\La\cup \La*_{\frac{1}{2}}\!M) \ar[u] _{e^*\tau_M\cup} \ar[r]_-{J^*} &
H^{*+1-n}(U_{\GH},\B) \ar[uu]^{\tau_{\GH}\cup}\\
H^{*+1-n}(\F_{\frac{1}{2}},  M*_{\frac{1}{2}}\!\La\cup \La*_{\frac{1}{2}}\!M)
\ar[u]_{k_{\GH}^*} \ar[r]_-{J^*}& H^{*+1-n}(\F ,\B)
\ar[u]^{R_{\GH}^*}\\ H^{*+1-n}(\La\x\La,M\x \La\cup \La\x M) \ar[u]_{\cut^*}
\ar@{=}[r] & H^{*+1-n}(\La\x\La,M\x \La\cup \La\x M) \ar[u]^{\cut^*} }
\end{equation*}
where $U_{\frac{1}{2}}:=U_{\GH}\cap \Lambda\times\{\frac{1}{2}\}$, and $%
k_{\GH} $ is the restriction to $U_{\frac{1}{2}}$ of the retraction map of
the same name defined above. The right vertical composition is our
definition of $\vee^*_{\Thom}$, while the left composition is an explicit
description of the map $\widetilde \vee^*$ defining the cohomology product
in \cite{GorHin} (see Figure 3 in Section 9 of that paper, where the sign $%
(-1)^{q(n-1)}$ is encoded in the letter $\omega$). So we are left to show
that this diagram commutes. The first square trivially commutes. The second
square commutes by naturality of the cap product (see Appendix~\ref{app:cap}%
) using that $\tau_{\GH}=J^*e^*\tau_M$ by the commutativity of the triangle
above. The third square commutes up to homotopy on the space level by Lemma~%
\ref{lem:GHhtpy}. Finally the last square commutes on the space level.

This shows that our definition is equivalent to that of \cite{GorHin}. Now
statement (a) follows from \cite[Prop 9.2]{GorHin} and statement (b) from
our definition~\ref{def:castalg}.
\end{proof}

\subsection{The algebraic homology coproduct and cohomology product}

In this section, we give a detailed proof of the algebraic properties
of our chosen sign convention for the homology coproduct and the
cohomology product, because it differs from what it suggested in
\cite{GorHin}.

\begin{thm}
  \label{thm:coGHalg} The algebraic homology coproduct
  $$\vee\colon
H_{k-n}(\Lambda,M)\to  \oplus_{p+q=k+1}H_{p-n}(\Lambda,M)\ot H_{q-n}(\Lambda,M)$$ 
is the twisted suspension of a graded coassociative and cocommutative coproduct on the degree-shifted homology $H_{*-n}(\La,M)$.
Dually, the algebraic cohomology product
$$\circledast\colon H^{p-n}(\La,M)\ot H^{q-n}(\La,M)\rar H^{(p+q-1)-n}(\La,M)$$
is the twisted desuspension of a graded associative and commutative product  on the degree-shifted cohomology $H^{*-n}(\La,M)$. Explicitly, 
for $[x]\in H^{p-n}(\La,M), [y]\in H^{q-n}(\La,M), [z]\in H^{r-n}(\La,M)$
$$[x]\oast [y]=(-1)^{pq+1}[y]\oast [x] \ \ \ \textrm{and}\ \ ([x]\oast [y])\oast [z]=(-1)^{r+1}[x]\oast ([y]\oast [z]).$$
\end{thm}

In the statment, the {\em twisted suspension} refers to the suspension of an operad by the homology of the one-point compactification of the simplex operad, as described eg.~in \cite[Sec. 4.1]{Kla13B}. This twisted suspension can be thought of as adding a length parameter to the inputs of the operations (or outputs in the case of a coproduct), with the lengths summing to 1.  The symmetric group action, that governs the commutativity rules for the operations, permutes the lengths, which gives a ``twisting'' in the suspension. These lengths in the case of our coproduct are the parameters $(s,1-s)$. The fact that the symmetry switches these parameters corresponds to a flip of the interval $I$ parametrizing the coproduct, and explains the ``$+1$'' in the sign $(-1)^{pq+1}$ in the graded commutativity. (We will also see this flip of the interval in the proof below.) 
And the fact that it is the suspension of a (co)associative operation gives the sign $(-1)^{r+1}$ in the associativity relation: if $\vee A=(s\ot s)\vee_0 (s^{-1}A)$ for $\vee_0$ a degree 0 coproduct satisfying that $(1\ot \vee_0)\circ \vee_0=(\vee_0\ot 1)\circ \vee_0$, one get such a sign for $\vee$ with our convention that $\vee$ acts from the right. (The twisting is not visible in the associativity relation.)

Recall that the cohomology product is not  counital, as can be computed in the case of spheres \cite[15.3]{GorHin}; morally, the only reasonable counit would be the empty set of loops, which is not part of our set-up. 

\begin{rem}[Compatibility with known results in Hochschild homology]\label{rem:HH}
The coproduct is an operation of degree $1-n$, and the statement above treats these two degree shifts separately, with $n$ becoming a shift of the homology and 1 a shift of the operation. This is necessary because the two shifts have different behavior. One way to understand why this could be expected is to look at the Hochschild homology model of the loop space.

The Jones isomorphism \cite{Jon87}  states that, over fields,  if $M$ is simply-connected, 
\begin{equation}\label{equ:Jones}
  H^*(LM)\cong HH_*(C^*(M),C^*(M)).
\end{equation}
Moreover, over $\mathbb{Q}$, by \cite{LamSta} 
there exists a  model for $C^*(M)$ as a strict Poincar\'e duality cdga (or {\em Frobenius} cdga). 
Algebraic versions of the Chas-Sullivan and Goresky-Hingston products have been constructed by a number of authors using this  Hochschild model of the loop space, see eg.~\cite{Abb16,Kla13B,Tradler}, and by \cite[Thm 1]{FelTho}  and \cite[Thm 1.3]{NaeWil}, we know that they are compatible with the operations defined here under the isomorphism (\ref{equ:Jones}).
Note that $C^*(M)$ is equivalent to a Poincar\'e duality algebra of {\em dimension $n$}. As explained in \cite[6.3]{WahWes}, the dimension shift in the Poincar\'e structure corresponds to a determinant bundle twisting on the prop of operations acting on the Hochschild complex. This determinant twisting can instead be described as a degree shift of the algebra ($H^*(LM)$ in our case) if we restrict to a product (with degree $n$ desuspension) or a coproduct (degree $n$ suspension). 

Klamt in \cite{Kla13B} describes operations on the Hochschild homology of Frobenius cdgas, including the coproduct as a degree $1$ operation. She shows in Section 4.1 of the paper that the coproduct is such a twisted suspension of a degree 0 coproduct.  This is exactly compatible with our sign computation.

Our result is also compatible with the work of Rivera-Wang \cite{RivWan}, that shows, using completely different methods, that the cohomology product is the $(n-1)$-suspension of a degree 0 associative product, in a situation where there is no commutativity. Indeed, without shifting the degree, the sign we obtain in the associativity relation is $(-1)^{r+n-1}$, which is also the sign that would arise from shifting an associative operation by degree $(n-1)$. 
\end{rem}

\begin{proof} 
  We first consider the commutativity relation. The proof given here is essentially the same as that of \cite[Prop 9.2]{GorHin}, just formulated for the coproduct and using our definition. We give it for convenience as it is short, and because it exhibits that the interval $I$ parametrizing the coproduct is flipped.

 Let $\chi\colon \La\x I\to \La\x I$ be the map defined by $\chi(\ga,s)=(\ga(s+\_\,),1-s)$, rotating the loop $\ga$ by $s$ and flipping the interval. This map has the effect of exchanging  $\ga[0,s]$ and $\ga[s,1]$ together with their parametrizing lengths $s$ and $1-s$.  It restricts to a map  $\chi\colon U_{\GH}\to U_{\GH}$ and $\chi\colon \F\to \F$, and fits into the following commuting diagrams: 
 $$\xymatrix{\La\x I \ar[d]_{e_I}\ar[r]^-\chi & \La\x I \ar[d]^{e_I} & & & \F  \ar[d]_{\cut}\ar[r]^-\chi & \F \ar[d]^{\cut}  \\
   M\x M \ar[r]^-t & M\x M  & & & \La\x \La \ar[r]^-t & \La\x \La }$$
 for $t\colon X\x Y \to Y\x X$ the twist map.  
 The map $\chi\colon \La\x I\to\La\x I$ is homotopic to $\id\x (-1)$, for $(-1)$ the flip map on the interval: a homotopy $H\colon (\La\x I)\x I\to \La\x I$ can be given by setting $H(\ga,s,r)=(\ga(rs+\_\,),1-s)$.  This homotopy restricts to a homotopy $\chi\simeq_{H}\id\x\{-1\}\colon  \B\x I\to \B$, for $\B=\La\x\del I\cup M\x I$ as it preserves the constant loops and exchanges $s=0$ and $s=1$ for all $r$. 
   (The homotopy $H$ does not restrict to either $U_{\GH}$ or $\F$, but we will not need that.) 
 This shows that 
 $$\chi_*=(\id,(-1))_*\colon H_*(\La\x I,\B) \rar H_*(\La\x I,\B).$$

 We are now ready to compute commutativity of the coproduct, using the following diagram:
  $$\xymatrix{C_*(\La) \ar[r]^-{\x I}\ar[d]_{\id} & C_{*}(\La\x I)\ar[rr]^{[\chi^*\tau_{\GH}\cap]} \ar[d]^{\chi}&& C_{*}(U)\ar[d]^{\chi}\ar[r]^-{R} & C_{*}(\F)\ar[d]^{\chi} \ar[r]^-{\cut} &
    C_{*}(\La\x\La) \ar[d]^t \ar@<1ex>[r]^-{\AW} &\ar@<1ex>[l]  \oplus C_{p}(\La)\ot C_{q}(\La)\ar[d]^{t_\ot}\\
    C_*(\La) \ar[r]^-{\x I} & C_{*}(\La\x I)\ar[rr]^{[\tau_{\GH}\cap]} && C_{*}(U) \ar[r]^-{R} & C_{*}(\F) \ar[r]^-{\cut} & C_{*}(\La\x\La) \ar@<-1ex>[r]_-{\AW} & \ar@<-1ex>[l]  \oplus C_{p}(\La) \ot C_{q}(\La).}$$
  Taking the chains relative to $M$, $\B$ or $M\x \La\cup \La\x M$ as appropriate, we consider the induced diagram in homology.  The first square commutes  up to the sign $(-1)$ by the above computation of $\chi$ on $H_*(\La\x I,\B)$. The second square commutes by naturality of the cap product. The third square commutes because $R\circ \chi\simeq \chi\circ R$ relative to $\B$:
  we have
\begin{align*}
  \chi\circ R(\ga,s)& = (\overline{\ga_0\ga(s)}\st \ga[s,1] *_{1-s}\ga[0,s]\st \overline{\ga(s)\ga_0}\ , \ 1-s) \\
  R \circ \chi(\ga,s)& = (\ga[s,1]\st \overline{\ga_0\ga(s)} *_{1-s} \overline{\ga(s)\ga_0}\st \ga[0,s]\ , \ 1-s) 
\end{align*}
with the first loop based at $\ga_0$ and the second based at $\ga(s)$. A homotopy can be obtained by sliding the basepoint along the geodesic $\overline{\ga_0\ga(s)}$.  
 The fourth square commutes on the space level as noted above. The map $t_\ot$ takes $A\ot B$ to $(-1)^{pq}B\ot A$ and  the last square commutes in homology. Now  $\chi^*\tau_{\GH}=\chi^*e_I^*\tau_M=e_I^*t^*\tau_M=(-1)^n\tau_{\GH}$ as $t^*$ changes $\tau_M$ by the sign $(-1)^n$.

So, for $A\in C_k(\La,M)$, the diagram gives us that  $$\vee_{\Thom} A=\sum_{p+q=k+1-n}A^0_{p}\ot A^1_{q}\simeq (-1)^{1+n}\sum_{p+q=k+1-n}(-1)^{pq}A^1_{q}\ot A^0_{p}.$$
The same diagram in cohomology gives for $[x]\in H^p(\La,M), [y]\in H^q(\La,M)$
\begin{align*}
  [x]\!\oast\![y]=(-1)^{n+np}\vee_{\Thom}^* ([x]\!\ot\![y])  =(-1)^{1+pq+np} \vee_{\Thom}^* ([y]\!\ot\! [x])  &=(-1)^{1+pq+np+n+nq} \vee_{\Thom}^* ([y]\!\ot\! [x]) \\
  & =(-1)^{(p+n)(q+n)+1}[y]\!\oast\! [x]
  \end{align*}
as stated in the theorem (after regrading). 
  
\medskip

The associativity relation corresponds to the two ways of successively splitting loops at two consecutive intersection points, as parametrized by the 2-simplex
$$\De^2=\{0\le s_1\le s_2\le 1\}.$$
We will use the two evaluation maps
$$e_1,e_2\colon \La\x \De^2 \ \rar\ M\x M$$
defined by $e_i(\ga,s_1,s_2)=(\ga(0),\ga(s_i))$, and the corresponding subspaces  
\begin{align*}
  U_{1}&=\{(\ga,s_1,s_2)\in \La\x \De^2 \ |\ |\ga(0)-\ga(s_1)|<\eps\}= e_1^{-1}(U_M)\\
  U_{2}&=\{(\ga,s_1,s_2)\in \La\x \De^2 \ |\ |\ga(0)-\ga(s_2)|<\eps\}=e_2^{-1}(U_M)\\
   U_{12}&=\{(\ga,s_1,s_2)\in \La\x \De^2 \ |\ |\ga(0)-\ga(s_1)|<\eps\ \textrm{and}\ |\ga(0)-\ga(s_2)|<\eps\}=U_1\cap U_2.\\
\end{align*}

For associativity, consider the following diagram, where we have marked in bold the squares where signs come in, and where we use the following abbreviated notations:
we write $\tau=\tau_{\GH}=e_I^*\tau_M$,  the map $cR=\cut\circ R\colon U\to \La^2$, 
with its twisted version $$cR^t=(1\x t)\circ (cR\x 1)\circ (1\x t)\colon U_{2}\to \La\x I \x \La$$ that takes $(\ga,s_1,s_2)$ to $(\ga'[0,s_2],s_1,\ga'[s_2,1])$, restricting to a map $cR^t\colon U_{12}\to U \x \La$. 
$$%
  \xymatrix @C=1pc {    C_*(\La) \ar[dd]_{\x I} \ar[r]^-{\x I} \ar@{}[ddr]|{\mbox{{\bf (0)}}}& C_*(\La\x I) \ar[d]^-{\x I} \ar[r]^-{\tau\cap }\ar@{}[dr]|{\mbox{(1)}} &  C_*(U) \ar[d]^{\x I} \ar[r]^-{cR} & C_*(\La^2) \ar@/^8ex/[dd] \ar[d]^{\x I} \ar[r]  \ar@{}[ddr]|{\mbox{ \ \ {\bf (10)}}}& C_*(\La)\ot C_{*}(\La) \ar@<1ex>[l] \ar[dd]^{(1)^r(\x I_1\ot 1)} \\
  & C_*(\La\x I^2) \ar@{}[dr]|{\mbox{(2)}} \ar[d]^{p\circ (1\x t)}\ar[r]^-{(\tau\x 1)\cap } &  C_*(U\x I) \ar[d]^{p\circ(1\x t)} \ar[r]^-{cR\x 1}  \ar@{}[dr]|{\mbox{(3)}}& C_*(\La^2\x I) \ar[d]^{1\x t}  & \\
    C_*(\La\x I) \ar[r]^-{\x_I \De^2}\ar[d]_{\tau\cap} \ar@{}[dr]|{\mbox{(4)}}& C_*(\La\x \De^2) \ar[d]_{e_1^*\tau_M\cap} \ar@{}[dr]|{\mbox{{\bf (5)}}}  
    \ar[r]^-{e_2^*\tau_M\cap } &  C_*(U_2) \ar@{}[dr]|{\mbox{\ \ \ \ \ (6)}}\ar[d]^{e_1^*\tau_M\cap} \ar[r]^-{cR^t}& C_*(\La\x I\x \La) \ar[r]\ar[d]^{(\tau\x 1)\cap} \ar@{}[dr]|{\mbox{(11)}} &  \ar@<1ex>[l] C_*(\La\x I)\ot C_{*}(\La) \ar[d]^{\tau\cap\ot 1}\\
C_*(U) \ar[d]_{cR} \ar[r]^-{\x_I\De^2}& C_*(U_1)  \ar@{}[dr]|{\mbox{(7)}}\ar[r]^(.5){e_2^*\tau_M\cap} \ar[d]^{cR\x 1} & C_*(U_{12}) \ar[d]_{cR\x 1} \ar[r]^-{cR^t} \ar@{}[dr]|{\mbox{(8)}} & C_*(U\x \La) \ar[r] \ar[d]^{cR\x 1}& \ar@<-1ex>[l]  C_*(U)\ot C_{*}(\La) \ar[d]^{cR\ot 1} \\
C_*(\La^2)  \ar[r]^-{\x I} \ar[d]& C_*(\La^2\x I) \ar@{}[dr]|{\mbox{{\bf (9)}}}\ \ar[r]_-{(1\x \tau_{})\cap}\ar[d]& C_*(\La\x U)  \ar[r]_{1\x cR}\ar[d]& C_*(\La^3)   \ar[d] \ar[r]&  C_*(\La^2) \ot C_{*}(\La) \ar[d]\\
C_{*}(\La)^{\ot 2}  \ar[r]_-{1\ot \x I} \ar@<-1ex>[u]& C_{*}(\La)\ot C_*(\La\x I)  \ar[r]_-{1\ot \tau_{}\cap} \ar@<1ex>[u]& C_{*}(\La)\ot C_*(U)  \ar[r]_{1\ot cR}\ar@<1ex>[u]& C_{*}(\La)\ot C_*(\La^2)  \ar[r]_-{} & C_{*}(\La)^{\ot 3}
}$$ 
Considering the diagram starts in degree $p+q+r-2+n$ and ends in total degree $p+q+r$, where we will write $C_*(\La)^{\ot 3}=\bigoplus_{p,q,r} C_p(\La)\ot C_q(\La)\ot C_r(\La)$.
\begin{enumerate}[(a)]
\item Diagram (0) commutes up to the sign $(-1)$ as both compositions cross with $\De^2$, but with opposite orientations.
  
\item Diagrams (1), (4) and (11) commute by (\ref{equ:capcross}), using that $e_1^*\tau_M=\tau\x 1$ for (4),  while diagram (9) commutes up to the sign $(-1)^{np}$ by the same equation if we write the target
$$C_*(\La)\ot C_*(U)=\bigoplus_{*=p+q+r}C_p(\La)\ot C_{q+r}(U).$$

\item Diagrams (2),(6), (7) commute by naturality of the cap product. For diagram (2) this follows from the fact that $(1\x t)^*e_2^*\tau_M=e_1^*\tau_M=\tau\x 1$,
  for diagram (6) because $cR^t\circ (e_I\x 1)=e_1$, and diagram (7) because $(cR\x 1)^*(1\x \tau)=e_2^*\tau_M$.

\item Diagrams (3) and (8) commute on the space level by definition of the map $cR^t$. 
  
\item Diagram (5) commutes up to $(-1)^n$ because of the compatibility with cup and cap and the fact that we commute two degree $n$ classes.

\item Diagram (10) commutes 
  if we write the target as $$C_*(\La\x I)\ot C_*(\La)=\bigoplus_{*=p+q+r-1+n}C_{p+q-1+n}(\La\x I)\ot C_r(\La),$$
 with the sign $(-1)^r$ already added to the rightmost vertical map.  
Indeed, 
starting from the top right corner, we see that a tensor product $A\ot B$ is mapped to the class $A\x I\x B$ following either sides of the diagram, with the orientation differs by that sign, as the interval is crossed with either $A$ or $A\x B$. 
\end{enumerate}

Reading the outside of the diagram,  we thus get that
$$\xymatrix{C_*(\La) \ar[rr]^-{\vee_{\Thom}} \ar[d]_{\vee_{\Thom}} && \bigoplus  C_{p+q-1+n}(\La)\ot C_{r}(\La) \ar[d]^-{(-1)^{r}\vee_{\Thom}\ot 1}\\
\bigoplus C_{p}(\La)\ot C_{q+r-1+n}(\La) \ar[rr]^{1\ot \vee_{\Thom}}  &&  \bigoplus C_{p}(\La)\ot C_{q}(\La) \ot C_{r}(\La)
}$$
commutes up to the sign  $(-1)^{1+n+np}$.
If we replace $\vee_{\Thom}$ by $\vee$ in the diagram, and $(-1)^{r}$ by $(-1)^{r+n}$ to take into account the degree shift, we get a diagram commuting up to the sign
 $$(-1)^{1+np}(-1)^{(n+np)+(n+nq)}(-1)^{(n+n(p+q-1+n))+(n+np)}=-1$$
 where the first sign comes from the sign commutativity of the corresponding square for $\vee_{\Thom}$, taking into account to the sign $(-1)^{r+n}$ placed on the right vertical arrow, the second from the deviation of using $\vee$ instead of $\vee_{\Thom}$ along the bottom of the diagram and the third for that deviation going along the bottom of the diagram. 
In cohomology, this gives the sign stated in the theorem. 
\end{proof}

\begin{rem}\label{rem:GHsign}
 The algebraic cohomology product of the paper \cite{GorHin}
  $$\circledast_{\GH}\colon H^{p}(\Lambda ,M)\otimes H^{q}(\Lambda ,M)\rightarrow H^{p+q+n-1}(\Lambda ,M)$$
is defined by 
\begin{equation*}
[x]\circledast_{\GH} [y]=(-1)^{(n-1)q}[x]\circledast_{\Thom} [y]
\end{equation*}
(see \cite[9.1]{GorHin} and Theorem~\ref{thm:cocoequ}). 
Inputing this sign change in the above computation, gives that this product satisfies the  graded commutativity and associativity relations 
$$[x]\circledast_{\GH} [y]=(-1)^{(p-n+1)(q-n+1)}[y]\circledast_{\GH} [x]  \ \ \ \textrm{and}\ \ \ ([x]\circledast_{\GH} [y])\circledast_{\GH} [z]=(-1)^{n(p+r)}[x]\circledast_{\GH}([y]\circledast_{\GH} [z])$$
if $[x]\in H^p(\La,M)$, $[y]\in H^q(\La,M)$, and $[y]\in H^r(\La,M)$. 
\end{rem}

\section{Computing the product and coproduct by intersecting chains}\label{sec:computing}

The original idea of Chas and Sullivan was that, given two nice enough cycles $A,B\in C_*(\La)$, we can define a product $A\wedge B$ by taking the concatenation of the loops in $A$ and $B$ on the locus where their basepoints agree, i.e. after taking the intersection product of $e_*(A)$ and $e_*(B)$. In this section, we make precise that this is indeed what the product $\wedge$ does on classes that evaluate at their basepoint to transverse classes in $M$, and we give the corresponding property for the coproduct $\vee$.
For the coproduct, this property is a little more delicate to state and prove because all cycles admit trivial self-intersections, so we will first need to make precise how the trivial self-intersections do not count.

\subsection{Computing the homology product}

Recall that two smooth maps $f\colon X\to M$ and $g\colon Y\to M$ are transverse if for each $x,y$ with $f(x)=p=g(y)$, we have that $f_*T_xX+g_*T_yY=T_pM$. This is equivalent to the transversality of the maps $f\x g\colon X\x Y\to M\x M$ and $\De\colon M\to M\x M$. Note also that if $f$ and $g$ are embeddings, then $(f\x g)^{-1}(\De M)\cong f(X)\cap g(Y)$. 

The following result is a slight generalization of  \cite[Prop 5.5]{GorHin} where the cycles were assumed to be embedded submanifolds of the loop space, and the sign was not computed: 
\begin{prop}\label{prop:transCS}
  Let $Z_1\colon \Si_1\to \La$ and $Z_2\colon \Si_2\to \La$ be cycles respresented by closed, oriented manifolds $\Si_1$ and $\Si_2$, with the property that $e\circ Z_1\colon \Si_1\to M$ and $e\circ Z_2\colon \Si_2\to M$ are transverse smooth maps. 
  Then 
  $$[Z_1]\wedge_{\Thom} [Z_2]=(Z_1\st Z_2)|_{\Si_1\x_e \Si_2} \ \in \ H_*(\La)$$
  is represented by the cycle $Z_1\wedge_{\Thom} Z_2\colon  \Si_1\x_e\Si_2\to \La$ taking $(x,y)$ to the concatenation $Z_1(x)\st Z_2(y)$, where $\Si_1\x_e \Si_2:= (e\!\circ\! Z_1\x e\!\circ\! Z_2)^{-1}(\De M)\subset \Si_1\x \Si_2$ is oriented so that the isomorphism 
  $$T(\Si_1\x \Si_2)|_{\Si_1\x_e \Si_2}\cong (e\!\circ\! Z_1\x e\!\circ\! Z_2)^*N\De(M) \op T(\Si_1\x_e \Si_2)$$ respects the orientation, with $N\De M$ oriented as in Remark~\ref{rem:orientations}. 
\end{prop}

\begin{rem}\label{rem:transCS}
Because the product only cares about the parts of the cycles mapped to the neighborhood $U_{\CS}$ of the figure 8 space $\La\x_M\La$, the assumption can be weakened, without any further work, to requiring only that the restriction of  $e\circ Z_1\x e\circ Z_{2}$ to $(Z_1\x Z_2)^{-1}(U_{\CS})= (e\circ Z_1\x e\circ Z_{2})^{-1}(U_{M})\subset (\Si_1\x \Si_2) $ is transverse to the diagonal in $M\x M$.  
\end{rem}

  \begin{ex}[The unit for $\wedge$] \label{ex:unit}
    Suppose that $A\colon \Si_A\to \La$ is a $p$-cycle  with the property that $e\circ A\colon \Si_A\to M$ is an embedded submanifold, and let $[M]\colon M\inc \La$ denote the cycle of all constant loops. Then $e\circ A$ and  $e\circ [M]$ are trivially transverse in $M$, intersecting in $\Si_A$,  and the proposition verifies that $[M]\wedge_{\Thom}[A]=[A]=(-1)^{n-np}[A]\wedge_{\Thom}[M]$,
    as the isomorphism $T(M\x \Si_A)\cong (1\x e\circ Z_A)^*N\De(M) \op T\Si_A$ is orientation preserving with our choice of orientation for $N\De M$ of Remark~\ref{rem:orientations}, while
     the isomorphism $T(\Si_A\x M)\cong (e\circ Z_A\x 1)^*N\De(M) \op T\Si_A$ only holds up to the sign $(-1)^{n-np}$. Note that the signs work out in such a way that $[M]\wedge [A]=[A]=[A]\wedge [M]$, as expected. 
    \end{ex}

  \begin{proof}[Proof of Proposition~\ref{prop:transCS}]
    Consider the following diagram:
    $$\xymatrix{& H_{p+q}(\Si_1\x \Si_2)\ar[r]^-{Z_1\x Z_2}\ar[d]_{[(Z_1\x Z_2)^*\tau_{\CS}\cap]} & H_{p+q}(\La\x \La,U_{\CS}^c)\ar[d]^{[\tau_{\CS}\cap]} \\
      H_{p+q-n}(\Si_1\x_e\Si_2) \ar[r]& H_{p+q-n}((Z_1\x Z_2)^{-1}(U_{\CS})) \ar[r]^-{Z_1\x Z_2} & H_{p+q-n}(U_{\CS})\ar[d]^{\concat\, \circ\, R_{\CS}} \\
      & & H_{p+q-n}(\La)}$$
    The square commutes by naturality of the cap product. Now $(Z_1\x Z_2)^*\tau_{\CS}=((Z_1\circ e)\x (Z_2\circ e))^*\tau_{M}$ and by the transversality assumption, this is a Thom class for the embedding $\Si_1\x_e \Si_2\inc \Si_1\x \Si_2$. Hence the left vertical map takes $[\Si_1\x \Si_2]$ to $\Si_1\x_e \Si_2$, with the given orientation by our conventions. The result follows, noting also that $R_{\CS}$ is the identity on pairs of loops that already intersect at their basepoints. 
   \end{proof}

   Let $\Omega M=\maps_*(S^1,M)\xrightarrow{\iota} \La$ denote the based loop space of $M$. 
The following, probably well-known fact, follows from Proposition~\ref{prop:transCS}, using also Remark~\ref{rem:transCS}. 

\begin{cor}
If $\dim(M)\neq 0$,   the Chas-Sullivan product is trivial on the homology of the based loop space $\iota_*H_*(\Omega M)\le H_*(\La)$. 
\end{cor}

\begin{proof}
  Let $Z_1\colon \Si_1\to \Omega M\to \La$ and $Z_2\colon \Si_2\to \Omega M\to \La$ be cycles in the image of the based loops in $\La$. Let $p$ be the basepoint and $q\neq p$ some other point in $M$, in the same path component.
  Replace $Z_2$ by the homologous cycle $Z_2'\colon \Si_2\to \La$ obtained from $Z_2$ by setting $Z_2'(x)=\delta*_{\frac{1}{3}} Z_2(x) *_{\frac{2}{3}} \overline{\delta}$ for $\delta$ a path from $q$ to $p$.
Take $\eps<|p-q|$. Then  $(Z_1\x Z_2')^{-1}(U_{\CS})= (e\circ Z_1\x e\circ Z_2')^{-1}(U_{M})=\emptyset$ as $e\circ Z_1(\Si_1)=\{p\}$ and $e\circ Z_2'(\Si_2)=\{q\}$ with $(p,q)\notin U_M$ by construction. It follows from Proposition~\ref{prop:transCS} and  Remark~\ref{rem:transCS} that $[Z_1]\wedge [Z_2]=[Z_1]\wedge [Z_2']=0$. 
  \end{proof}

\subsection{Computing the homology coproduct}
   
The coproduct is a relative operation, and to prove the analogous result in that case, we will first show that the coproduct can equivalently be defined as relative to a slightly larger subspace than the one used in the original definition. As above, we write $\B:=\La\x \del I\cup M\x I$.
We consider the following two variants of $\B$:  
\begin{align*}
   \Bb &:=\{(\ga,s)\in \La\x I \ |\ \ga(t)=\ga(0) \ \ \forall t\in [0,s] \ \textrm{or}\ \forall t\in [s,1]\}\\
  \B_\eps &:=\{(\ga,s)\in \La\x I \ |\ |\ga(t)-\ga(0)|<\eps \ \ \forall t\in [0,s] \ \textrm{or}\ \forall t\in [s,1]\}
\end{align*}
Note that there are inclusions 
$$\B\inc \Bb \inc \B_\eps.$$
The space $\Bb$ identifies with the inverse image of $M\x \La\cup \La\x M$ under the cut map, and $\B_\eps$ is an open neighborhood of $\Bb$: 
$$\xymatrix{&\F\ar[rr]^-{\cut} & & \La\x \La \\
&    \B_\eps \ar@{^(->}[u] \ar[rr]^-{\cut} && \La_\eps\x \La \cup \La\x \La_\eps \ar@{^(->}[u]\\
\B\ \ar@{^(->}[r]  & \Bb \ar@{^(->}[u] \ar[rr]^-{\cut} && M\x \La \cup \La\x M. \ar@{^(->}[u]
 }$$
 where $\La_\eps=\{\ga\in \La \ |\ |\ga(t)-\ga(0)|<\eps \ \textrm{for all}\ t\in [0,1]\}$ is an open neighborhood of the constant loops in $\La$.

It follows from Lemma~\ref{lem:GHhtpy} that the inclusion  $(\F,\B)\inc (U_{\GH},\B)$ is a relative homotopy equivalence, with homotopy inverse given by the retraction maps. We show now that we can replace $\B$ by $\Bb$ and $\B_\eps$ in the following way
\begin{lem}\label{lem:hompairs}
Let $\eps\le \eps_1<\frac{\rho}{3}$.  The inclusion   $i\colon (\F,\Bb)\inc (U_{\GH,\eps},\B_{\eps_1}\cap U_{\GH,\eps})$ is a homotopy equivalence of pairs. 
\end{lem}

\begin{proof}
For $\mu\in [0,1]$, let  $K=K_\mu\colon M\times M\rightarrow M$  be a smooth map that satisfies 
$$  
K_{p}(q)= K_{\mu,p}(q)= \left\{ \begin{array}{ll} 
                     p  & \text{ if }|p-q|<\eps_1 \mu  \\
                     q  & \text{ if }|p-q|\geq 2\eps_1 \mu 
                   \end{array}\right.  \ \ \ \textrm{and}\ \ \   |K_p(q)-p| \le |p-q|  $$
 that is $K_p$ is a map that does not increase the distance to $p$, collapses the ball of radius
 $\eps_1\mu$ around $p$ to $p$ and fixes anything away
 from the ball of radius $2\eps_1\mu$ around $p$. As $\eps_1$ is small with respect to the injectivity radius, such a map exists because it exists in $\RR^n$.
Note that when $\mu=0$, $K_p=K_{0,p}\colon M\to M$ is the identity. 

Now define $Q_\mu\colon U_{\GH,\eps}\rightarrow U_{\GH,\eps}$  
by setting 
\[
Q_\mu(\gamma,s)=(K_{\mu,\gamma (0)}\circ \gamma,s). 
\]
This is well-defined as $|\gamma(s)-\gamma(0)|<\eps$ implies that the same holds for $K_{\mu,\gamma (0)}\circ \gamma$ by assumption on $K_{\mu,\ga(0)}$. For the same reason, it restricts to a map  $Q_\mu\colon \B_{\eps_1}\to \B_{\eps_1}$. It also restricts to maps  
$Q_\mu\colon \F\rightarrow \F$ and $Q_\mu\colon \Bb\to \Bb$ because $\gamma(t)=\gamma(0)$ gives that the same holds for $K_{\mu,\gamma (0)}\circ \gamma$. 
Also $Q_0=\id$ while 
$Q_1$ retracts every part of $\gamma $ at distance at most $\eps_1$  of $\ga(0)$ to $\ga(0)$ itself. In particular, 
$$Q_1\colon (U_{\GH,\eps},\B_{\eps_1})\rar (\F,\Bb)$$
Varying the parameter $\mu\in [0,1]$ gives homotopies $i\circ Q_1\simeq \id$ and  $Q_1\circ i\simeq \id$,  considering $Q_\mu$ as a 
self-map of $U_{\GH,\eps}$ in the first case, and of $\F$ in the second, noting that it restricts as required to $\B_{\eps_1}$ and $\Bb$. 
\end{proof}

If follows from the lemma that $\B_{\eps_1}\cap U_{\GH,\eps}\simeq \Bb$, and in particular that $\B_\eps\simeq \Bb$. 
It not clear that these two spaces should be homotopy equivalent to $\B$.

\begin{prop}\label{prop:copeps}
Let $\eps_0<\eps\le \eps_1<\frac{\rho}{3}$. The coproduct $\vee_{\Thom}$ can be defined using the sequence of maps 
\begin{multline*}H_*(\La,M) \xrightarrow{\x I}  H_{*+1}(\La \x I, \B_{\eps_1}\cup U_{\GH,\eps_0}^c) \xrightarrow{\tau_{\GH} \cap}  H_{*+1-n}(U_{\GH,\eps},\B_{\eps_1}\cap U_{\GH,\eps})   \\
 \xrightarrow{Q_1} H_{*+1-n}(\F,\Bb) \xrightarrow{\cut} H_{*+1-n}(\La\x \La,M\x \La\cup \La\x M)
\end{multline*}
for $\Bb,\B_{\eps_1}$ as defined above and $Q_1\colon U_{\GH}\to \F$ as in the proof of Lemma~\ref{lem:hompairs} a homotopy inverse for the inclusion. 
\end{prop} 

\begin{proof}
The statement follows from the commutativity of the following diagram
$$\xymatrix{H_*(\La,M) \ar[r]^-{\x I} & 
H_{*+1}(\La \x I, \B\cup U_{\GH,\eps_0}^c) \ar[d]^-{\tau_{\GH} \cap} \ar[r] &H_{*+1}(\La \x I, \B_{\eps_1}\cup U_{\GH,\eps_0}^c) \ar@<2ex>[d]^-{\tau_{\GH} \cap}  \\
& H_{*+1-n}(U_{\GH,\eps},\B) \ar[r]  \ar[d]_R^\cong & H_{*+1-n}(U_{\GH,\eps},\B_{\eps_1}\cap U_{\GH,\eps})   \ar@<2ex>[d]^{Q_1}_\cong \\
& H_{*+1-n}(\F,\B)  \ar@<-3ex>[u] \ar[dr]_-{\cut} \ar[r] & H_{*+1-n}(\F,\Bb)  \ar@<2ex>[d]^{\cut}  \ar@<1ex>[u] \\
 & &H_{*+1-n}(\La\x \La,\B),}$$
where $Q_1$ is defined in the proof of Lemma~\ref{lem:hompairs} above. 
The first square commutes by naturality of the cap product, the second because it commutes on the level of spaces for the inclusions, that are homotopy inverses for $R$ and $Q_1$ by Lemma~\ref{lem:GHhtpy}  and Lemma~\ref{lem:hompairs}, and the triangle commutes on the space level. 
\end{proof}

We will use the above description of the coproduct to give a geometric way to compute it on ``nice'' cycles, where nice here will mean that they are modeled by manifolds that, after applying the evaluation map and removing the trivial self-intersections, are transverse to the diagonal in $M\x M$.

\begin{prop}\label{prop:strong4}  
  Let $Z\colon (\Sigma,\Si_0) \rightarrow  (\Lambda,M)$ be a relative $k-$cycle represented by  an oriented manifold pair $(\Si,\Si_0)$, and write $\Si_\B:=\Si_0\x I \cup \Si\x \del I$. 
Suppose that the restriction 
\begin{align*}
E(Z):=e_I\circ (Z\times I)|_{(\Si\x I)\minus \Si_\B } \colon (\Si\x I)\minus \Si_\B & \ \longrightarrow\ M\times M\\
(\sigma ,t) &\ \ \mapsto\ \; \big(Z(\sigma )(0),Z(\sigma)(t)\big)
\end{align*}
 is a smooth map transverse to the diagonal $\Delta\colon M\to M\x M$. 
 Then 
\begin{equation*}
[\vee Z]=[\cut\circ (Z\times I)|_{\overline{\Si_{\De}}}]\in
H_{k+1-n}(\Lambda\times \Lambda ,M\times \Lambda \cup \Lambda\times M)
\end{equation*}
for $\overline{\Si_\De}$ the closure inside $\Si\x I$ of
$$\Si_\Delta:=E(Z)^{-1}(\De M)\ \subset \ (\Si\x I)\minus \Si_\B$$ 
oriented so that the isomorphism $T_{(\s,t)}(\Si\x I)\cong N_{E(Z)(\s,t)}\De(M)\oplus T_{(\s,t)}\Si_\De$ respects the orientation. 
\end{prop}

\begin{rem}
(1) Note that the map $\overline{E}(Z)=e_I\circ (Z\times I) \colon \Si\x I \longrightarrow M\times M$ 
will essentially never be transverse to the diagonal as $\Si_\B$ is by definition mapped to  $\B=M\x I\cup \La \x \del I$ by $Z\x I$, and hence entirely to the diagonal of $M$ by $\overline{E}(Z)$.
Likewise,  $\overline{E}(Z)^{-1}(\De M)$ 
will typically not be a submanifold of $\Si\x I$.

\smallskip

\noindent
(2) Just as for the product, one can without any further work weaken the assumption to only require that  $(Z\x I)^{-1}(U_{\GH,\eps}) \minus \Si_{\eps_1}=K \minus \Si_{\eps_1}$ for $K\subset \Si\x I$  a manifold, neighborhood retract, intersecting transversally $\De M$ under the map $e_I\circ (Z\x I)$, where $\eps\le \eps_1<\frac{\rho}{3}$, that is we can throw away anything that is not in the inverse image of $U_{\GH}\minus \B_{\eps_1}$.
\end{rem}

\begin{proof}[Proof of Proposition~\ref{prop:strong4}] 
We compute the coproduct using the composition given by Proposition~\ref{prop:copeps}. Write $\Si_{\B_\eps}=(Z\x I)^{-1}(\B_\eps)$. 
Let $\eps_0<\eps$ and consider the diagram \\
$\xymatrix{
&  H_{k+1}(\Si\x I,\Si_{\B_\eps}) 
  \ar[r]^-{Z\x I} & H_{k+1}(\La\x I,\B_\eps \cup U_{\GH,\eps_0}^c) \ar@<10ex>@/^10ex/[ddd]_-{[\tau_{\GH}\cap]} \\
&  H_{k+1}(\Si\x I\minus \Si_\B,\Si_{\B_\eps}\minus \Si_\B) \ar[u]^\cong 
  \ar[r]^-{Z\x I} \ar[u]^\cong \ar[d]_-{(Z\x I)^*\tau_{\GH}\cap}
 &   H_{k+1}(\La\x I\minus \B,(\!\B_\eps\minus \B\!) \cup U_{\GH,\eps_0}^c) \ar[u]_\cong \ar[d]_-{[\tau_{\GH}\cap]} \\
H_{k+1-n}(\Si_\De,\Si_\De\cap \Si_{\B_\eps}) \ar[d]_\cong \ar[r] & H_{k+1-n}(\Si_{U_\eps}\minus \Si_\B,\Si_{\B_\eps}\minus \Si_\B )\ar[d]_\cong \ar[r]^-{Z\x I} & H_{k+1-n}(U_{\GH}\minus \B,\B_\eps\minus \B)  \ar[d]^\cong\\
H_{k+1-n}(\overline{\Si_\De},\overline{\Si_\De}\cap \Si_{\B_\eps}) \ar[r] & H_{k+1-n}(\Si_{U_\eps},\Si_{\B_\eps} )\ar[r]^-{Z\x I} & H_{k+1-n}(U_{\GH},\B_\eps) \ar@<2ex>[d]^{Q_1}_\cong \\
H_{k+1-n}(\overline{\Si_\De},\del \overline{\Si_\De}) \ar[rr]\ar[u] & & H_{k+1-n}(\F,\B)\ar@<1ex>[u]
}$
where we apply excision for the triples $(\B\subset \B_\eps\subset U_{\GH}=U_{\GH,\eps})$ and $(\B\cup U_{\GH}^c\subset \B_\eps\cup U_{\GH,\eps_0}^c\subset \La\x I)$ and their inverse images under $Z\x I$. 
Now $(Z\x I)^*\tau_{\GH}=(Z\x I)^*e_I^*\tau_{M}$ which, by our transversality assumption on $Z\x I$, is a Thom class from the embedding $\Si_\De\subset (\Si\x I)\minus \Si_\B$. 
Hence $(Z\x I)^*e_I^*\tau_{M}\cap [\Si\x I\minus \Si_\B]=[\Si_\De]$, oriented as in the statement by our conventions.
The result then follows from the commutativity of the diagram. 
\end{proof}

We will apply the proposition to do several computations below, see in particular Examples~\ref{ex:genus3} and \ref{ex:square}, and Proposition~\ref{prop:spheres1} below. We start the next section by a generalization of the simplest case of the proposition, namely the case where $\Si_\De$ is empty.

\subsection{Support of the homology product and cohomology product} 
\label{sec:supportNEW}

We will show in this section that the coproduct $\vee$ vanishes on the simple loops, and more generally that the
iterated coproduct $$\vee^{k}=(1\x\dots\x 1\x \vee)\circ\dots\circ (1\x\vee)\circ \vee\colon  H_*(\La,M)\to H_*(\La^{k+1},\cup_i \La^i\x M\x \La^{k-i})$$
vanishes on loops having at most $k$-fold self-intersections. As a consequence, we are guaranteed to find loops
with multiple self-intersections when a power of the coproduct is not $0$.
Examples will follow. We start by defining what we mean by multiple
self-intersections.

\begin{Def}
\label{def:kfold} We will say that a (non-constant) loop $\gamma \in \Lambda 
$ has a \textit{k-fold intersection at }$p\in M$ if $\gamma ^{-1}\{p\}$
consists of $k$ points.
\end{Def}

\begin{figure}[hb]
\includegraphics[width=10.4cm]{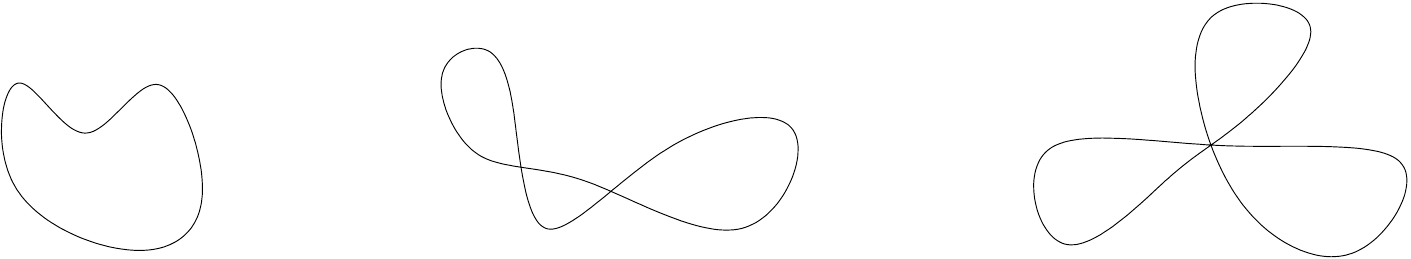}
\caption{Simple loop, loop with two 2-fold intersections and loop with a
3-fold intersection}
\end{figure}

Note that if $\gamma$ is parametrized proportional to arc length (see
Section~\ref{H1loopssec}) then $\gamma \in \Lambda $ has a $k$-fold
intersection at the basepoint if and only if it is the concatenation of $k$
nontrivial loops (and no more), so roughly speaking 
\begin{equation*}
k\text{-fold intersection }\leftrightsquigarrow \text{ \ }k\text{ loops.}
\end{equation*}
If a loop is stationary at some point $p$, it will have an infinite self-intersection at that point. Note though that for non-constant loops, such self-intersections can be removed by working instead with the homotopic subspace of loops parametrized proportional to arc length. 

\bigskip Recall the operator 
\begin{equation*}
\Delta \colon H_{\ast }(\Lambda )\longrightarrow H_{\ast +1}(\Lambda )
\end{equation*}%
coming from the $S^{1}$-action on $\Lambda $ that rotates the loops.

\begin{thm}
  \label{thm:coproduct2} 
  The coproduct $\vee$ and its iterates $\vee^k$ have the following support property:

\begin{enumerate}[(i)]

\item If $Z\in C_{\ast }(\Lambda,M)$ has the property that every nonconstant
loop in the image of $Z$ has at most $k$-fold intersections, then%
\begin{equation*}
\vee^{k}[Z]=\vee^{k}(\Delta[Z])=0\in H_*(\Lambda^k,\cup_i \La^i\x M\x \La^{k-i}).
\end{equation*}

\item If $Z\in C_{\ast }(\Lambda,M)$ is a cycle with the property that every
nonconstant loop in the image of $Z$ has at most a $k$-fold intersection at
the basepoint, then%
\begin{equation*}
\vee^{k}[Z]=0\in H_*(\Lambda^k,\cup_i \La^i\x M\x \La^{k-i}).
\end{equation*}
\end{enumerate}
\end{thm}

\begin{rem}
As an immediate consequence of (ii) we have the following (stronger)
statement:

\begin{enumerate}
\item[(ii')] Suppose $Z\in C_{\ast }(\Lambda,M)$ is a cycle with the property
that for \textit{some} $s\in S^{1}$, every nonconstant loop $\gamma $ in the
image of $Z$ has at most a $k$-fold intersection $\gamma (s)$. Then%
\begin{equation*}
  \vee^{k}[Z]=0 \in H_*(\Lambda^k,\cup_i \La^i\x M\x \La^{k-i}).
\end{equation*}
\end{enumerate}

On the chain level the coproduct $\vee$ sees only
self-intersections at the basepoint; $\vee\circ \Delta $ picks up
on other self-intersections. See Example~\ref{ex:genus3} below. 

The equations in Theorem~\ref{thm:coproduct2} are valid in $H_{\ast
}(\Lambda^k,\cup_i \La^i\x M\x \La^{k-i})$ and therefore, when applicable, after the
K\"unneth map in $H_{\ast }(\Lambda,M)^{\otimes k}$. But they are
valid for any choice of coefficients if we do not use the K{\"u}nneth formula.
\end{rem}

\begin{proof}[Proof of Theorem~\ref{thm:coproduct2}]
We start with the case $k=1$, that is assuming that loops in the image of $Z\colon (\Si,\Si_0)\to (\La,M)$ only have trivial self-intersections at their basepoint, or in other words
  $(Z\x I)^{-1}(\F)=(Z\x I)^{-1}(\B)=\Si_0\x I\cup \Si\x \del I$.
  The statement will follow from Proposition~\ref{prop:copeps} if we can find $\tilde\eps\le \eps\le  \eps_1<\frac{\rho}{3}$ such that $Z\x I^{-1}(U_{\GH,\tilde\eps})\subset \Si_{\B_{\eps_1}}:=Z\x I^{-1}(\B_{\eps_1})$.
  
We start by considering $\B_\eps$.   Because $\B_{\eps}$ is open, and $\Si\x I$ compact, the space $(Z\x I)^{-1}(\B_{\eps}^c)\subset \Si\x I$ is compact.  Let
  $$\eps'=\min\{|\ga(s)-\ga(0)|\ |\ (\s,s)\in (Z\x I)^{-1}(\B_{\eps}^c), Z(\s)=\ga\}.$$
  We have $\eps'>0$ by our assumption on $Z$: if $\ga(s)=\ga(0)$ with
  $(\ga,s)$ in the image of $\B_\eps^c$, we cannot have $s=0,1$ or $\gamma$ constant, so $\ga$ must have a non-trivial self-intersection at time $s$, contradicting the assumption that it is simple if not constant. 
  Let $\tilde\eps=\min(\eps,\eps')$ and $\eps_1=\eps$. Then $Z\x I^{-1}(U_{\GH,\tilde \eps})\subset \Si_{\B_{\eps_1}}$ by our choice of $\tilde\eps\le \eps_1$: if $|\ga(s)-\ga(0)|<\tilde \eps$, then $(Z\x I)(\ga,s)\notin \B_\eps^c=\B_{\eps_1}^c$.
  Using the definiton of $\vee$ in  Proposition~\ref{prop:copeps}, we get that $[\vee Z]=0$.  
  This proves (ii) in the case $k=1$.

  Statement $(i)$ in the case $k=1$ follows from statement $(ii)$ as, if $Z$
consists only of simple and trivial loops, then both $Z$ and $\Delta Z$
consists only of trivial loops and those with no self-intersection at the
basepoint. This completes the proof in the case $k=1$.

\medskip

For $k>1$, let
$$\B_\eps^k=\{(\ga,\underline s)\in \La\x \De^k\ |\ |\ga(t)-\ga(0)|<\eps \textrm{ for all } t\in [s_i,s_{i+1}] \textrm{ for some } 0\le i\le k\}$$
where $\underline s=(0=s_0\le s_1\le \dots\le s_k\le s_{k+1}=1)$. Just like in the case $k=1$, we have  $\B_\eps^k$ is open, and hence $(Z\x \De^k)^{-1}((\B_{\eps}^k)^c)\subset \Si\x \De^k$ is compact.
Generalizing the case $k=1$, let
$$\eps'=\min \big\{\max_{0\le i\le k}\{|\ga(s_i)-\ga(0)|\} \ |\ (\s,\underline s)\in (Z\x \De^k)^{-1}((\B_{\eps}^k)^c), Z(\s)=\ga\big\}.$$
We have $\eps'>0$ by our assumption on $Z$ since $\eps'=0$ would imply the existence of a non-constant loop $\ga$ in the image of $Z$ having an at least $k$-fold self-intersection at $0$.  
This $\eps'$ has the property that if $\ga=Z(\s)$ is in the image of $Z$ and $(Z\x \De^k)(\s,\U s)\notin \B_\eps^k$, then there exists an $1\le i\le k$ such that $|\ga(s_i)-\ga(0)|\ge \eps'\ge \tilde \eps$, for $\tilde\eps=\min(\eps,\eps')$ as above.

As in the proof of associativity of the coproduct (Theorem~\ref{thm:coGHalg}), the iterated coproduct $\vee^k$ can be computed by first crossing with $\De^k$, then capping with the class
$\tau_{1\dots k}:=e_k^*\tau_M\cup \dots\cup e^*_1\tau_M$, for $e_i\colon \La\x \De^k\to M\x M$ the evaluation at $(0,s_i)$, and derefter applying the retraction maps. The target of the cap product $[\tau_{1\dots k}\cap ]$ is the space
$$U_{1\dots k,\tilde \eps}=\{(\ga,\U s) \ | \ |\ga(s_i)-\ga(0)|<\tilde \eps \ \textrm{for all} \ 1\le i\le k\}.$$
Now if  $(\ga,\U s)$ is in the image of the cycle $[\tau_{1\dots k}\cap (Z\x \De^k)]$, then $\ga$ was in the image of $Z$ and $(\ga,\U s)$ satisfies that $|\ga(s_i)-\ga(0)|<\tilde \eps$ for all $i$. Hence $(\ga,\U s)$ has image in $\B^k_\eps$ by our choice of $\tilde\eps$, and computing the coproduct using $\B_\eps$ (as in Proposition~\ref{prop:copeps}) yields that the coproduct is trivial, proving (ii) in the general case.

Just as in the case $k=1$, statement (i) follows from (ii) as if $Z$ consists only of loops that are either constant or have at most $k$-fold intersections, then both $Z$ and $\De Z$ consists of  only of loops that are either constant or have at most $k$-fold intersections at the basepoint. 
  \end{proof}

\begin{ex}
  \label{ex:genus3}
  Here is an example that illustrates the consequences of
Theorem~\ref{thm:coproduct2}. Let $M$ be a surface of genus 3 and
let $\gamma$ be the loop pictured in Figure~\ref{fig:genus3}, starting
at the intersection point $p$, and let $X=\{\gamma \}\in C_{0}(\Lambda,M)$.
Then $\gamma $ has a 2-fold intersection at the basepoint. Despite this, $%
\vee [X]=\widehat{\vee }[X]=0$ for (at least) two reasons: (i) $\widehat{%
\vee }[X]$ has degree $-1$ and (ii) $X$ is homologous to $Y=\{\delta \}\in
C_{0}(\Lambda,M)$ where $\delta $ traces out the same path as $\gamma $, but
starts at the point $q$. Since $\delta$ has no self-intersection at the
basepoint, $\widehat{\vee }[X]=\widehat{\vee }[Y]=0$ by Theorem~\ref%
{thm:coproduct2}.

\begin{figure}[h]
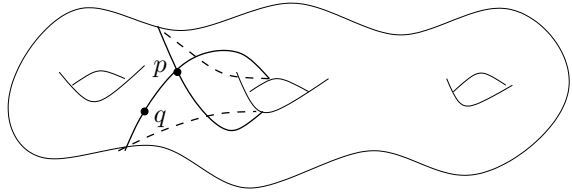

\centering
\begin{lpic}{genus3loop(0.55,0.55)}
\lbl[b]{37,27;$p$}
\lbl[b]{37,15;$q$}
\end{lpic}
\caption{Loop in a genus 3 surface}
\label{fig:genus3}
\end{figure}

Now consider the cycle $\Delta Y\in C_{1}(\Lambda,M)$. Suppose $p=\delta
(t_{0})=\delta (t_{1})$ with $0<t_{0}<t_{1}<1$. To compute $\widehat{\vee }%
\Delta Y$, we use Proposition~\ref{prop:strong4} with $\Si=S^1$ and $\Si_0=\emptyset$. Consider 
\begin{equation*}
\begin{array}{ccccl}
S^1\times I & \xrightarrow{\Delta Y\times I} & \Lambda \times I & 
\overset{e_I}{\longrightarrow} & M\times M \\ 
(t,s) & \mapsto & (\delta _{t},s) & \mapsto & (\delta _{t}(0),\delta
_{t}(s))=(\delta (t),\delta (t+s))%
\end{array}%
\end{equation*}
where $\delta_t$ is the loop defined by $\delta _{t}(s)=\delta (t+s)$. On
the \textit{interior} of $S^1\times I$, the image of this map intersects the
diagonal $\Delta \subset M\times M$ exactly twice, at the image of the
points $(t,s)=(t_{0},t_{1}-t_{0})$ and $(t,s)=(t_{1},t_{0}-t_{1}+1)$. 
One checks 
that $Z\x I(S^1\x I)$ is transverse to $\De M$ if $\delta $ self-intersects transversely at $p$, which we may assume. One can also check that the orientation of $T(S^1\x I)$ at the two points is opposite. (Note that this sign computation agrees with the computation of the Thom sign in the commutativity relation in the proof of Theorem~\ref{thm:coGHalg}, when in this case $n=2$ and $p=0=q$ so that $(-1)^{1+n+pq}=-1$.) 
Hence we can apply Proposition~\ref{prop:strong4} with $\Si_\De=\overline{\Si_\De}=\{(t_{0},t_{1}-t_{0}),(t_{1},t_{0}-t_{1}+1)\}\subset \Si=S^1\x I$. 
Let%
$$\begin{array}{rcccl}
\zeta &=&\delta [t_{0},t_{1}]&=&\ga[0,t_1-t_0] \\
\xi &=&\delta [t_{1},t_{0}+1]&=&\ga[t_1-t_0,1]
\end{array}$$
be the two loops that make up $\ga$. Then the proposition gives that  the coproduct is 
\begin{equation*}
\vee \Delta Y=\{\pm(\zeta ,\xi ),\mp(\xi ,\zeta )\}\in C_{0}(\Lambda \times
\Lambda,M\times \Lambda \cup \Lambda \times M),
\end{equation*}
where the actual sign depends on choices of
orientation of $\delta$ and of the surface. The homology class $[\vee \Delta Y]$ 
is nontrivial because $(\zeta,\xi)$ and $(\xi ,\zeta )$ belong to different
components of $\Lambda \times \Lambda$, both disconnected from $M\times \Lambda \cup \Lambda \times M$.

To sum up: we can perturb away the self intersection of $\gamma$ at the
basepoint, from which it follows that $\vee[X]=0$. But  $\vee [\Delta X]=\vee[\Delta Y]\neq 0$, and from this we can
conclude that every representative $Z$ of $[\Delta X]$ contains a
nonconstant loop with nontrivial self-intersection.

Note that in order to compute $[\vee \Delta Y]$, we only looked at intersections in the \textit{interior }of $S^1\times I$%
, and the fact that the boundary $S^1\x \del I$ maps to the diagonal played no role. 
This is exactly the main content of  Proposition~\ref{prop:strong4}. 
\end{ex}

\begin{ex}[Square of a loop in a surface]\label{ex:square}
Note that the same computation gives that, for $\ga$ a simple loop in a surface, 
$$\vee([\ga\st \ga])=0=\vee \Delta ([\ga\st \ga]).$$
Indeed, one can deform such a loop $\ga\st \ga$ so that it has exactly one transverse self-intersection as it lives in an orientable surface by assumption. The two resulting terms in the coproduct will cancel by the above computation (or by the sign of the commutativity relation of $\vee_{\Thom}$ as given in the proof of Theorem~\ref{thm:coGHalg}). 
\end{ex}

For $k=1$, Theorem~\ref{thm:coproduct2} says that $\vee[Z]=\vee\circ \Delta [Z]=0$ for any homology class $[Z]$ that can
be represented using simple and trivial loops only, and that $\vee[Z]=0$ for any $[Z]$ that has a representative consisting entirely of
trivial loops and loops that have no self-intersection at the basepoint
(that is, no self-intersections of order $k>1$ at the basepoint). This
support statement confirms our intuition that the coproduct

\begin{center}
``\emph{looks for self-intersections, and cuts them apart}".
\end{center}

From the point of view of geometry, if there are no self-intersections, the
coproduct should vanish, and the coproduct is not foolish enough
to mistake the tautology $\gamma (0)=\gamma (0)$  (the equation you get when
you set $\gamma (0)=\gamma (s)$ and let $s\rightarrow 0$) for a
self-intersection! For example, the union of ``all circles, great and small"
on a sphere, where a \textit{circle} on the unit sphere $S^{n}\subset 
\mathbb{R}^{n+1}$ is by definition the non-empty intersection of $S^{n}$
with a $2$-plane in $\mathbb{R}^{n+1}$, defines a non-trivial homology class
in $H_{3n-2}(\Lambda S^n)$. We can use our theorem to conclude that the
coproduct vanishes on this homology class, even though it includes the
constant loops.

\begin{rem}
  A $(k+1)$-fold self-intersection in a loop may seem as a rather unlikely event. However, it is not so unlikely in families of loops. In fact, such a $(k+1)$-fold self-intersection is generically avoidable precisely when the iterated coproduct $\vee^k$ anyway vanishes for degree reasons, i.e. in degrees $<k(n-1)$.
  Indeed, if $\Sigma \subset \Lambda$ is a cycle represented by a smooth submanifold,
then the map 
\begin{eqnarray*}
\Sigma \times \Delta^{k} &\rightarrow &M^{k+1} \\
(\gamma ,s_{1},...,s_{k}) &\rightarrow &(\gamma (0),\gamma
(s_{1}),...,\gamma (s_{k}))
\end{eqnarray*}%
will generically not intersect the diagonal $\Delta_{k+1}M\subset M^{k+1}$
(signalling that $(k+1)$-fold intersections at the basepoint can be avoided) precisely when 
\begin{equation*}
\dim (\Sigma \times \Delta^{k})+\dim M< \dim M^{k+1}, 
\end{equation*}%
that is when $\dim \Sigma < k(n-1)$. For $(k+1)$-self-intersection somewhere along the loops, this condition becomes $\dim \Sigma < k(n-1)-1$, which is also when the operation $\vee\De$ vanishes for degree reasons. 

For example in dimension $n=2$, with $k=1$ and dim $\Sigma =0$: if a loop
intersects itself transversely at the basepoint we can by perturbation
move the self-intersection away from the basepoint, as $0<k(n-1)$, but we cannot
remove the self-intersection entirely, given that $0\not < k(n-1)-1$. 
\end{rem}

Finally we give the dual statement to Theorem~\ref{thm:coproduct2}, where we write $[x_{1}]\circledast [x_{2}]{\circledast}\dots {\circledast }[x_{k}]$ for any bracketing of this product.

\begin{thm}
\label{thm:product2} The product $\circledast$ and its iterates 
$\oast^k$ have the following support property:

\begin{enumerate}[(i)]

\item If $[Z]\in H_{\ast }(\Lambda,M )$ has a representative with the property
that every nonconstant loop has at most $k$-fold intersections, and if $%
[x_{1}],\dots,[x_{k}]\in H^{\ast }(\Lambda,M )$, then 
\begin{equation*}
\langle [x_{1}]\circledast [x_{2}]{\circledast}\dots {\circledast }[x_{k}],Z\rangle=0\text{ and }\langle [x_{1}]\circledast[x_{2}]{\circledast } \dots {\circledast }%
[x_{k}],\Delta [Z]\rangle=0.
\end{equation*}

\item If $[Z]\in H_{\ast }(\Lambda,M)$ has a representative with the property
that every nonconstant loop has at most a $k$-fold intersection at the
basepoint, and if $[x_{1}],\dots,[x_{k}]\in H^{\ast }(\Lambda,M)$, then 
\begin{equation*}
\langle [x_{1}]{\circledast }[x_{2}]{\circledast } \dots {\circledast }[x_{k}],[Z]\rangle=0.
\end{equation*}
\end{enumerate}
\end{thm}

\begin{rem}
There is another, well-known "dual" statement to Theorem~\ref%
{thm:coproduct2}, which is concerned with the Chas-Sullivan product
instead, and is a consequence of Proposition~\ref{prop:transCS}: If $[X]$ and $[Y]$ are homology classes that have representatives
whose evaluations at $0$ are disjoint, then their product $[X]\wedge [Y]=0$.
To see the ``duality", consider the picture one must see to have a
nontrivial product or coproduct: if $\widehat{\vee }[Z]\neq 0$ or $[X]\wedge
[Y]\neq 0$, one must see a picture that looks like this:

\centering
\includegraphics[width=5cm]{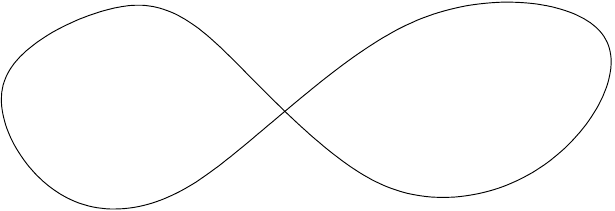}
\end{rem}

\begin{proof}[Proof of Theorem~\ref{thm:product2}]
  The statement is a reformulation of Theorem~\ref{thm:coproduct2} as we have the equality 
  $\langle\, [x_{1}]\circledast (\dots \oast (x_{k-2}{\circledast } (x_{k-1}\oast [x_{k}]))\dots ),Z \rangle=\langle [x_{1}]\ot [x_{2}]\ot\dots \ot [x_{k}],\vee^k Z\rangle$
  by definition of the cohomology product, and the product is associative up to sign (Theorem~\ref{thm:coGHalg}). 
\end{proof}

  \subsection{Computation of the coproduct for odd spheres (part I)}\label{sec:compspheres1}

Let $S^n$ be an odd dimensional sphere. From the computation of Cohen-Jones-Yan in \cite{CJY}, we have that $H_{\ast
}(\Lambda S^{n})=\textstyle{\bigwedge} (A)\otimes \mathbb{Z}[U]$ as a Chas-Sullivan algebra, with
generators $A\in H_0(\Lambda S^n)$, $1=[S^n]\in H_n(\Lambda S^n)$ and $U\in
H_{2n-1}(\Lambda S^n)$. 
We fix the orientation of the class $A\wedge U\in H_{n-1}(\Lambda S^n)$ by
requiring that, for $Z\colon \Sigma_{n-1}\to \Omega S^n$ a cycle
representing it, the map $\Sigma_{n-1}\times S^1\to S^n$ taking $(a,t)$ to $%
Z(a)(t)$ in homology is $(-1)^n $ times the fundamental class $[S^n]$. (The
sign $(-1)^n=-1$ is chosen for compatibility with the normal orientation of $%
\Delta S^n$ inside $S^n\times S^n$, see Remark~\ref{rem:orientations}.)  This fixes the orientation of 
$U$, and hence determines the orientation of each class $U^{\wedge k}$ and $A\wedge U^{\wedge k}$.

Let $\iota\colon \Omega S^n\to \La S^n$ denote the inclusion. Note that all classes of the form $A\wedge U^{\wedge n}$ are in the image of $\iota_*$. 
From the Gysin formula \cite[p 151]{GorHin}, we get that
\begin{align*}
   A\wedge U^{\wedge k}&=\iota_*(A\wedge U) \wedge (U)^{\wedge k-1}=\iota_*((A\wedge U)\bullet (A\wedge U)^{\bullet k-1})= (A\wedge U)^{\bullet k}
\end{align*}
where $\bullet$ denotes the Pontrjagin product in $H_*(\Omega S^n)$ induced by concatenation. Note that $A\wedge U^{\wedge k}$ has even degree $k(n-1)$. 
We will use this description of the classes $A\wedge U^{\wedge n}$ to compute their coproduct.

\begin{prop}\label{prop:spheres1}
  Let $n\ge 3$ be odd and write $H_{\ast}(\Lambda S^{n})=\textstyle{\bigwedge} (A)\otimes \mathbb{Z}[U]$ as above. The coproduct $\vee_{\Thom}$  is given on the  generators $A\wedge U^{\wedge n}$
  by the formula 
\begin{align*}
\vee_{\Thom} (A\wedge U^{\wedge k})&=\Sigma _{i=2}^{k-1}(A\wedge U^{\wedge i-1})\times (A\wedge U^{\wedge k-i})  \in H_*(\La S^n\x \La S^n,S^n\x \La S^n\cup \La S^n\x S^n)
\end{align*}
\end{prop}

Because each class $A\wedge U^{\wedge i-1}$ has even degree and $n$ is odd, the sign change $(-1)^{n-np}$ equals $(-1)$ for each term, so we have that 
\begin{align*}
\vee (A\wedge U^{\wedge k})=-\vee_{\Thom} (A\wedge U^{\wedge k})&=-\Sigma _{i=2}^{k-1}(A\wedge U^{\wedge i-1})\times (A\wedge U^{\wedge k-i}). 
\end{align*}
We will use the Thom version of the coproduct in the proof. 

\begin{proof} 
 A circle on $M$ is a pair $(\gamma, V )$ where 
$\ga(S^1)\subset M$ is the nonempty intersection of
$M$ with a 2-plane (not necessarily through the origin) in $\RR^{n+1}$, parametrized
at constant speed, and $V$ is a unit vector in $T_{\ga(0)}M$ with 
$\ga'(0) = \lambda V$ for some $\la\ge 0$, with  $\la\neq 0$ if and only if  $\ga$ is non-constant.

Let $p\in M$ and let $V$ be a unit vector in $T_pM$. The class $A\wedge U \in H_{n-1}(\La M)$ 
can be represented by the space $C_{p;V}$ of circles starting at $p$ with initial velocity
of the form $\la V$ for some $\la\ge 0$.
We want to use Proposition~\ref{prop:strong4} to compute the coproduct of $A\wedge U^{\wedge k}=(A\wedge U)^{\bullet k}$.

We can parametrize the class $C_{p;V}$ with $S^{n-1}\cong D^{n-1}/\del D^{n-1}$, identifying $D^{n-1}$ with the unit vectors $L\in \RR^{n+1}$ perpendicular to $V$ and such that $L\cdot p\ge 0$. 
Then $(V,L)$ represents 
the circle $(\ga,V)$ with $\ga$ the intersection of the plane $(V,L)$ with the sphere. The circle is a constant loop exactly when $L$ is tangent to the sphere, which happens exactly when $L\in \del D^{n-1}$, so we get a relative class $$Z_{AU}\colon (S^{n-1},\{b\}) \rar (\La S^n,S^n)$$ for $b$ the South pole of the sphere representing the collapsed
$\del D^{n-1}$. 

We can describe the product $(A\wedge U)^{\bullet k}$ as $C_{p,V_1}\bullet \dots\bullet C_{p,V_k}$ for $V_1,\dots,V_k$ pairwise distinct tangent vectors at $p$, which defines a chain representative 
$$Z_k\colon ((S^{n-1})^k,\{b\}^k) \rar (\La S^n,S^n). $$
When $k=1$, $Z_k$ consists of simple and constant loops, giving $\vee Z_1=0$ by Theorem~\ref{thm:coproduct2}, which agrees with the statement in the proposition.
So assume $k\ge 2$. The resulting chain $E(Z_k)=e_I\circ (Z_k\x I)$ is not at all transverse to the diagonal $\De S^n$ as all loops in the image of $Z_k$ have a self-intersection at time $s=\frac{i}{k}$ for
$i=1,\dots,k-1$.  We can deform $Z_k$ using the first component of the
tubular embedding  $$\nu_{\GH}=(\nu_{\GH}^{1},\nu_{\GH}^{2})\colon e^*(TM)|_{\F_{(0,1)}} \rar \La\x (0,1)$$
of Proposition~\ref{coemb}: Define $\tilde Z_k\colon  ((S^{n-1}) ^k,\{b\}^k) \rar (\La S^n,S^n)$
by
$$\tilde Z_k(x)=\nu_{\GH}^1\big(\dots \big(\nu_{\GH}^1\big(\nu_{\GH}^1\big(Z_k(x),\frac{1}{k},\la(x)(V_1-V_2)\big),\frac{2}{k},\la(x)(V_2-V_3)\big)\dots \big),\frac{k-1}{k},\la(x)(V_{k-1}-V_k)\big)$$
where $\la(x)=\ell(Z_k(x))$ is the total length of the loop $Z_k(x)$. In particular, when $Z_k(x)$ is constant (for $x=(b,\dots,b)$), we still have that $\tilde Z_k(x)$ is constant.
Here we choose the function $\mu$ in the definition of the push map $h$ of Lemma~\ref{H}  that controls $\nu_{\GH}$ in such a way that the deformation $\nu_{\GH}$ at time $\frac{i}{k}$ only affects the loop in the interval $(\frac{i}{k}-\frac{1}{2k},\frac{i}{k}+\frac{1}{2k})$. 
The definition of $\tilde Z_k$ then makes sense because $(Z_k(x),\frac{i}{k})$ lies in $\F_{(0,1)}$ for each $i=1,\dots,k-1$.  The effect of this deformation is to remove all the self-intersections at time $\frac{i}{k}$ for each $i$. The only remaining self-intersections are those coming from $Z_k(x)$ being constant on
a subinterval $[\frac{i}{k},\frac{i+1}{k}]$; after pulling the above loops tight, we can ensure that such self-intersections precisely happen at times $\frac{2i-1}{2k}$.   These non-trivial intersections are parametrized by the submanifold $\cup_{i=1}^k (S^{n-1})^{i-1}\x \{b\} \x (S^{n-1})^{k-i}\x \{\frac{2i-1}{2k}\} \subset (S^{n-1})^k\x I$. 

In the notation of Proposition~\ref{prop:strong4}, we have 
$\Si=(S^{n-1})^k$, with $\Si_0=(\{b\})^k$ and $\Si_\B=\Si_0\x I\cup \Si\x \del I\subset \Si\x I$. Then $E(\tilde Z_k)\colon  \Si\x I\minus \Si_\B \to M\x M$ and intersects  the diagonal in
$$\Si_\De=\bigcup_{i=1}^k (S^{n-1})^{i-1}\x \{b\} \x (S^{n-1})^{k-i} \x \{\frac{2i-1}{2k}\}=\overline{\Si_\De}.$$
We are left to check that the intersection is transverse. Let $x=(x_1,\dots, x_{i-1},b,x_{i+1},\dots,x_k,\{\frac{2i-1}{2k}\})$ be an element of $\Si_{\De}$. The tangent space $T_x(\Si\x I)$ at $x$ is isomorphic to
$$ D(x_1)\x \dots\x D(x_{i-1})\x D(b)\x D(x_{i+1})\x \dots,D(x_k)\x [\frac{2i-1}{2k}-\eps,\frac{2i-1}{2k}+\eps]$$
where $D(y)$ denotes a small disc centered at $y\in S^{n-1}$, and where $D(b)\x  [\frac{2i-1}{2k}-\eps,\frac{2i-1}{2k}+\eps]$ the directions  in  $T(\Si\x I)$ that are normal to $T\Si_\De$. 
 Transversality follows from the fact that the points of $D(b)\minus \{b\}$ correspond to all the directions $L$ that are close to tangent at $p$, perpendicular to $V_i$, generating in $T(S^n\x S^n)/T\De S^n$ the subspace that is normal to $(0,V_i)$, 
while varying
$t\in [\frac{2i-1}{2k}-\eps,\frac{2i-1}{2k}+\eps]$ gives the remaining normal direction $(0,V_i)$.

After applying the cut-map, and throwing away the terms $i=1$ and $i=k$ that have image in $S^n\x \La S^n$ and $\La S^n\x S^n$ respectively, we obtain that
$$\vee_{\Thom} (A\wedge U^{\wedge k})=\sum_{i=2}^{k-1}(A\wedge U)^{\bullet i-1}\x (A\wedge U)^{\bullet k-i}$$
as claimed in the statement of the proposition. There is no sign as we have chose the orientation of $A\wedge U$ precisely so that the identification
$D(b)\x  [\frac{2i-1}{2k}-\eps,\frac{2i-1}{2k}+\eps]\to N\De S^n$ is orientation preserving, and hence we also get that the identification $T(\Si\x I)\cong E(\tilde Z_k)^*N\De S^n\op T\Si_\De$  along $\Si_{\De}$ is orientation preserving as the subspaces $D(x_i)$ are all even dimensional. 
\end{proof}

\begin{rem}\label{rem:equsimple}
\begin{figure}[ht]
\includegraphics[width=4cm]{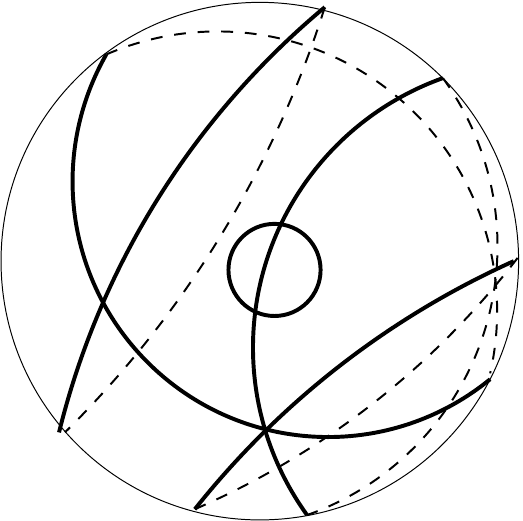}
\caption{Connected union of 5 circles on $S^2$}
\label{fig:circles}
\end{figure}

Let $G\subset S^{n}$ be a connected union of $k$ circles on 
$S^n$ (e.g.~as in Figure~\ref{fig:circles}). Let $\Lambda _{G}\subset \La=\Lambda S^n$
be the set of loops $\gamma \colon S^{1}\rightarrow S^n$ with image $G$ that are
injective except at the ``vertices": We allow $\gamma \in \Lambda _{G}$ to
have a $k$-fold self-intersection with $k>1$ only at a point where two
distinct circles in $G$ intersect, or if $G$ consists of a single point. (In
other words, $\gamma$ is an ``Eulerian path'', as in the K{\"o}nigsberg
bridges problem.) Note that $\Lambda _{G}$ is nonempty by Euler's theorem.
If $\gamma \in \Lambda _{G}$, it is easy to see that $\gamma $ has at most $%
k $-fold self-intersections, since there are at most $2k$ edges at each
vertex. \ Let $\mathcal{G}_{k}\subset \Lambda $ be the union of all the $%
\Lambda _{G}$, where $G$ is the union of $k$ circles: 
\begin{equation*}
\mathcal{G}_{k}=\underset{%
\begin{array}{c}
G=\{c_{1},...,c_{k}\} \\ 
\gamma \in \Lambda _{G}%
\end{array}%
}{\bigcup}\hspace{-10mm}\gamma \ \ \ \ \subset \Lambda
\end{equation*}%
\ Then it follows from Theorem~\ref{thm:coproduct2} that 
\begin{equation*}
\vee ^{k}=0\text{ on }\mathcal{G}_{k}.
\end{equation*}
The space $\underset{k}{\cup }\mathcal{G}_{k}$ (or some subset thereof; such
as the subset of loops parametrized proportional to arc length, or with an
orientation condition) has been suggested as a ``small model", a subspace of 
$\Lambda $ that is invariant under the natural $O(2)$-action, and that
reflects in some sense the equivariant topology of $\Lambda$\footnote{Wilhelm Klingenberg, private conversation}. More examples of small models that have been studied during
the search for closed geodesics include the finite dimensional approximation of
Morse \cite{Morse},  the space of biangles \cite{BTZ}, and the labeled configuration space model of  \cite{CarCoh}. See also Remark~\ref{rem:simplemodel} for a non-equivariant model.
\end{rem}

\section{The lifted homology coproduct and  cohomology product}
\label{sec:extend}

In this section we will define a lift of the coproduct $\vee$ to the
non-relative chains $C_*(\Lambda)$, and  of the product $\oast$ to the non-relative cochains $C^*(\Lambda)$. This will allow us later in the section to analyse the, now meaningful, compositions
$\widehat\vee\circ \wedge$ and $\wedge\circ \widehat\vee$.

\medskip

As above, we let $e\colon \Lambda \longrightarrow M$ denote the evaluation
map at $0$, and we write $i\colon M\to \Lambda$ for the inclusion as constant
loops. We also write $1$ for the map induced by the identity on chains,
cochains, homology or cohomology.

The maps 
\begin{equation*}
\xymatrix{\La \ar[r]_e & M \ar@/_1pc/[l]_i }
\end{equation*}
satisfy that $e\circ i$ is the identity map on $M$. Hence they induce a
decomposition 
\begin{equation*}
C^{\ast }(\Lambda )\cong C^{\ast }(\Lambda ,M)\oplus C^{\ast }(M) \ \ \ 
\text{ and } \ \ \ H^{\ast }(\Lambda )\cong H^{\ast }(\Lambda ,M)\oplus
H^{\ast }(M)
\end{equation*}
(see e.g.~\cite[p147]{Hatcherbook}).
This allows to define an ``\textit{extension by 
}$0$" $\widehat{\oast}$ of the product $\oast$ by setting

\begin{itemize}
\item $\widehat{\circledast }=\circledast$ on $C^{\ast }(\Lambda ,M)\otimes
C^{\ast }(\Lambda ,M)$;

\item $\widehat{\circledast }=0$ on $C^{\ast }(M)\otimes C^{\ast}(\Lambda)$
and $C^{\ast }(\Lambda)\otimes C^{\ast }(M)$.
\end{itemize}

The product $\widehat{\circledast }$ will in addition have the property that its image lies in $%
C^{\ast }(\Lambda ,M)\subset C^{\ast }(\Lambda )$. Note also that the
splitting map $C^*(\Lambda)\to C^*(\Lambda,M)$ is simply the map 
\begin{equation}  \label{equ:split}
(1-e^*i^*)\colon C^*(\Lambda)\longrightarrow C^*(\Lambda,M).
\end{equation}
This is the standard formula for the splitting, and one can check that for
any cochain $x\in C^*(\Lambda)$, the cochain $(1-e^*i^*)x$ vanishes on any
chain $A\in i_*C_*(M)$, so the map $(1-e^*i^*)$ indeed has image in $C^*(\Lambda,M)$%
. Moreover, the map takes $e^*C^*(M)$ to 0 and $C^*(\Lambda,M)$ to itself
via the identity.

\begin{thm}
\label{thm:liftedproduct1} The maps $e\colon \Lambda \leftrightarrows M : i$
induce a natural lift
\begin{equation*}
\widehat{\circledast }\colon C^{p}(\Lambda )\otimes C^{q}(\Lambda )\longrightarrow
C^{p+q+n-1}(\Lambda )
\end{equation*}%
of the product $\circledast$ of Definition~\ref{def:cast}. It is given by
the following formula: If $x,y\in C^{\ast }(\Lambda )$,  then 
\begin{equation*}
\begin{aligned}\label{proformula} x\widehat{\circledast }y&=(1-e^{\ast
  }i^{\ast })x\circledast (1-e^{\ast }i^{\ast })y \\
\end{aligned}
\end{equation*}
considering $(1-e^*i^*)$ as a map from $C^*(\Lambda)$ to $%
C^*(\Lambda,M)\subset C^*(\Lambda)$, and has the following properties:

\begin{enumerate}[(a)]

\item The product $\widehat{\circledast }$ is the unique ``extension by 0"
on the trivial loops, that is satisfying that 
\begin{equation*}
\langle x\widehat{\circledast }y,Z\rangle=0\text{ if }x\in e^{\ast }C^{\ast
}(M),y\in e^{\ast }C^{\ast }(M),\text{\textit{\ or }}Z\in i_{\ast }C_{\ast}(M), 
\end{equation*}
while $x\widehat\oast y=x\oast y$ if $x,y\in C_*(\La,M)$. 

\item The associated cohomology product satisfies the following graded associativity and commutativity relations: 
$$[x]\widehat\oast [y]=(-1)^{pq+1}[y]\widehat\oast [x] \ \ \ \textrm{and}\ \ ([x]\widehat\oast [y])\widehat\oast [z]=(-1)^{r+1}[x]\widehat\oast ([y]\widehat\oast [z])$$
for any $[x]\in H^{p-n}(\Lambda)$,  $[y]\in H^{q-n}(\Lambda)$ and $[z]\in H^{r-n}(\Lambda)$.

\item Morse theoretic inequality: Let $[x],[y]\in H^{\ast}(\Lambda )$, then 
\begin{equation*}
\Cr([x]\widehat{\circledast }[y])\geq \Cr([x])+\Cr([y])
\end{equation*}
where $\Cr([x])=\sup\{l\in \mathbb{R} \, |\, [x] \text{ is supported on }
\Lambda^{\ge L}\}$ for $\Lambda^{\ge L}$ the subspace of loops of energy at
least $L^2$.
\end{enumerate}
\end{thm}

Recall from Section~\ref{H1loopssec} that the loops of energy at most $L^2$
also have length at most $L$ as $\ell(\gamma)^2\le E(\gamma)$, for $\ell$
the length function and $E$ the energy.

\begin{proof}
We define the lifted product $\widehat{\circledast }$ by the following
commutative diagram: 
\begin{equation*}
\xymatrix{ C^{j}(\Lambda )\otimes C^{k}(\Lambda ) \ar@{->>}[d]_{p^{\ast
}\otimes p^{\ast }}\ar[r]^-{\widehat{\circledast}} & C^{j+k+n-1}(\Lambda )
\\ C^{j}(\Lambda ,M)\otimes C^{k}(\Lambda ,M)\ \ar[r]^-{\circledast}& \
C^{j+k+n-1}(\Lambda ,M) \ar@{^(->}[u]_{q^{\ast }} }
\end{equation*}
where $p_*=(1-i^*e^*)$ is the splitting map (\ref{equ:split}) and $q^*$ is
the canonical inclusion. We see that this gives the formula in the
statement, where we have though suppressed the inclusion $q^*$. As the maps $%
e$ and $i$ are natural, this lift is natural in $M$.

\smallskip

We check (a): We have that $\langle x\widehat{\circledast }y,Z\rangle=
\langle p^*x\oast p^*y,q_*Z\rangle$ for $q_*\colon C_*(\Lambda)\to
C_*(\Lambda,M)$ the dual of the map $q^*$. Now the latter bracket is $0$ if $%
x\in e^{\ast }C^{\ast }(M)$ as $p^*(x)=0$ in that case, if $y\in e^{\ast
}C^{\ast }(M)$ as $p^*y=0$ in that case, or if $Z\in i_{\ast }C_{\ast}(M)$
as $q_*(Z)=0$ in that case. On the other hand, $p^*(x)=x$ and $p^*(y)=y$ for $x,y\in C^*(\La,M)$, giving that $x\widehat\oast y=x\oast y$ in that case. And because $C_*(\Lambda)\cong
C_*(\Lambda,M)\oplus C_*(M)$,
this determines the map.

\medskip

Graded associativity and commutativity follows directly from the graded commutativity of $%
\oast
$, as proved in Theorem~\ref{thm:coGHalg}. For associativity, if $[x]$, $[y]$, 
$[z]\in H^{\ast}(\Lambda)$ are as in the statement, 
\begin{align*}
([x]\widehat{\circledast }[y])\widehat{\circledast }[z] &= (1-e^{\ast
}i^{\ast })\Big( (1-e^{\ast}i^{\ast })[x]\circledast (1-e^{\ast }i^{\ast
})[y] \Big) \circledast (1-e^{\ast }i^{\ast })[z] \\
&= (-1)^{r+1}\Big( (1-e^{\ast}i^{\ast })[x]\circledast (1-e^{\ast }i^{\ast })[y] \Big) %
\circledast (1-e^{\ast }i^{\ast })[z] \\
&=(-1)^{r+1} (1-e^{\ast}i^{\ast })[x]\circledast \Big((1-e^{\ast }i^{\ast })[y]
\circledast (1-e^{\ast }i^{\ast })[z]\Big) \\
&=(-1)^{r+1} (1-e^{\ast}i^{\ast })[x]\circledast (1-e^{\ast }i^{\ast })\Big((1-e^{\ast
}i^{\ast })[y] \circledast (1-e^{\ast }i^{\ast })[z]\Big) \\
&=(-1)^{r+1}[x]\widehat{\circledast }([y]\widehat{\circledast }[z])
\end{align*}%
where the first and last equality hold by definition of $\widehat{\oast}$,
the second and fourth using that any product $[a] \circledast [b] \in
H^*(\Lambda,M)\cong \ker i^{\ast }$ so $(1-e^{\ast }i^{\ast })=1$ on such
products. Finally the middle equality is the graded associativity of $\circledast $. This proves statement (b).

The last property will be deduced from the corresponding results for $\oast$ in \cite{GorHin}, which is justified by our Theorem~\ref{thm:cocoequ}.
Given $[x]\in H^*(\Lambda)$, write $[x]=[x_0]+[x_1]$ with $x_0\in e^*C^*(M)$ and $x_1\in q^*C^*(\Lambda,M)$. Then $\Cr([x])=\Cr([x_1])$ as
$[x_0]$ is represented by loops of length $0$. Now $(1-i^*e^*)[x]=[x_1]$,
which gives that $\Cr\big((1-i^{\ast }e^{\ast })[x]\big)=\Cr([x_1])=\Cr([x])$.
So 
\begin{align*}
\Cr([x]\widehat{\circledast }[y]) &=\Cr((1-i^{\ast }e^{\ast })[x]\circledast
(1-i^{\ast }e^{\ast })[y]) \\
&\geq \Cr((1-i^{\ast }e^{\ast })[x])+\Cr((1-i^{\ast }e^{\ast })[y]) = \Cr([x])+\Cr([y]),
\end{align*}
where the middle inequality is given by \cite[(1.7.3)]{GorHin}. Statement
(c) follows, which finishes the proof of the theorem.
\end{proof}

Dually, we have the following result: 

\begin{thm}
\label{thm:extendedcoproduct} The maps $e\colon \Lambda \leftrightarrows M : i$
induce a natural lift
\begin{equation*}
\widehat{\vee}\colon C_{*}(\Lambda )\longrightarrow C_*(\Lambda) \ot C_*(\Lambda )
\end{equation*}%
of the coproduct $\vee$ of Definition~\ref{def:stickyco}. It is given by the
following formula: If $Z\in C_{\ast }(\Lambda)$, then 
\begin{equation*}
\widehat{\vee}Z=(1-i_*e_*)\otimes(1-i_*e_*)\vee (q_*Z)
\end{equation*}
with $(1-i_*e_*)\otimes (1-i_*e_*)\colon C_*(\Lambda,M)\ot C_*(\Lambda,M) \longrightarrow
C_*(\Lambda)\times C_*(\Lambda)$ induced by the splitting, and $q_*\colon C_*(\Lambda)\to C_*(\Lambda,M)$
the projection. It has the following properties:

\begin{enumerate}[(A)]

\item The coproduct $\widehat{\vee}$ is the unique ``extension by 0" on the
trivial loops, satisfying 
\begin{equation*}
\langle x,\widehat \vee Z\rangle=0\ \text{ if }\ x\in e^{\ast }C^{\ast }(M)\otimes C^*(\Lambda) \ \text{or} \ C^*(\Lambda)\otimes e^*C^*M), 
\text{\textit{\ or }} Z\in i_{\ast }C_{\ast}(M).
\end{equation*}

\item Duality: The lifted product and coproduct, and the Kronecker
product satisfy 
\begin{equation*}
\langle x\widehat{\circledast}y,Z\rangle=\langle x\otimes y,%
\widehat{\vee }Z\rangle
\end{equation*}
for any $x\in C^{p}(\Lambda )$, $y\in C^q(\Lambda)$ and $Z\in
C_{p+q+n-1}(\Lambda )$.

\item Support:  The lifted coproduct $\widehat\vee$ and its iterates $\widehat\vee^k$ have the following support property:

\begin{enumerate}[(i)]

\item If $Z\in C_{\ast }(\Lambda)$ has the property that every nonconstant
loop in the image of $Z$ has at most $k$-fold intersections, then%
\begin{equation*}
\widehat\vee^{k}[Z]=\widehat\vee^{k}(\Delta[Z])=0\in H_*(\Lambda^k).
\end{equation*}

\item If $Z\in C_{\ast }(\Lambda)$ is a cycle with the property that every
nonconstant loop in the image of $Z$ has at most a $k$-fold intersection at
the basepoint, then%
\begin{equation*}
\widehat\vee^{k}[Z]=0\in H_*(\Lambda^k).
\end{equation*}
\end{enumerate}
\end{enumerate}
\end{thm}

\begin{proof}
We define the extended coproduct $\widehat{\vee}$ to be the map making the
following diagram commutative 
\begin{equation*}
\xymatrix{C_*(\La) \ar@{->>}[d]_{q_*} \ar[r]^-{\poc} & C_*(\La\x\La) \\
C_*(\La,M) \ar[r]^-{\vee} & C_*(\La,M)\ot C_*(\La,M)\ar@{^(->}[u]_{P_*} }
\end{equation*}
where the inclusion $P_*=(1-i_*e_*)\otimes(1-i_*e_*)$. This gives precisely the formula in
the statement. Note that in order to lift the coproduct $\vee$, we need to
``complete" $\vee$ by subtracting something coming from $C_*(\Lambda)\otimes
C_*(M)+C_*(M)\otimes C_*(\Lambda)$ that has the same boundary in that complex, so that
if $\vee A$ was a relative cycle, $\widehat{\vee} A$ becomes an actual
cycle. This is precisely what the map $(1-i_*e_*)\otimes(1-i_*e_*)$ does.

Statement (A) follows from the fact that $q_*(Z)=0$ if $Z\in i_*H_*(M)$ and
the fact that $P^*$ vanishes on the image of $C^*(M)\otimes C^*(\Lambda)$ and $%
C^*(\Lambda)\otimes C^*(M)$ as $(1-e^*i^*)=0$ on $e^*C^*(M)$. And the uniqueness
follows again from the splittings, as in statement (a) in Theorem~\ref%
{thm:liftedproduct1}.

Statement (B) follows from the fact that 
\begin{equation*}
\begin{aligned}\  \langle
  x\widehat{\circledast}y,Z\rangle 
  &=\langle \vee^*\big((1-e^*i^*)x\otimes(1-e^*i^*)y\big),Z\rangle \\
&=\langle x\otimes y,\big((1-i_*e_*)\otimes(1-i_*e_*)\big){\vee}Z\rangle =\langle x\otimes y,\widehat{\vee }Z\rangle \end{aligned}
\end{equation*}
when $x\in C^{p}(\Lambda )$, $y\in C^q(\Lambda)$ and $Z\in
C_{p+q+n-1}(\Lambda )$.

Finally statement (C) follows from Theorem~\ref{thm:coproduct2} and the fact that
$$\widehat\vee^k[Z]=(1-i_*e_*)\ot\dots\ot(1-i_*e_*)\vee^k[Z]$$
and hence $\widehat\vee^k[Z]=0$ if and only if $\vee^k q_*[Z]=0$ as $(1-i_*e_*)$ is injective on the image of $\vee$. 
\end{proof}

\subsection{Triviality of product following the coproduct}

\label{sec:vanishing}

We will show in this section that, with our choice of lift of the coproduct, the composition $\wedge\circ \widehat\vee$ is trivial, that is the product and coproduct behave like the bracket and cobracket in an involutive Lie bialgebra.
After completing our computation of the coproduct on the loop space of spheres,  we show in Remark~\ref{rem:Frob} that the reversed composition 
$\widehat\vee\circ \wedge$ does not satisfy the expected Frobenius relation, and show in Remark~\ref{rem:genus1} that introducing a circle action in the middle of the composition   $\wedge\circ \widehat\vee$ makes it a non-trivial operation.

\begin{thm}
\label{thm:trivial} The composition 
\begin{equation*}
C_*(\Lambda) \overset{\widehat{\vee}}{\longrightarrow }C_{*}(\Lambda)\otimes C_*(\Lambda) \overset{\wedge}{\longrightarrow } C_{*+1-2n}(\Lambda)
\end{equation*}
of the lifted coproduct followed by the  Chas-Sullivan product induces the
zero map in homology. The same holds for the composition $\wedge_{\Thom}\circ \widehat\vee_{\Thom}$. 
\end{thm}

\begin{rem}
The above result implies that the compositions $\wedge\circ \widehat{\vee}_{\Thom}$ and $\wedge_{\Thom}\circ \widehat{\vee}$ are trivial mod~$2$. As we will see in Remark~\ref{rem:wedgepoc},
reducing mod $2$ is necessary for these other vanishing statements.
\end{rem}

We will give two proofs of the theorem: a short proof  using signs assuming that $2$ is invertible in the coefficients,
and  a maybe more enlightening geometric proof showing that this vanishing is closely related to the vanishing of the
so-called trivial coproduct. We will in fact show that the Chas-Sullivan product is essentially trivial on the figure eight space $\La\x_M\La$, which is also the target of the coproduct. (See Proposition~\ref{prop:CSF2}.)

\begin{proof}[Proof of Theorem~\ref{thm:trivial} using signs when $2$ is invertible]
 Recall from Theorems~\ref{thm:CSalg}  and~\ref{thm:coGHalg} that the product $\wedge$ is graded commutative while the coproduct $\widetilde\vee$ is graded anti-commutative, after a degree $n$ shift of the homology.
 More precisely, for $A\in C_{k+n}(\Lambda;k)$, we have 
  $$\vee A =\sum_{p+q=k+1}A_p^0\ot A_q^1\ \ \simeq  \sum_{p+q=k+1}(-1)^{pq+1}A_p^1\ot A_q^0,$$
for $A_p^0\ot A_q^1$ the summand of bidegree $(p+n,q+n)$  in $\vee A$. 
Hence $\wedge(\vee(A))=\sum_{p+q=k+1}A_p^0\wedge A_q^1\simeq \sum_{p+q=k+1}(-1)^{pq+1}A_p^1\wedge A_q^0$. On the other hand, by the graded commutativity of the product, we also have
$$\wedge(\vee(A))=\sum_{p+q=k+1}A_p^0\wedge A_q^1\ \ \simeq \sum_{p+q=k+1}(-1)^{pq}A_p^1\wedge A_q^0.$$
Hence $\wedge(\vee([A]))=-\wedge(\vee([A]))$ in homology and therefore must be 0 if $2$ is invertible.

The vanishing for the Thom-signed versions follows from the fact that the sign change is the same for both operations, so $\wedge_{\Thom}\circ \widehat\vee_{\Thom}=\wedge\circ \widehat\vee$. 
\end{proof}

For the geometric proof of the result, we recall first the \emph{trivial coproduct}, whose triviality was pointed
out by Tamanoi \cite{Tam}. One can define a coproduct using the sequence of
maps analogous to that of the coproduct studied here, but fixing the
parameter $s=\frac{1}{2}$. More precisely, let 
\begin{equation*}
\F_{\!\frac{1}{2}}=\{\gamma\in \Lambda\ |\ \gamma(\frac{1}{2})=\gamma(0)\} \
\subset \Lambda
\end{equation*}
and consider its neighborhood 
\begin{equation*}
U_{\frac12}=U_{\frac12,\varepsilon}=\{\gamma\in \Lambda\ |\ |\gamma(\frac{1}{2}%
)-\gamma(0)|<\varepsilon\} \ \subset \Lambda.
\end{equation*}
Then restricting the tubular neighborhood $\nu_{\GH}$ to $s=\frac{1}{2}$
defines a tubular neighborhood 
\begin{equation*}
\nu_{\frac12}\colon  \xymatrix{e^*(TM) \ar[d]\ar[r] & \La \\ \F_{\!\frac{1}{2}}& }
\end{equation*}
of $\F_{\!\frac{1}{2}}$ inside $\Lambda$, with image $U_{\frac12}$. Denote by $%
\kappa_{\frac12}=(k_{\frac12},v_{\frac12})$ the associated collapse map, $R_{\frac12}\simeq k_{\frac12}$ the
restriction of $R_{\GH}$ to time $s=\frac{1}{2}$, and 
\begin{equation*}
\tau_{\frac12}=e_{\frac12}^*(\tau_M)\in C^n(\Lambda\times\Lambda,U_{T,\varepsilon_0}^c)
\end{equation*}
the associated Thom class, for $e_{\frac12}\colon \La\to M\x M$ the evaluation at $0$ and $\frac12$. This yields a degree $-n$ coproduct which we can
model via the sequence of maps 
\begin{equation*}
\vee_{\!\frac{1}{2}}\colon C_*(\Lambda)\xrightarrow{\tau_{\frac12}\cap} C_{*-n}(U\!_{\frac12})\xrightarrow{R_{\frac12}} C_{*-n}(\F_{\!\frac{1}{2}})\xrightarrow{\cut}C_{*-n}(\Lambda\times \Lambda).
\end{equation*}
Tamanoi, who also suggested the description of the coproduct in terms of the
retraction map $R_{\frac12}$, showed in \cite{Tam} that this operation is almost
identically zero!
One way to formulate what Tamanoi proved is the following: The coproduct $\vee_{\frac12}$ satisfies the Frobenius identity
$$\vee\!_{\frac12}([A]\wedge [B])=(\vee\!_{\frac12}[A])\wedge [B]=(-1)^{n-np}[A]\wedge (\vee\!_{\frac12}[B])\in H_*(\La\x\La)$$
(see \cite[Thm 2.2]{Tam}). Tamanoi then applied this identity to either the left or right term being the unit $[M]$, and used this to conclude that the coproduct $\vee_{\frac12}$ can only be non-trivial in degree $n$, and only on the constant loops $H_n(M)\le H_n(\La)$. We prove here a slight refinement of that result, using the same idea. 
Let 
\begin{equation*}
\vee^\F_{\!\frac{1}{2}}\colon C_*(\Lambda)\xrightarrow{\tau_{\frac12}\cap}C_{*-n}(U_{\frac12})\xrightarrow{R_{\frac12}} C_{*-n}(\F_{\!\frac{1}{2}})\cong C_{*-n}(\La\x_M\La)
\end{equation*}
be the trivial coproduct prior to the cutting map, so that $\vee_{\!\frac{1}{%
2}}=\cut\circ \vee^\F_{\!\frac{1}{2}}$. 

\begin{lem}
  \label{lem:precop2}
The induced map in homology $\vee^\F_{\!\frac{1}{2}}\colon H_*(\La)\to H_{*-n}(\La\x_M\La)$, under the splittings
$H_*(\La)\cong H_*(\La,M)\oplus H_*(M)$ and $H_*(\La\x_M \La)\cong H_*(\La\x_M \La,\De M)\oplus H_*(M)$, is only non-zero on the summands $H_*(M)$ of the source and target. On that summand it is the map capping with $(-1)^ne_M$, for $e_M$ the Euler class of $M$. 
\end{lem}

\begin{proof}
We first show  that the map $\vee^\F_{\frac12}\colon H_*(\La)\to H_*(\La\x_M\La)$ factors through both
$H_*(M\x_M\La)$ and $H_*(\La\x_MM)$, using that $\vee^\F_{\frac12}=\vee^\F_{\frac34}$ in homology, where we replace $\frac12$ by $\frac34$ everywhere in the definition of the map, and use that $\F_{\frac34}\cong\F_{\frac12}\cong \La\x_M\La$ are all homeomorphic spaces. Indeed, consider the sequence of maps
$$\xymatrix{C_p(\La)\ot C_n(\La)\ar[d]_\x & \ar[l]_-{1\ot i} C_p(\La)\ot C_n(M) \ar[d]^\x & \ni (-1)^{n-np}A\ot [M]\\
  C_{p+n}(\La\x \La)\ar[d]_{[\tau_{\CS}\cap]} & C_{p+n}(\La\ot M) \ar[l]_-{1\x i}\ar[d]^{[(1\x i)^*\tau_{\CS}\cap]} & \\
  C_p(U_{\CS}) \ar[d]_{R_{\CS}} & C_p(U_{\CS}\cap (\La\x M)) \ar[l]_-{1\x i} \ar[d]^{R_{\CS}} &\\
  C_p(\La\x_M\La) \ar[d]_{\concat} & C_p(\La\x_M M) \ar[l]_-{1\x i} \ar[r]^-{[1\x i^*\tau_{\frac12}\cap]} \ar[d]_{[1\x i^*\tau_{\frac12}\cap]} & C_{p-n}(\La\x_M M)\ar[d]^{1\x i}\\
  C_p(\La) \ar[r]^-{[\tau_{\frac34}\cap]} & C_{p-n}(U_{\frac34}) \ar[r]^-{R_{\frac34}} & C_{p-n}(\F_{\frac34}) \cong C_{p-n}(\F_{\frac12})\cong C_{p-n}(\La\x_M\La).
}$$
Starting with $ (-1)^{n-np}A\ot [M]$ on the top right corner and following the outside of the diagram in homology computes
$$\vee^\F_{\frac34}( (-1)^{n-np}[A]\wedge_{\Thom} [M])=\vee_{\frac12}^\F[A]$$
and we see that this computation factors through
$$H_{p-n}(\La\x_M M)\le H_{p-n}(\La\x_M\La)\cong H_{p-n}(\La\x_M\La,\La\x_MM)\op H_{p-n}(\La\x_MM).$$
An analogous diagram using $M\x \La$ instead of $\La\x M$ shows that the coproduct also factors $H_{p-n}(M\x_M \La)\le H_{p-n}(\La\x_M\La)$. 
Using the splitting $$H_*(\La\x_M\La)\cong H_*(\La\x_M\La,\La\x_MM)\op H_*(\La\x_MM,M)\op H_*(M),$$
we can conclude  that the map $\vee^\F_{\frac12}$ in fact has image in the summand $H_*(M)$ only, because $H_*(M\x_M\La)\to H_*(\La\x_M\La)$ has image in the first and last summand only. 

Finally, the map $\vee^\F_{\frac12}$ respects the splittings
$H_*(\La)\cong H_*(\La,M)\oplus H_*(M)$ and $H_*(\La\x_M \La)\cong
H_*(\La\x_M \La,\De M)\oplus H_*(M)$, that is splits as the sum of two
maps (and no mixed terms). Indeed,
$\vee^\F_{\frac12}=R_{\frac12}\circ [\tau_{\frac12}\cap]$. The
retraction map $R_{\frac12}$ commutes with the maps $i$ and $e$
inducing the splitting. For $[\tau_{\frac12}\cap]$, this property is the commutativity of 
$$\xymatrix{H_*(M) \ar[d]_{[(-1)^ne_M\cap]=[i^*\tau_{\frac12}\cap]}  \ar[r]^i & \ar[d]^{[\tau_{\frac12}\cap]} H_*(\La) \ar[r]^e & H_*(M)  \ar[d]^{[(-1)^ne_M\cap]} \\
    H_{*-n}(M) \ar[r]^i & H_{*-n}(U_{\frac12}) \ar[r]^e & H_{*-n}(M)
  }$$
  which holds because  $i^*\tau_{\frac12}=(-1)^ne_M$ on $M$ as the normal bundle of the diagonal is isomorphic to the tangent bundle, with a change of orientation $(-1)^n$ by our convention (Remark~\ref{rem:orientations}), 
  while  $(-1)^ne^*e_{M}=e_{\frac12}=e_{\frac12}^*\tau_M\in H^n(\La,U_{\frac12}^c)$ because $e_{\frac12}\simeq (e,e)\colon \La\to M\x M$ by sliding the evaluation from $\frac12$ to $0$ along the loop, and $(e,e)^*\tau_M=e^*\De^*\tau_M=(-1)^ne^*e_M$.
  
Hence $\vee_{\frac12}$ is only non-trivial on the summand $H_*(M)$ of the source and target, and has the form given in the proposition. 
  \end{proof}

The Chas-Sullivan product $\wedge_{\Thom}$ is defined as a composition 
\begin{equation*}
H_*(\Lambda)\otimes H_*(\Lambda) \overset{\times}{\longrightarrow }%
H_*(\Lambda\times\Lambda) \xrightarrow{\wedge_{\Thom}^\x} H_*(\Lambda)
\end{equation*}
where $\wedge_{\Thom}^\x$ is the \emph{short} Chas-Sullivan product.
The key ingredient in the geometric proof of Theorem~\ref{thm:trivial} is that  the Chas-Sullivan product behaves similarly to the map $\vee_{\frac12}^\F$
when restricted to the figure eight space $\La\x_M\La$:
\begin{prop}\label{prop:CSF2}
  The short Chas-Sullivan product $\wedge^\x$, restricted to the image of $H_*(\La\x_M\La)\cong H_*(\La\x_M\La,M)\op H_*(M)$ in $H_*(\La\x \La)$, is only non-trivial on the second summand, the constant loops, where it is the map capping with $(-1)^ne_M$, for $e_M$ the Euler class of $M$.
  \end{prop}

  \begin{rem}
    The above statement may sound a little surprising if compared to Proposition~\ref{prop:transCS}: how can the product know whether the loops are trivial or not? Recall that $e_M\cap [M]=\chi(M)$ is the Euler characteristic of $M$. Suppose $\chi(M)\neq 0$ and  $A,B$ are two smooth families of loops such that $A\x B\in C_n(\La\x_M\La)$ is parametrized by $\De M$ under the evaluation at $0$.
    If the loops in the images of $A$ and $B$ were all non-trivial, one would get two non-vanishing vector fields on $M$ by taking the derivative at $0$, which is not possible if $\chi(M)\neq 0$.
    In a generic situation, we will be able to push all the self-intersections away from each other using an appropriate linear combination of these two vector fields, except at the points where both families have trivial loops. This shows that their product indeed has image in the constant loops $H_0(M)\le H_0(\La\x_M\La)\le H_0(\La\x \La)$.  What the proposition says is that
    if we write $$[A\x B]=[A\x B]_0 + [A\x B]'\in H_n(M)\op H_n(\La\x_M\La,\De M)\cong H_n(\La\x_M\La)$$
    then $[A\wedge B]=\wedge[A\x B]_0=(-1)^n\chi_M$, which is also what we can compute geometrically. 
    \end{rem}

\begin{proof}[Proof of Proposition~\ref{prop:CSF2}]
Note first that concatenation gives an identification 
\begin{equation*}
\Lambda\times_M\Lambda \overset{\cong}{\longrightarrow }\F_{\!\frac{1}{2}%
}\subset \Lambda
\end{equation*}
compatible with the splittings $H_*(\La\x_M\La)\cong H_*(\La\x_M\La,M)\op H_*(M)$ and
$H_*(\F_{\!\frac{1}{2}})\cong H_*(\F_{\!\frac{1}{2}},M)\op H_*(M)$. Let $\tilde j\colon \F_{\!\frac{1}{2}}\to \Lambda$ denote the canonical inclusion and consider the diagram 
\begin{equation*}
\xymatrix{ H_*(\La\x \La) \ar[r]^-{\tau_{\CS}\cap} & H_{*-n}(U_{\CS})
\ar[r]^-{R_{\CS}} & H_{*-n}(\La\x_M\La)\\ 
H_*(\La\x_M\La,M)\op H_*(M) \ar[u]_j\ar[d]_\cong \ar[r]^-{j^*\tau_{\CS}\cap} &
H_{*-n}(\La\x_M\La) \ar[d]_\cong \ar[r]^-{id} \ar[u]^j & H_{*-n}(\La\x_M\La,M)\op H_*(M)  \ar[d]^\cong \ar@{=}[u]  \\
H_*(\F_{\!\frac{1}{2}},M)\op H_*(M)  \ar[d]^{\tilde j} \ar[r]^-{\tilde j^*\tau_{\frac12}\cap}& H_{*-n}(\F_{\!\frac{1}{2}}) \ar[r]^-{id}
\ar[d]^{\tilde j} & H_{*-n}(\F_{\!\frac{1}{2}},M)\op H_{*-n}(M) \ar[d]^\cong \\
H_*(\La,M) \op H_*(M) \ar[r]^-{\tau_{\frac12}\cap} & H_{*-n}(U_{\frac12}) \ar[r]^-{R_{\frac12}} & H_{*-n}(\La\x_M \La,M) \op H_{*-n}(M). }
\end{equation*}
The top row of the diagram is the Chas-Sullivan product $\wedge_{\Thom}^\x$, pre-concatenation, 
while the bottom row is the trivial pre-coproduct $\vee^\F_{\!\frac{1}{2}}$
defined above.
The subspaces $\Lambda\times_M\Lambda\cong \F_{\!\frac{1}{2}}$
are both fixed by the retractions, which gives the commutativity of the
right column. The top and bottom squares in the left column commute by
naturality of the cap product, and the middle square because the classes $%
j^*\tau_{\CS}$ and $\tilde j^*\tau_{T}$  both identify with $(-1)^ne^*e_M$, compatibly under the isomorphism $\Lambda\times_M\Lambda \overset%
{\cong}{\longrightarrow }\F_{\!\frac{1}{2}}\subset \Lambda$.

Let $[A]=[A_0]+[A']\in H_*(\Lambda\times_M\Lambda,M)\op H_*(M)  \cong H_*(\F_{\!\frac{1}{2}},M)\op H_*(M)$. Going along the
top of the diagram computes $\wedge_{\Thom}^\x(j_*([A]))\in H_{*-n}(\La\x_M\La,M)\op H_*(M)$. 
On the other hand, by Lemma~\ref{lem:precop2},  going along the bottom of the diagram takes $[A]$ to  $(-1)^ne_M\cap [A_0]$. Together with the commutativity of the diagram, this gives the statement. 
\end{proof}

\begin{proof}[Proof of Theorem~\protect\ref{thm:trivial}]
  By definition, the coproduct $\widehat{\vee}$ has image inside $%
H_*(\Lambda\times_M\Lambda)\subset H_*(\Lambda\times \Lambda)$. Moreover, by
Theorem~\ref{thm:extendedcoproduct}~(B) its image does not intersect the
component of the constant loops $H_*(M\times M)\subset H_*(\Lambda\times
\Lambda)$, hence also not its diagonal subspace $H_*(M)\subset
H_*(\Lambda\times_M\Lambda)\subset H_*(\Lambda\times\Lambda)$. The result is
then a consequence of Proposition~\ref{prop:CSF2}.
\end{proof}

\subsection{Computations in the case of odd spheres (part II)}

\label{sec:computations}

We will now finish our computation of the coproduct when $M=S^n$ is an odd dimensional sphere, as started in Proposition~\ref{prop:spheres1}, and
investigate formulas involving the product and coproduct in that case.

Recall from Section~\ref{sec:compspheres1} the Chas-Sullivan ring
for $S^{n}$, $n\ge 3$ odd, as computed in the paper of Cohen, Jones and Yan 
\cite[Thm 2]{CJY}: 
\begin{equation*}
H_{\ast }(\Lambda S^{n})=\textstyle{\bigwedge} (A)\otimes \mathbb{Z}[U]
\end{equation*}
where $A\in H_{0}(\Lambda S^{n})$, and $U\in H_{2n-1}(\Lambda S^{n})$. This
ring has a unit $1=[M]\in H_{n}(\Lambda S^{n})$, represented by the constant
loops. 
We extend Proposition~\ref{prop:spheres1} to show that 

\begin{prop}
\label{prop:copspheres}The coproduct on $H_*(\Lambda S^n)$, for $n\ge 3$
odd, is given on generators by the formula 
\begin{align*}
\widehat{\vee}_{\Thom} (A\wedge U^{\wedge k}) = -\widehat\vee (A\wedge U^{\wedge k}) &=\hspace{5mm} \Sigma _{j=2}^{k-1}(A\wedge U^{\wedge j-1})\times (A\wedge U^{\wedge (k-j)})\\
\widehat{\vee}_{\Thom} U^{\wedge k}&=\hspace{5mm} \Sigma _{j=2}^{k-1}(A\wedge U^{\wedge j-1}\times
                                  U^{\wedge (k-j)}) + (U^{\wedge j-1}\times A\wedge U^{\wedge (k-j)})\\
   \widehat\vee U^{\wedge k}&=\hspace{5mm} \Sigma _{j=2}^{k-1}-(A\wedge U^{\wedge j-1}\times
                                  U^{\wedge (k-j)}) + (U^{\wedge j-1}\times A\wedge U^{\wedge (k-j)})
\end{align*}
\end{prop}

Before proving the proposition, we use it to confirm our result of the
previous section in the case of odd spheres:

\begin{cor}
\label{cor:spheresvw} In $H_*(\Lambda S^n)$ for $S^n$ an odd sphere with $%
n\ge 3$, we have that 
\begin{align*}
  \wedge\circ \widehat{\vee} (A\wedge U^{\wedge k}) &=0=\wedge\circ \widehat{\vee} (U^{\wedge k}),
\end{align*}
which confirms Theorem~\ref{thm:trivial} in the case of odd spheres.
\end{cor}

\begin{proof}[Proof of Corollary~\protect\ref{cor:spheresvw}]
The fact that $\wedge\widehat{\vee} (A\wedge U^{\wedge k})=0$ follows
from the formula obtained for $\widehat{\vee} (A\wedge U^{\wedge k})$ in
Proposition~\ref{prop:copspheres} together with the fact that $A\wedge
A=0$. For the other computations, we have that the degree $\deg(A\wedge U^{\wedge
j})=j(n-1)$ is even while the degree $\deg(U^{\wedge (k-1-j)})=(k-1-j)(n-1)-n$ is odd, and hence, after shifting their degree by $n$, we have that the first one is odd and the second even. As the product $\wedge$ is graded commutative after the degree shifting, we thus have that 
\begin{align*}
(A\wedge U^{\wedge j-1})\wedge U^{\wedge (k-j)} & = A\wedge U^{\wedge k-1}= U^{\wedge (k-j)}\wedge (A\wedge U^{\wedge j-1}) 
\end{align*}
Given that $k-j$ ranges between $k-2$ and $1$ while $j-1$ ranges between $1$ and $k-2$, the left and right terms in the formula for $\widehat{\vee} U^{\wedge k}$
obtained in Proposition~\ref{prop:copspheres} cancel in pairs when the product is applied,
 which proves the result.
\end{proof}

\begin{rem}
\label{rem:wedgepoc} Note that $\wedge\circ \widehat{\vee}_{\Thom} (U^{\wedge k}) = (2k-4)
A\wedge U^{\wedge k-1} \neq 0$ so  the formula $\wedge\circ \widehat{\vee}=0$ only holds if the product and coproduct are both taken with the algebraic choice of signs, or both with the Thom signs, and not with mixed choices of signs. 
\end{rem}

The first part of Proposition~\ref{prop:copspheres} follows directly from Proposition~\ref{prop:spheres1} as the lift $\widehat\vee=\vee$ when all the  classes involved are non-relative. We will deduce the second part of the proposition from computations made in \cite{GorHin} for the dual cohomology product. 
We recall from \cite[Sec 15]{GorHin} the relative cohomology ring $(H^*(\La S^n,S^n),\oast_{\GH})$ for an
odd sphere $S^n$, where the product $\oast_{\GH}=(-1)^{(n-1)q}\oast_{\Thom}=\oast_{\Thom}$ as $n$ is odd (see Remark~\ref{rem:GHsign}). This ring is generated by four classes: $\omega,
X,Y$ and $Z$ of degrees $n-1,2n-2, 2n-1$ and $3n-2$ respectively, and 
\begin{equation*}
H^{\ast }(\Lambda S^n,S^n) \ =\ \textstyle{\bigwedge} (Y,Z)\otimes \mathbb{Z}%
[\omega,X]/\sim
\end{equation*}
for $\sim$ generated by the two relations $X\oast_{\Thom} X=\omega^{\oast_{\Thom} 3}$ (where we note that $\omega^{\oast_{\Thom}k}$ makes sense because $\omega$ has even dimension and $n$ is odd, making $\oast_{\Thom}$ associative) and $X\oast_{\Thom} Y=\omega\oast_{\Thom} Z$.
As explained in \cite{GorHin}, this ring can also be written as $H^{\ast}(\Lambda S^n,S^n)=\textstyle{\bigwedge} (u)\otimes \mathbb{Z}[t]_{\geq 2}$
where $\omega=t^{2}$, $X=t^{3}$, $Y=u t^{2}$ and $Z=ut^{3}$. 
Here $\mathbb{Z}[t]_{\geq 2}$ denotes the ideal generated by $t^{2}$ and $t^{3}$ in the polynomial ring.

To deduce from this a computation of the homology coproduct, we need the
following compatibility between the homology and cohomology generators.

\begin{lem}
\label{lem:sphereskron} Let $n\ge 3$ be odd. The generators $A,U\in
H_*(\Lambda S^n)$ and $\omega,X,Y$ and $Z$ in $H^*(\Lambda S^n,S^n)$ can be
chosen so that for every $k\ge 1$,  
\begin{align*}
\big\lgl \omega^{\oast_{\Thom} k} , A\wedge U^{\wedge 2k-1}\big\rgl &= 1 \\
\big\lgl X\oast_{\Thom} \omega^{\oast_{\Thom} k-1} , A\wedge U^{\wedge 2k}\big\rgl &= 1 \\
\left\langle Y\circledast_{\Thom} \omega ^{\circledast_{\Thom} k-1},U^{\wedge
(2k-1)}\right\rangle & =1 \\
\left\langle Z\circledast_{\Thom} \omega ^{\circledast_{\Thom} k-1},U^{\wedge
(2k)}\right\rangle &=1
\end{align*}
\end{lem}

The case $k=1$ in the lemma can be taken as our definition of the cohomology
generators, once the homology generators are fixed. For $k\ge 2$ this is a
signed version of a general principle established in \cite[Sec 13,14]{GorHin}
saying that ``to go up one level in cohomology on a manifold with all
geodesics closed, you multiply by $\omega$'', and ``to go up one level in
homology, you multiply by $\mathbf{\Theta}$", $\mathbf{\Theta}$ being $%
U^{\wedge 2}$ in the case of spheres. See Remark~\ref{rem:levelsforspheres}
for a graphic representation of this phenomenon.

\begin{proof}[Proof of Lemma~\protect\ref{lem:sphereskron}]
We choose the orientation of $X$ and $\omega$ in such a way that 
\begin{equation*}
\langle \omega,A\wedge U\rangle=1 \ \ \text{ and }\ \ \langle X,A\wedge
U^{\wedge 2}\rangle=1. 
\end{equation*}
Note that this is compatible with the equation $X\oast_{\Thom} X=\omega^{\oast_{\Thom}3}$ \cite[(15.5.1)]{GorHin} 
as, using Proposition~\ref{prop:spheres1}, we have that $\vee_{\Thom}(A\wedge U^{\wedge 5})$ has exactly one term $A\wedge U^{\wedge 2}\ot A\wedge U^{\wedge 2}$, while
$\vee_{\Thom}^2(A\wedge U^{\wedge 5})$ has exactly one term $A\wedge U\ot A\wedge U\ot A\wedge U$.  

It follows inductively that the first two equations in the statement hold  as $\vee_{\Thom}(A\wedge U^{\wedge 2k-1})$ has exactly one term $A\wedge U\ot (A\wedge U^{\wedge 2k-3})$, while
the coproduct $\vee_{\Thom}(A\wedge U^{\wedge 2k-1})$ has exactly one term $(A\wedge U^{\wedge 2})\ot (A\wedge U^{\wedge 2k-3})$.

We then define $Y=\iota^!\iota^*\omega$ and $%
Z=\iota^!\iota^*X$ \cite[(15.6.2)]{GorHin} for $\iota\colon \Omega S^n\to
\Lambda S^n$ the inclusion.
It then follows that 
$$\begin{array}{rll}
\left\langle Y\circledast_{\Thom} \omega ^{\circledast_{\Thom} i},U^{\wedge
(2i+1)}\right\rangle &=\left\langle \iota^{!}\iota^{\ast }\omega
                         \circledast_{\Thom} \omega ^{\circledast_{\Thom} i},U^{\wedge (2i+1)}\right\rangle &\\
    &=\left\langle \iota^{!}(\iota^{\ast }\omega \circledast _{\Omega }\iota^{\ast
}\omega ^{\circledast_{\Thom} i}),U^{\wedge (2i+1)}\right\rangle & \\
&=\left\langle \omega ^{\circledast_{\Thom} i+1},\iota_{\ast }\iota_{!}U^{\wedge
(2i+1)}\right\rangle &=\left\langle \omega ^{\circledast_{\Thom} i+1},A\wedge U^{\wedge
(2i+1)}\right\rangle = 1
\end{array}$$
and 
$$\begin{array}{rll}
\left\langle Z\circledast_{\Thom} \omega ^{\circledast_{\Thom} i},U^{\wedge
(2i+2)}\right\rangle &=\left\langle \iota^{!}\iota^{\ast }X\circledast_{\Thom}
\omega ^{\circledast_{\Thom} i},U^{\wedge (2i+2)}\right\rangle &\\
&=\left\langle \iota^{!}(\iota^{\ast }X\circledast _{\Omega }i^{\ast
}\omega ^{\circledast_{\Thom} i}),U^{\wedge (2i+2)}\right\rangle &\\
&=\left\langle X\circledast_{\Thom} \omega ^{\circledast_{\Thom} i},i_{\ast }i_{!}U^{\wedge
(2i+2)}\right\rangle 
&=\left\langle X\circledast_{\Thom} \omega ^{\circledast_{\Thom} i},A\wedge U^{\wedge
(2i+2)}\right\rangle=1
\end{array}$$
where we used the Gysin formulas given in \cite[p151]{GorHin}.
\end{proof}

\begin{rem}
\label{rem:levelsforspheres} For the convenience of the reader, we recall
from \cite{GorHin} that the homology and cohomology of $\Lambda=\Lambda S^n$
(with $n\ge 3$ odd) fit together as follows: \newline

\smallskip

\noindent {\small $%
\begin{array}{lccccccccccc}
\text{degree} & 0 & n\!-\!1 & n & 2n\!-\!2 & 2n\!-\!1 & 3n\!-\!3 & 3n\!-\!2
& 4n\!-\!4 & 4n\!-\!3 & 5n\!-\!5 & 5n\!-\!4 \\ 
&  &  &  &  &  &  &  &  &  &  &  \\ 
H_{\ast }(M) & A &  & U^{0} &  &  &  &  &  &  &  &  \\ \cline{3-8}
H_{\ast }(\Lambda ^{\leq 2\pi },\Lambda ^{<2\pi }) &  & \multicolumn{1}{|c}{
AU} &  & AU^{2} & U &  & \multicolumn{1}{c|}{U^{2}} &  &  &  &  \\ 
\cline{3-12}
H_{\ast }(\Lambda ^{\leq 4\pi },\Lambda ^{<4\pi }) &  &  &  &  &  & 
\multicolumn{1}{|c}{AU^{3}} &  & AU^{4} & U^{3} &  & \multicolumn{1}{c|}{
U^4\!\!} \\ \cline{7-12}
H_{\ast }(\Lambda ^{\leq 6\pi },\Lambda ^{<6\pi }) &  &  &  &  &  &  &  &  & 
& \multicolumn{1}{|c}{AU^{5}} &  \\ \cline{11-12}
&  &  &  &  &  &  &  &  &  &  &  \\ \cline{3-8}
H^{\ast }(\Lambda ^{\leq 2\pi },\Lambda ^{<2\pi }) &  & \multicolumn{1}{|c}{
\omega=t^2} &  & X=t^3 & Y=ut^2 &  & \multicolumn{1}{c|}{\!Z=ut^3\!} &  &  & 
&  \\ \cline{3-12}
H^{\ast }(\Lambda ^{\leq 4\pi },\Lambda ^{<4\pi }) &  &  &  &  &  & 
\multicolumn{1}{|c}{\omega^{2}} &  & \omega X & \omega Y &  & 
\multicolumn{1}{c|}{\omega Z\!\!} \\ \cline{7-12}
H^{\ast }(\Lambda ^{\leq 6\pi },\Lambda ^{<6\pi }) &  &  &  &  &  &  &  &  & 
& \multicolumn{1}{|c}{\omega ^{3}} &  \\ \cline{11-12}
\end{array}%
$}

\smallskip

\noindent In each degree $k$ we have either that both $H_k(\Lambda S^n)$ and 
$H^k(\Lambda S^n)$ are 0, or both are equal to $\mathbb{Z}$. In the latter
case, we wrote a generator. To save space we have left out the products $%
\wedge$ and $\widehat{\oast}_{\Thom}=\oast_{\Thom}$ from the notation in the table; so for
example $AU^{2}:=A\wedge U^{\wedge 2}$ in the table.
The boxes in the table correspond to the homology of the unit
tangent bundle, shifted up by successive multiplication with $U^2$ in
homology and $\omega$ in cohomology.
\end{rem}

\begin{proof}[Proof of Proposition~\protect\ref{prop:copspheres}] 
The first statement follows from Proposition~\ref{prop:spheres1}, so we are left to check the second. 
 We have  $Y\U\circledast\omega ^{\U\circledast k-1}=u t^{2k}$ and $%
Z\U\circledast \omega ^{\U\circledast k-1}=ut^{2k+1}$, where we write $\U \oast = \circledast_{\Thom} $ for readability. So Lemma~\ref{lem:sphereskron} shows that $U^{\wedge k}$ is dual to $ut^{ k+1}$. This last product can be decomposed as 
\begin{equation*}
(u\U\oast t^{\U\oast j+1})\U\oast t^{\oast k-j}=t^{\U\oast k-j}\U\oast (u\U\oast t^{\U\oast j+1})
\end{equation*}
for any $1\le j\le k-2$ as $\U\oast$ is graded commutative after shifting the degrees by $n$ because $n$ is odd, and $u\U\oast t^{\U\oast j+1}$ has odd degree $(j+1)(n-1)+1$, and hence even $n$-shifted degree. 
Hence 
\begin{equation*}
\widehat{\vee} U^{\wedge k}=\Sigma _{j=1}^{k-2}\pm (A\wedge U^{\wedge
j}\times U^{\wedge (k-1-j)}) + \pm (U^{\wedge j}\times A\wedge U^{\wedge
(k-1-j)}).
\end{equation*}
Left is to work out the signs for each term. We have already seen that there
are no signs coming from the decomposition as products. There two potential
additional two signs coming from the duality:

\begin{enumerate}[(i)]

\item $\langle c\times d, \vee E\rangle = (-1)^{\deg(d)(n-1)}\langle c\oast %
d , E\rangle$,

\item $\langle c\times d , C\times D\rangle =(-1)^{\deg(D)\deg(c)}\langle
c,C\rangle \langle d,D\rangle$.
\end{enumerate}

Now in the case at hand, we have assumed that $n$ is odd, so the first
equation does not give a sign. Also, in all cases either $C$ or $D$ will be
a term of the form $A\wedge U^{\wedge j}$, which is of degree $j(n-1)$, that
is always even dimensional when $n$ is odd. So the second equation will
likewise not give any sign.
\end{proof}

\begin{rem}[Higher genus non-trivial operations] 
\label{rem:genus1} If we represent the product and coproduct by pairs of
pants, the composition of the coproduct followed by the product is
represented by a genus one surface with two boundary components. We have
shown above that this genus one operation is trivial. If one inserts the operation $(1\times \Delta)$ in between
the product and the coproduct, for $\Delta$ the degree 1 operation coming
from the circle action, the resulting operation is non-trivial: for $n\ge 3$ odd, the operation 
\begin{equation*}
t:=\wedge\circ (1\times \Delta) \circ \widehat{\vee}\colon H_*(\Lambda
S^n)\longrightarrow H_{*-2n+2}(\Lambda S^n)
\end{equation*}
is non-trivial on the classes $A\wedge U^{\wedge k}$ with $k\ge 3$. Indeed, from Proposition~\ref%
{prop:copspheres}, we have that 
\begin{equation*}
\widehat{\vee} (A\wedge U^{\wedge k})=\Sigma _{j=1}^{k-2}(A\wedge U^{\wedge
j})\times (A\wedge U^{\wedge (k-1-j)}).
\end{equation*}
From \cite[Lem 6.2]{HinRad}, we know moreover that 
\begin{equation*}
\Delta(A\wedge U^{\wedge k})=(-1)^kk\,U^{\wedge (k-1)}
\end{equation*}
(taking into account our sign convention at the beginning of Section~\ref{sec:compspheres1} compared to \cite{HinRad}, see the proof of the lemma). Putting these two computations
together we get that 
\begin{align*}
\wedge\circ (1\times \Delta) \circ \widehat{\vee} (A\wedge U^{\wedge
k})&=\Sigma_{j=1}^{k-2}(-1)^{k-1-j}(k-1-j)A\wedge U^{\wedge (k-2)} \\
&=(-1)^k\lfloor\frac{k-1}{2}\rfloor A\wedge U^{\wedge (k-2)}
\end{align*}
which in particular is non-zero. 

The operation $t$ would still be associated to a genus 1 surface, and
composing $t$ with itself $g$ times can be associated to a genus $g$
surface. Note also that the operation $t^g$ is likewise non-trivial on $%
A\wedge U^{\wedge k}$ whenever $k\ge 2g+1$. Similar computations in
algebraic models of the loop space can be found in \cite[Prop 4.1]{Wah16}
(see also \cite{BoeEga}), and in fact, these papers indicate that there
should be many such non-trivial operations combining the product, coproduct
and $\Delta$ operations, associated to classes in the homology of the
Harmonic compactification of the moduli space of Riemann surfaces. (In this compactification, the class $t$ corresponds to a 2-parameter family of genus 1 surfaces made out of two pairs of pants glued along two circles $C_1$ and $C_2$, with one parameter changing the gluing along $C_2$ by a rotation and the other parameter varying the relative size of $C_1$ and $C_2$ while fixing the sum of the lengths to be 1 at all times.)  These more general operations
will be studied in our following paper \cite{HinWah2}.
\end{rem}

\begin{rem}[Failure of the Frobenius formula] 
\label{rem:Frob} It has been suggested that there should be a formula of the
form%
\begin{equation}  \label{equ:66}
\vee (A\wedge B)=(\vee A)\wedge B+A\wedge (\vee B)
\end{equation}
i.e. 
\begin{equation*}
\vee \circ \wedge \overset{?}{=}(1\times \wedge )\circ \vee \times 1+(\wedge
\times 1)\circ 1\times \vee
\end{equation*}
\cite[p 349]{Sul04} relating the coproduct and the product. It is very
difficult to make sense of this formula in relative homology. It is natural
to ask if the product and \textit{lifted} coproduct satisfy 
\begin{equation}  \label{equ:Sast}
\widehat{\vee }\circ \wedge =(1\times \widehat{\wedge })\circ \vee \times 1+(%
\widehat{\wedge }\times 1)\circ 1\times \vee
\end{equation}

Our choice of lift $\widehat{\vee }$ does not satisfy the formula (\ref%
{equ:Sast}): one can check by a computation similar to that in Remark~\ref%
{rem:genus1} for odd spheres. If the formula (\ref{equ:Sast}) were true we
would have for example 
\begin{equation}  \label{equ:Sastast}
\widehat{\vee }(\mathbf{\Theta} \wedge \mathbf{\Theta} )=\widehat{\vee }(%
\mathbf{\Theta} )\wedge \mathbf{\Theta} +\mathbf{\Theta} \wedge \widehat{%
\vee }(\mathbf{\Theta} )
\end{equation}
where $\mathbf{\Theta} :=U^{\wedge 2}\in H_{3n-2}(\Lambda S^{n})$. Now the
right hand side is $0$ since $\widehat{\vee }(\mathbf{\Theta} )=0$ by
Theorem~\ref{thm:trivial}; indeed $\mathbf{\Theta} $ is represented by the
space of circles and has support in the simple and constant loops. But the
left hand side of (\ref{equ:Sastast}) must \textit{not} be $0$ since $%
\mathbf{\Theta} \wedge \mathbf{\Theta} $ is not in the image of $\Lambda
^{\le L }$ and thus $\mathbf{\Theta} \wedge \mathbf{\Theta} $ is dual to a
decomposable cohomology class: The coproduct $\widehat{\vee }(\mathbf{\Theta%
} \wedge \mathbf{\Theta} $ has one nonzero term for each way of writing the
dual $u\circledast t^{\circledast 5}$ of $\mathbf{\Theta} \wedge \mathbf{%
\Theta} $ as a product. The conclusion is that our lifted coproduct fails
to satisfy (\ref{equ:Sastast}) for fundamental reasons.

More generally, using direct computation and testing $A=B=U^{\wedge 2k}=%
\mathbf{\Theta} ^{\wedge k}$ in (\ref{equ:66}), we did not find equality,
but rather that the two sides differ by 
\begin{eqnarray*}
&&\widehat{\vee }(U^{\wedge 2k}\wedge U^{\wedge 2k})-\left( (\widehat{\vee }%
U^{\wedge 2k})\wedge U^{\wedge 2k}+U^{\wedge 2k}\wedge (\widehat{\vee }%
U^{\wedge 2k})\right) \\
&=&-A\wedge U^{\wedge (2k-1)}\times U^{\wedge 2k}-A\wedge U^{\wedge 2k}\times
U^{\wedge (2k-1)}+U^{\wedge (2k-1)}\times A\wedge U^{\wedge 2k}+U^{\wedge
2k}\times A\wedge U^{\wedge (2k-1)}
\end{eqnarray*}%
yielding four nonzero terms which do not add to $0$. It is also notable that
the four terms are ``in the middle". From a naive geometric perspective, (%
\ref{equ:66}) is not to be expected to be true: it says that, ``the
self-intersections in $A\wedge B$, (i.e.~$\vee (A\wedge B)$) come from the
self-intersections in $A$ (i.e.~$(\vee A)\wedge B$) and from the
self-intersections in $B$ (i.e.~$A\wedge (\vee B)$)''. But what about the
self-intersections in $A\wedge B$ of the form $\alpha \ast \beta $ where $%
\alpha \in \func{Im}A$ and $\beta \in \func{Im}B$ are both simple? The
evidence from the finite dimensional approximation is that the difference
between the left- and right-hand side in equation (\ref{equ:66}) picks up
the second-order intersections of $A$ and $B$. 
\end{rem}

\section{Application to spheres and projective spaces}

\label{sec:int}

The iterated coproduct ${\vee}^k$ and its lifted version $\widehat\vee^k$,  via Theorems~\ref%
{thm:coproduct2} and \ref{thm:extendedcoproduct}, gives us a way to study the following geometric
invariant of homology classes in the loop space:

\begin{Def}
Given $[X]\in H_{\ast }(\Lambda )$, we define the \textit{intersection
multiplicity} of $[X]$ as follows: 
\begin{equation*}
\func{int}([X]):=\inf_{\begin{subarray}{c}A\in C_*(\La)\\
[A]=[X]\end{subarray}}\sup_{\substack{ \gamma \in \func{Im}A  \\ \ell(\gamma )>0  \\ p\in M}}\#(\gamma ^{-1}\{p\}).
\end{equation*}
and the \textit{basepoint intersection multiplicity} of $[X]$ by 
\begin{equation*}
\func{int}_{0}([X]):=\inf_{\substack{ A\in C_*(\Lambda)  \\ [A]=[X]}}\sup 
_{\substack{ \gamma \in \func{Im}A  \\ \ell(\gamma )>0}}\#(\gamma
^{-1}\{\gamma (0)\}).
\end{equation*}
\end{Def}

Note that we are not counting the number of pairs $(s,t)$ with $\gamma
(s)=\gamma (t)$, but rather for a fixed point $p$ how many times the loop $%
\gamma $ goes through $p$. A representative $A$ for $[X]$ consisting
entirely of piecewise geodesic loops, each with at most $N$ pieces and each
piece of length $\leq \rho /2$, and parametrized proportional to arc
length, will have finite intersection multiplicities, since a given point $p$
intersects each piece at most once. As such representatives always exist,
the intersection multiplicity is necessarily a finite number. Also, by
definition and by Theorem~\ref{thm:extendedcoproduct}
\begin{equation}  \label{equast}
\func{int}([X])\leq k\ \ \Longrightarrow\ \ \func{int}_{0}([X])\leq k\ \
\Longrightarrow\ \ \widehat{\vee }^{k}[X]=0.
\end{equation}

The following result shows that the reverse implications hold in the case of
spheres and projective spaces.

\begin{thm}
\label{thm:int}If $M=S^{n}$, $\mathbb{R}P^{n}$, $\mathbb{C}P^{n}$, $\mathbb{H}P^{n}$, or $\mathbb{O}P^{2}$, then for any $[X]\in H_{\ast }(\Lambda )$, and $k\geq 1$, 
\begin{equation}  \label{equastast}
\func{int}([X])\leq k\ \ \Longleftrightarrow\ \ \func{int}_{0}([X])\leq k\ \
\Longleftrightarrow\ \ \widehat{\vee}^{k}[X]=0.
\end{equation}
\end{thm}

Thus for these manifolds the vanishing of the iterated coproduct is a
perfect predictor of the intersection multiplicity of each homology class $%
[X]$, except that it cannot be used to distinguish between intersection
multiplicity $0$ and $1$: we have that $\widehat{\vee }^{1}=0$ on both $%
H_{\ast }(M)$ (where $\func{int}=\func{int}_{0}=0$ by definition) and on $%
H_{\ast }(\Lambda ^{\leq L})-H_{\ast }(M)$ (where $\func{int}=\func{int}%
_{0}=1$), for $\Lambda^{\leq L}$ the subspace of loops of energy at most $L^2
$.

\smallskip

We will show below that (\ref{equastast}) follows from the following two
hypotheses:

\begin{enumerate}
\item[(A)] $M$ is a compact Riemannian manifold all of whose geodesics are
closed and of the same minimal period $L$,   and

\item[(B)] $H_{\ast }(\Lambda ^{\leq L})$ is supported on the union of the
simple and the constant loops.
\end{enumerate}

Ultimately the proposition is a topological statement, so (\ref{equastast})
will also follow if $M$ is a compact smooth manifold that carries a metric
satisfying (A) and (B). The manifolds listed in the statement are rather special; they are the compact rank 1 symmetric spaces, see eg., \cite[Chap 3]{Besse}. 

\begin{lem}
Properties (A) and (B) hold for spheres and projective spaces with their
standard metric.
\end{lem}

\begin{proof}
The standard metric on a sphere or projective space has the property (A) 
\cite[3.31,3.70]{Besse}. We will show that (B) is also satisfied for the standard
metric on these spaces. Recall that a circle on the sphere $S^{n}\subset 
\mathbb{R}^{n+1}$ is a nonempty intersection of $S^{n}$ with a $2$-plane,
considered here as parametrized injectively and proportional to arc length
on $S^{1}$. The prime closed geodesics are intersections with $2$-planes
through the origin, and have length $L=2\pi$. 

If $M=\mathbb{K}P^{n}$ is a projective space,  with $\mathbb{K}=\mathbb{R}, \mathbb{C}$, $\mathbb{H}$ or $\mathbb{O}$, with the standard metric, then every
prime geodesic is closed and of length $L=2\pi$, and lies in a unique totally geodesic
projective line $\mathbb{K}P^1$ of constant curvature 1, where  $\mathbb{K}P^1\cong S^q$ of real dimension $q=1, 2, 4$, or $8$. (See eg.~\cite[Thm 5.2.1]{Kli78}.) 
A circle on $M$ is by definition a circle on such a projective line.

Define
\begin{equation*}
\Theta :=\{\text{circles}\}\subset \Lambda ^{\leq L }
\end{equation*}
for $L$ the minimal period of geodesics.
The simple closed geodesics on $M$ are a submanifold  $\Si\subset \Theta$, and 
they constitute a Morse-Bott nondegenerate submanifold for the
energy function on $\Lambda$ \cite{Ziller,Morse}. Each circle is simple or
trivial and has length $\leq L $, so it is enough to show that $H_{\ast
}(\Lambda ^{\leq L })$ is supported on $\Theta$. Because $\Theta$ is an
embedded submanifold, lying below level $L $ except along the critical
submanifold of great circles, and of dimension equal to the sum of the index
($n-1$ for the spheres, resp.~$q-1$ for the projective spaces, see \cite[p11]{Ziller}) and the dimension of the space $\Si$ of simple closed geodesics,
we have an isomorphism, for any coefficients, 
\begin{equation*}
H_{\ast }(\Theta ,\Theta ^{<L })\overset{\cong}{\longrightarrow}H_{\ast
}(\Lambda^{\le L} ,\Lambda ^{<L })
\end{equation*}%
by \cite[Thm D2]{GorHin} with $V=\Theta$ and $X=\La$, using that $\La^{<L}\cup \Si=\La^{\le L}$. We also have 
\begin{equation*}
H_{\ast }(\Theta ^{=0})\overset{\cong}{\longrightarrow}H_{\ast }(M)
\end{equation*}%
Putting these together, using the fact there are no critical points with
length in $(0,L)$, we see that the inclusion of $\Theta $ in $\Lambda $
induces an isomorphism 
\begin{equation*}
H_{\ast }(\Theta )\overset{\cong}{\longrightarrow }H_{\ast }(\Lambda ^{\leq
L })
\end{equation*}
which gives (B).
\end{proof}

\begin{proof}[Proof of Theorem~\protect\ref{thm:int}]
Let $M$ be a sphere or projective space. By the lemma we may assume that $M$
satisfies properties (A) and (B). We will derive (\ref{equastast}) from
these two properties. By Theorem~\ref{thm:extendedcoproduct}, we only need
to check the implications 
\begin{equation*}
\widehat{\vee}^k[X]=0 \ \ \Longrightarrow \ \ \func{int}_{0}([X])\leq k \ \ 
\text{and} \ \ \func{int}([X])\leq k.
\end{equation*}
So assume $k\geq 1$ and $\widehat{\vee }^{k}[X]=0.$ We claim that $[X]$ is
supported on $\Lambda ^{\leq kL }$. Assume otherwise; then the image of $[X]$
in $H_{\ast }(\Lambda ,\Lambda ^{\leq kL })$ is nonzero. Let $\FF$ be a field
so that the image of $[X]$ in $H_{\ast }(\Lambda ,\Lambda ^{\leq kL };\FF)$ is
nonzero, and let $[x]\in H^{\ast }(\Lambda ,\Lambda ^{\leq kL };\FF)$ be so
that the Kronecker product satisfies 
\begin{equation*}
\langle [x],[X]\rangle \neq 0.
\end{equation*}
By \cite[Thm 14.2, Cor 14.8]{GorHin}, the class $[x]$ is a finite sum of
terms of the form $[x_{1}]\circledast \dots\circledast [x_{j}]=[x_{1}]\widehat{\circledast }\cdots \widehat{\circledast }[x_{j}]$,
where $[x_{i}]\in H^{\ast }(\Lambda ^{\leq L },M;\FF)$ and where $j>k$. But then for
some term we have

\begin{equation*}
0\neq \langle [x_{1}]\widehat{\circledast }\cdots\widehat{\circledast }%
[x_{j}],[X]\rangle=\langle [x_{1}]\otimes \dots\otimes [x_{j}],\widehat{\vee 
}^{j-1}[X]\rangle
\end{equation*}%
which is a contradiction since $j>k$ and $\widehat{\vee }^{k}[X]=0$. So $[X]$
is supported on $\Lambda ^{\leq kL }$.

By \cite[Prop 5.3, Thm 13.4]{GorHin}, the Chas Sullivan ring is generated by 
$H_{\ast }(\Lambda ^{\leq L })$. Moreover, $H_{\ast }(\Lambda ^{\leq kL })$, 
$k>1$, is supported on 
\begin{equation*}
\mathbf{\Theta}^{\wedge k}=\{\gamma _{1}\star \gamma _{2}\star \cdots\star
\gamma _{k} \ |\ \gamma _{i}\in \Theta \text{ and \ }\gamma
_{1}(0)=\dots=\gamma _{k}(0)\}.
\end{equation*}
It follows that $X$ has a representative supported in $\mathbf{\Theta
^{\wedge k}}$, and thus%
\begin{equation*}
\func{int}(X)\leq \func{int}(\mathbf{\Theta}^{\wedge k}):=\sup_{\substack{ 
_{\substack{ \gamma \in \mathbf{\Theta}^{\wedge k}  \\ \ell (\gamma
)>0}}  \\ p\in M}}\#(\gamma ^{-1}\{p\}).
\end{equation*}
The reader can check that if $\gamma \in \mathbf{\Theta}^{\wedge k}$, and $%
\gamma $ is not constant, then $\gamma $ is the $\star$-concatenation of $j$
nontrivial circles, for some $j$ with $1\leq j\leq k$, and thus that $\gamma
^{-1}\{p\}$ consists of at most $k$ points, so $\func{int}(\mathbf{\Theta}
^{\wedge k})=k$. As the same holds for $\func{int}_0$, this proves the
proposition.
\end{proof}

\begin{rem}\label{rem:simplemodel}
If $M$ is a sphere or projective space, the space of circles 
\begin{equation*}
\cup _{k\geq 1}\mathbf{\Theta}^{\wedge k}\subset \Lambda M
\end{equation*}%
is a small simple space carrying the topology of $\Lambda M$ in the sense
described above. Unfortunately it is not invariant under the natural $S^{1} $
action on $\Lambda M$. (See Remark~\ref{rem:equsimple} for a discussion of equivariant models.) 
\end{rem}

\begin{ex}[Theorem~\protect\ref{thm:int} in the case of odd dimensional
spheres]
Let $M=S^{n}$ with $n$ odd and greater than $2$. We use again the notation
from Section~\ref{sec:computations} so that 
\begin{equation*}
H_{\ast }(\Lambda S^{n})=\wedge (A)\otimes \mathbb{Z}[U];
\end{equation*}
thus every homology class is a multiple of $A\wedge U^{\wedge k}$, $k\geq 0$
(in dimension $k(n-1)$), or a multiple of $U^{\wedge k}$ , $k\geq 0$\ (in
dimension $(k+1)n-k$). In the metric in which all geodesics are closed with
minimal length $L$, the homology classes $A\wedge U^{\wedge (2m-1)}$, $%
A\wedge U^{\wedge 2m}$, $U^{\wedge (2m-1)}$, and $U^{\wedge 2m}$ are "at
level $m$", that is, they are supported on $\Lambda ^{\leq mL }$ (but not on 
$\Lambda ^{(m-1)L }$); see \cite[15.2 and Fig 9]{GorHin}. Thus, by the proof
of Theorem~\ref{thm:int}, they are also supported on the cycle $\mathbf{%
\Theta }^{\wedge m}$, the space of all concatenations of $m$ circles with a
common basepoint. It follows that they have representatives with at most $m$%
-fold intersections, and $\vee ^{m}$ should annihilate each of these
classes. Using Proposition~\ref{prop:copspheres} we compute: We have 
\begin{equation*}
\vee ^{m-1}(A\wedge U^{\wedge 2m-1})=(A\wedge U)^{\times m}\neq 0
\end{equation*}%
but 
\begin{equation*}
\vee ^{m}(A\wedge U^{\wedge 2m-1})=0\text{.}
\end{equation*}
Also 
\begin{equation*}
\vee ^{m-1}(A\wedge U^{\wedge 2m})=\Sigma _{j=0}^{m-1}(A\wedge U)^{\times
(m-(j+1))}\times (A\wedge U^{\wedge 2})\times (A\wedge U)^{\times j}\neq 0
\end{equation*}%
but 
\begin{equation*}
\vee ^{m}(A\wedge U^{\wedge k})=0.
\end{equation*}%
The reader can likewise check that $\vee ^{m-1}(U^{\wedge (2m-1)})\neq 0$
and $\vee ^{m-1}(U^{\wedge 2m})\neq 0$ but $\vee ^{m}(U^{\wedge 2m-1})=\vee
^{m}(U^{\wedge 2m})=0$. Thus for each of the four non-trivial homology
classes $Z$ at level $m$, we have 
\begin{equation*}
\vee ^{m}Z=0
\end{equation*}
which confirms (\ref{equast}), and also 
\begin{equation*}
\vee ^{m-1}Z\neq 0.
\end{equation*}
It follows that 
\begin{equation*}
\vee ^{m}Z=0\ \Longleftrightarrow \ Z\text{ is supported on }\Lambda ^{\leq
mL }\ \Longleftrightarrow \ \text{int}([Z])\leq m.
\end{equation*}%
This confirms Theorem~\ref{thm:int} in the case of an odd sphere $S^{n}$, $%
n>2$. Note that the level of a homology class is not monotone in the degree;
the class of lowest degree at level $m+1$ is in a lower degree than the
class in highest degree in level $m$: $\ A\wedge U^{\wedge (2m+1)}$ is in
dimension $(2m+1)(n-1)$ and level $m+1$ but $U^{\wedge 2m}$ is in dimension $%
(2m+1)(n-1)+1$ and level $m$.
\end{ex}

\appendix

\section{Cap product}

\label{app:cap}

Because the cap product is fundamental to our paper, we review its
properties, in the exact terms we will be using them.

\subsection{Relative cap product and its naturality}
Let $X_0,X_1\subset X$ and $Y_0,Y_1\subset Y$ be spaces and (possibly empty)
subspaces satisfying that $C_*(X_0)+C_*(X_1) = C_*(X_0\cup X_1)$ and the
same for $Y_0,Y_1$. Suppose $f\colon (X,X_0,X_1)\to (Y,Y_0,Y_1)$ is a continuous
map from $X$ to $Y$ taking $X_i$ to $Y_i$ for $i=0,1$. Then the cap product
and the map $f$ define maps 
\begin{equation*}
\xymatrix{ C^q(X,X_0) \ot C_p(X,X_0\cup X_1) \ar[r]^-\cap
\ar@<+8ex>[d]_{f_*} & C_{p-q}(X,X_1)\ar[d]^{f_*} \\ C^q(Y,Y_0) \ot
C_p(Y,Y_0\cup Y_1) \ar[r]^-\cap \ar@<+9ex>[u]_{f^*} & C_{p-q}(Y,Y_1). }
\end{equation*}
Naturality of the cap product then says that for any $\alpha\in C^q(Y,Y_0)$
and $B\in C_p(X,X_0\cup X_1)$, we have that 
\begin{equation}  \label{equ:cap}
\alpha \cap f_*(B)= f_*(f^*(\alpha)\cap B).
\end{equation}
In particular for a map $f\colon X\to Y$ and a class $\alpha\in C^k(Y)$, a diagram
of the form 
\begin{equation*}
\xymatrix{C_*(X) \ar[d]_{f^*(\al)\cap }\ar[r]^{f_*} & C_*(Y)\ar[d]^{\al
\cap}\\ C_{*-k}(X) \ar[r]^{f_*} & C_{*-k}(Y) }
\end{equation*}
always commutes.

Slightly more generally, suppose that $D_*\le C_*(X)$ is a subcomplex
satisfying that for any $A\in D_p$, and any $q<p$, the restriction $%
A|_{\sigma_q}$ to the front $q$-face of $\sigma_p$ is in $D_q$. Then the
cap product $\cap\colon C^q(X,X_0)\otimes C_p(X,X_0) \to C_{p-q}(X)$
descends to a product 
\begin{equation*}
\cap\colon C^q(X,X_0)\otimes C_p(X,X_0)/D_p \to C_{p-q}(X)/D_{p-q}
\end{equation*}
where $C_p(X,X_0)/D_p$ is by definition $C_p(X)/(C_p(X_0)+D_p)$. The
standard relative cap product described above is the case when $D_*=C_*(X_1)$
with $C_*(X_0)+C_*(X_1) = C_*(X_0\cup X_1)$, an assumption that we now see
can be dropped by replacing $C_*(X,X_0\cup X_1)$ by $C_*(X)/\big(%
C_*(X_0)+C_*(X_1)\big)$ in the definition of the relative cap product.

In the paper, we will use the standard case, and the case $%
D_*=C_*(X_1)+C_*(X_2)$ for $X_1,X_2\subset X$ that do not necessarily
satisfy $C_*(X_0)+C_*(X_1) = C_*(X_0\cup X_1)$. Because this generalized cap
product is defined as a quotient of the classical cap product, it has the
same naturality properties.

\subsection{Cap product using small simplices}\label{sec:smallsimp}
To spaces $U_0\subset U_1\subset X$ such that the interiors of $U_0^c$ and $%
U_1$ cover $X$, we can associate the chain complex $C^{\mathfrak{U}}_*(X)$
of \emph{small simplices with respect to $\mathfrak{U}=\{U_0^c,U_1\}$},
whose $p$-chains are formal sums of maps $\sigma_p\colon \De^p\to X$ whose image lies
entirely inside either $U_0^c$ or $U_1$. The inclusion $C_*^{\mathfrak{U}%
}(X)\hookrightarrow C_*(X)$ is a chain homotopy equivalence, with explicit
chain homotopy inverse 
\begin{equation*}
\rho\colon C_*(X) \longrightarrow C^{\mathfrak{U}}(X)
\end{equation*}
given e.g.~in \cite[Prop 2.21]{Hatcherbook}, using subdivisions.

Now given a class $\alpha\in C^k(U_1,U^c_0)$, we will write 
\begin{equation*}
[\alpha\;\cap]  \colon C_*(X) \longrightarrow C_{*-k}(U_1)
\end{equation*}
for the composition of maps 
\begin{equation}  \label{equ:capdef}
[\alpha\;\cap]  \colon C_*(X) \twoheadrightarrow C_*(X,U_0^c) \overset{\rho}{\longrightarrow }C^{%
\mathfrak{U}}_*(X,U_0^c) \twoheadrightarrow C_*(U_1,U_0^c) \xrightarrow{\alpha\cap} C_{*-k}(U_1)
\end{equation}
where the middle two maps give an explicit chain inverse for excision, the
map $\rho$ being as above and the following map projecting away the
simplices not included in $U_1$.

In this language, naturality of the cap product becomes the following: 

\begin{lem}\label{lem:natdiagram}
  Let $A\subset U_0\subset U_1\subset X$ and $B\subset V_0\subset
  V_1\subset Y$ be tuples of spaces with $X=\mathring U_0^c\cup \mathring U_1$ and $Y=\mathring V_0^c\cup \mathring V_1$.  Let 
  $f\colon (X,A,U_0,U_1)\to (Y,B,V_0,V_1)$ and fix $\alpha\in C^k(V_1,V_0^c)$.
Then the diagram 
  \begin{equation*}  
\xymatrix{H_*(X,A) \ar[d]_{[f^*(\al)\cap] }\ar[r]^{f_*} & H_*(Y,B)\ar[d]^{[\al\cap]}\\ H_{*-k}(U_1,A) \ar[r]^{f_*} & H_{*-k}(V_1B) }
\end{equation*}
commutes. 
\end{lem}

Note the above naturality statement does not strictly hold on the level of chains because the map $\rho\colon C_*(X)\to C_*^{\mathfrak{U}}(X)$ that subdivides chains is not natural before going to homology. 

\begin{proof} 
  We need to show that the diagram
  $$\xymatrix{ H_*(X,A) \ar[r]\ar[d]_f &H_*(X,U_0^c\cup A) \ar[r]^-{\rho} \ar[d]^f& H^{\mathfrak{U}}_*(X,U_0^c\cup A) \ar[r] \ar[d]^f& H_*(U_1,U_0^c\cup A) \ar[r]^-{\alpha\cap} \ar[d]^f& H_{*-k}(U_1,A) \ar[d]^f\\
    H_*(Y,B) \ar[r] & H_*(Y,V_0^c\cup B) \ar[r]^-{\rho} & H^{\mathfrak{U}}_*(Y,V_0^c\cup B) \ar[r] & H_*(V_1,V_0^c\cup B) \ar[r]^-{\alpha\cap} & H_{*-k}(V_1,B)
  }$$
 commutes, where the second and fourth squares need some attention. The second square commutes because the map $\rho$ is a homotopy inverse to the inclusion, which is natural, and the last square commutes by the standard naturality of the cap product. 
\end{proof}

We give now an example of how such a product will appear in the present
paper.

\begin{ex}
\label{ex:cap} Fix $\varepsilon>\varepsilon_0>0$ and let $X=M^2$ with $%
U_0=U_{M,\varepsilon_0}$ and $U_1=U_{M,\varepsilon}$ for 
\begin{equation*}
U_M=U_{M,\varepsilon}=\{(x,y)\in M^2\ |\ |x-y|<\varepsilon\}\ \subset\ M^2
\end{equation*}
as above, and the same for $\varepsilon_0$. Recall from Section~\ref{tubsec}
that we have picked a Thom class $u_M$ for our manifold $M$ such that $%
\tau_M:=\kappa_M^*u_M$ vanishes on chains supported on the complement of $%
U_{M,\varepsilon _{0}}$, i.e., $\tau_M \in
C^n(M^2,U_{M,\varepsilon_0}^c)$. We will also write 
\begin{equation*}
\tau_M \in C^n(U_{M},U_{M,\varepsilon_0}^c)
\end{equation*}
for its restriction to $U_{M}$. Taking $(X,U_0,U_1)=(M^2,U_{M,\varepsilon_0},U_{M})$ and $\alpha=\tau_M$, the sequence of maps
(\ref{equ:capdef}) defines a map 
\begin{equation*}
[\tau_M\cap] \colon C_*(M^2)\longrightarrow C_{*-n}(U_{\! M}).
\end{equation*}

\smallskip

Consider now also the space $Y=\Lambda^2$, with 
\begin{equation*}
V_1=U_{\CS}=U_{\CS,\varepsilon}:=\{(\gamma,\lambda)\in \Lambda^2\ |\
|\gamma(0)-\lambda(0)|<\varepsilon\}\ \subset \ \Lambda^2
\end{equation*}
and likewise $V_0=U_{\CS,\varepsilon_0}$. Evaluating each loop at $0$ defines a
map 
\begin{equation*}
e\times e\colon (\Lambda^2,V_0,V_1) \longrightarrow (M^2,U_0,U_1)
\end{equation*}
and pulling back the class $\tau_M$ along this map defines a class 
\begin{equation*}
\tau_{\CS}:=(e\times e)^*\tau_M\in C^n(\Lambda^2,U_{\CS,\varepsilon_0}^c)
\end{equation*}
which again can be restricted to $U_{\CS}$. Applying (\ref{equ:capdef}) to this case gives a map 
\begin{equation*}
[\tau_{\CS}\cap]\colon C_*(\Lambda^2)\longrightarrow C_*(U_{\CS}).
\end{equation*}
Moreover, Lemma~\ref{lem:natdiagram} gives that the diagram 
\begin{equation*}
\xymatrix{H_*(\La^2) \ar[d]_{[\tau_{\CS}\cap] }\ar[rr]^-{e\x e} &&
H_*(M^2)\ar[d]^{[\tau_M \cap]}\\ H_{*-k}(U_{\CS}) \ar[rr]^-{e\x
e} && H_{*-k}(U_{M})}
\end{equation*}
commutes.
\end{ex}

\subsection{Commuting the cap and cross products}

We recall here how the cap and cross products interact, for easy reference when we use it:
Suppose $\alpha\in H^*(X)$, $\beta\in H^*(Y)$, $a\in H_*(X)$ and $b\in H_*(Y)$. Then
\begin{equation}\label{equ:capcross}
  (\alpha\x \beta)\cap (a\x b)=(-1)^{|\beta||a|}(\alpha\cap a)\x (\beta\cap b)\in H_*(X\x Y).
  \end{equation}
 (See eg.~\cite[VI. Thm 5.4]{Bredon}.)

\section{Signs in the intersection product}

\label{app:signs} The intersection product $H_*(M)\otimes H_*(M) \to
H_{*-n}(M)$ can be defined using Poincar{\'e} duality, or using the above
Thom collapse map $\kappa_M$: we write
$$\bullet_{\Poin}\colon H_p(M)\ot H_q(M)\xrightarrow{D} H^{n-p}(M)\ot H^{n-q}(M)\xrightarrow{\x} H^{2n-p-q}(M^2) \xrightarrow{\De^*}H^{2n-p-q}(M)\xrightarrow{D} H_{p+q-n}(M)$$
for the definition using Poincar\' e duality and
$$\bullet_{\Thom}\colon H_p(M)\ot H_q(M)\xrightarrow{\x}  H_{p+q}(M^2)\xrightarrow{[\tau \cap]}  H_{p+q-n}(U) \xrightarrow{r}  H_{p+q-n}(M)$$
for the definition using the Thom collapse, where $r\colon U\cong TM \to M$ denotes the retraction of the tubular embedding. 

As the following proposition shows, there is a sign
difference $(-1)^{n-np}$ between the two definitions, and only the definition using Poincar%
{\'e} duality directly yields an associative product. If one uses the Thom
collapse map, one needs to correct the resulting product by a sign to make
it associative.

\begin{prop}\label{prop:intalg} For $[a]\in H_p(M)$,  $[b]\in H_q(M)$ and $[c]\in H_r(M)$, we have 
\begin{enumerate}
\item  $[a]\bullet_{\Poin}[b]=(-1)^{n-np}[a]\bullet_{\Thom} [b] \in H_{p+q-n}(M)$.
\item  The product $\bullet_{\Poin}$ is  unital with unit $[M]$, associative, and satisfies the graded commutativity relation $[a]\bullet_{\Poin}[b]=(-1)^{(n-p)(n-q)}[b]\bullet_{\Poin} [a]$.   
\item The product $\bullet_{\Thom}$ is  associative up to sign
$$([a]\bullet_{\Thom} [b])\bullet_{\Thom}[c]=(-1)^{n-np}[a]\bullet_{\Thom}([b]\bullet_{\Thom} [c])$$
  and satisfies the graded commutativity relation $[a]\bullet_{\Thom}[b]=(-1)^{n-pq}[b]\bullet_{\Thom} [a]$.
  \end{enumerate}
\end{prop}

\begin{proof}
To prove (1), we consider the following diagram: 
\begin{equation*}
\xymatrix{H_{p}(M)\ot H_{q}(M) \ar[r]^-\x & H_{p+q}(M^2)
\ar[r]^-{[\tau_M\cap]} & H_{p+q-n}(U_\eps) & \ar[l]_{\De_*}^\cong H_{p+q-n}(M) \\
H^{n-p}(M)\ot H^{n-q}(M) \ar[u]^{\cap [M]\ot \cap [M]}\ar[r]^-\x &
H^{2n-p-q}(M^2)\ar[u]^{\cap [M^2]} \ar[rr]^{\Delta^*} && H^{2n-p-q}(M)
\ar[u]_{\cap [M]}.}
\end{equation*}
The top line computes the product $\bullet_{\Thom}$ while going around the diagram the other way computes the product $\bullet_{\Poin}$.
Now the left square commutes up to a sign $%
(-1)^{n(n-q)}=(-1)^{n-nq}$ as 
\begin{equation*}
([a]\cap[M])\times([b]\cap [M])=(-1)^{n\deg(b)}([a]\times [b])\cap [M\times M]
\end{equation*}
by (\ref{equ:capcross}). 
For the right square, the defining property of the  Thom class $\tau_M$ is that $\tau_M\cap [M^2]=\Delta_*[M]$. 
Hence
\begin{align*}
\Delta_*(\Delta^*[c]\cap [M])=[c]\cap \Delta_*[M]&=[c]\cap (\tau_M\cap
[M^2])=([c]\cup \tau_M)\cap [M^2] \\
&=(-1)^{nk}(\tau_M\cup [c])\cap [M^2]=(-1)^{nk}\tau_M\cap ([c]\cap [M^2])
\end{align*}
for any $[c]\in H^k(M^2)$, where the first equality uses the naturality of the cap product (\ref{equ:cap}). This shows that the right square commutes up to the sign $(-1)^{n(2n-p-q)}=(-1)^{np+nq}$.
So the whole diagram commutes up to a sign $(-1)^{n-np}$, giving the result. (See also \cite[Prop 4]{CJY}.)

\medskip

The properties of $\bullet_{\Poin}$ given in (2) follow directly from the corresponding properties of the cup product:
\begin{align*}
  [a]\bullet_{\Poin}[b]&=D(D^{-1}[a]\cup D^{-1}[b])
                 =(-1)^{(n-p)(n-q)}D(D^{-1}[b]\cup D^{-1}[a])=(-1)^{(n-p)(n-q)}[b]\bullet_{\Poin}[a]
                  \end{align*}
                  and likewise for associativity:
\begin{align*}
  ( [a]\bullet_{\Poin}[b])\bullet_{\Poin} [c] =D\Big(D^{-1}\big(D(D^{-1}[a]\cup D^{-1}[b])\big)\cup D^{-1}[c]\Big)
                                  &=D\Big(D^{-1}[a]\cup D^{-1}[b]\cup D^{-1}[c]\Big)\\
                                  &=\cdots =[a]\bullet_{\Poin}([b]\bullet_{\Poin} [c])
\end{align*}
and the unit:
$$ [M]\bullet_{\Poin}[a]=D(D^{-1}[M]\cup D^{-1}[a])=D(1\cup D^{-1}[a])=[a]=\dots =[a]\bullet_{\Poin}[M].$$
and likewise for multiplying with $[M]$ on the right. 

To check the corresponding properties for the Thom product $\bullet_{\Thom}$, we can combine the computations in (1) and (2), which gives:
\begin{align*}
  [a]\bullet_{\Thom}[b]=(-1)^{n-np}[a]\bullet_{\Poin}[b]=(-1)^{n-np+(n-p)(n-q)}[b]\bullet_{\Poin}[a]&=(-1)^{n-np+(n-p)(n-q)+n-qn}[b]\bullet_{\Thom}[a] \\
  &=(-1)^{n-pq}[b]\bullet_{\Thom}[a] 
  \end{align*}
and
\begin{multline*}
  ([a]\bullet_{\Thom} [b])\bullet_{\Thom}[c]=(-1)^{n-np+n-n(p+q-n)}([a]\bullet_{\Poin} [b])\bullet_{\Poin}[c] =(-1)^{nq-n}[a]\bullet_{\Poin}([b]\bullet_{\Poin} [c])\\
  =(-1)^{nq-n+n-nq+n-np}[a]\bullet_{\Thom}([b]\bullet_{\Thom} [c])=(-1)^{n-np}[a]\bullet_{\Thom}([b]\bullet_{\Thom} [c]).
  \end{multline*}

For the intrigued reader that wonders how sign associativity could ever arise in such a naturally defined product, or for the reader who wants to see some robustness in our sign computations, we also give a direct proof of that last sign, from the definition of $\bullet_{\Thom}$. Essentially the same diagram can also be used to compute the sign for the associativity of the Chas-Sullivan product $\wedge_{\Thom}$. Consider the diagram:
$$\xymatrix{H_p(M)\ot H_q(M)\ot H_r(M) \ar[r]^-\x\ar[d]_\x & H_{p+q}(M^2)\ot H_r(M) \ar[r]^-{[\tau\cap]\ot 1} \ar[d]_\x\ar@{}[dr]|{(1)}  & H_{p+q-n}(U)\ot H_r(M) \ar[d]^\x \\
  H_p(M)\ot H_{q+r}(M^2)\ar@{}[dr]|{(2)} \ar[r]^-\x\ar[d]_{1\ot [\tau\cap]} & H_{p+q+r}(M^3) \ar[r]_-{[(\tau\x 1)\cap]} \ar[d]_{[(1\x\tau)\cap]} \ar@{}[ddr]|{(3)}   & H_{p+q+r-n}(U\x M) \ar[d]^{r\x 1} \\
  H_p(M)\ot H_{q+r-n}(U) \ar[r]^-\x\ar[d]_{1\ot r} & H_{p+q+r-n}(M\x U)  \ar[d]_{1\x r} & H_{p+q+r-n}(M^2)\ar[d]^{[\tau\cap]} \\
   H_p(M)\ot H_{q+r-n}(M) \ar[r]^-\x & H_{p+q+r-n}(M^2) \ar[r]^-{[\tau\cap]}  & H_{p+q-2n}(U) 
}$$
Diagram (1) commutes while diagram (2) commutes up to the sign $(-1)^{np}$ by (\ref{equ:capcross}).
The map $r\colon U\to M$ takes $(x,y)$ to $x$, but  is homotopic (using the unique geodesic between $x$ and $y$) to the map $r'$ taking $(x,y)$ to $y$ instead.  Hence
$ (1\x r)^*\tau=\tau\x 1$ and $ (r\x 1)^*\tau= (r'\x 1)^*\tau=1\x \tau$. 
It follows that diagram (3) commutes up to the sign $(-1)^n$ because one composition caps with the class
$ (1\x r)^*\tau\cup(1\times \tau) =(\tau\x 1)\cup (1\times \tau)$ while the other caps with class $(r\x 1)^*\tau\cup  (\tau\times 1)=(1\x \tau)\cup (\tau\times 1)$, and the result follows from the sign commutativity of the cup product.  
\end{proof}

\bibliographystyle{plain}
\bibliography{biblio}

\end{document}